\documentclass{article}
\usepackage[english]{babel}
\usepackage[letterpaper,top=2cm,bottom=2cm,left=3cm,right=3cm,marginparwidth=1.75cm]{geometry}

\usepackage{comment}
\usepackage{amsmath}
\usepackage{amssymb}
\usepackage{xcolor}
\usepackage[colorlinks=true, allcolors=blue, urlcolor = black]{hyperref}
\usepackage[nameinlink,capitalise]{cleveref}
\usepackage{tikz}
\usetikzlibrary{graphs}
\usepackage{amsthm}
\usepackage{mathtools}
\usepackage{thmtools} 
\usepackage{nicematrix}
\usepackage{float} 
\usepackage{enumitem} 
\usepackage[normalem]{ulem}
\usepackage{subcaption}

\let\vec\undefined
\DeclareMathOperator{\vec}{vec}
\newcommand{\indep}{\perp\!\!\!\perp}
\newcommand{\Mcal}{\mathcal{M}}
\newcommand{\J}{\mathcal{J}}

\newcommand{\R}{\mathbb{R}}

\newcommand{\pa}{\mathrm{pa}}
\newcommand{\omegatwo}{\omega^{(2)}}
\newcommand{\omegathree}{\omega^{(3)}}

\newtheorem{theorem}{Theorem}[section]
\newtheorem{definition}[theorem]{Definition}
\newtheorem{proposition}[theorem]{Proposition}
\newtheorem{lemma}[theorem]{Lemma}
\newtheorem{corollary}[theorem]{Corollary}

\theoremstyle{definition}
\newtheorem{remark}[theorem]{Remark}

\theoremstyle{definition}
\newtheorem{example}[theorem]{Example}

\title{Identifiability in Graphical Discrete Lyapunov Models}

\author{
Cecilie Olesen Recke$^*$ \\
{\em University of Copenhagen} \\
\hphantom{xxxxxxxx}\texttt{cor@math.ku.dk}\hphantom{xxxxxxxx}\\ 
\and
Sarah Lumpp$^*$ \\
{\em Technical University of Munich} \\ 
\texttt{sarah.lumpp@tum.de}\\
\and
Nataliia Kushnerchuk$^*$ \\
{\em Aalto University} \\
\texttt{nataliia.kushnerchuk@aalto.fi}\\
\and
Janike Oldekop \\
{\em Technical University of Berlin} \\
\texttt{oldekop@math.tu-berlin.de}\\
\and
Jiayi Li \\
{\em Max Planck Institute for Molecular}\\ {\em Cell Biology and Genetics}\\
\texttt{jli@mpi-cbg.de} \\
\and
Jane Ivy Coons$^{\diamond}$ \\
{\em Worcester Polytechnic Institute}\\
\texttt{jcoons@wpi.edu} \\
\and
Elina Robeva$^{\diamond}$ \\
{\em University of British Columbia} \\
\texttt{erobeva@math.ubc.ca}
}

\date{}
\begin{document}

\maketitle
\def\thefootnote{*}\footnotetext{Indicates co-first authors.}
\def\thefootnote{$\diamond$}\footnotetext{Indicates corresponding authors.}
\begin{abstract}

In this paper, we study discrete Lyapunov models, which consist of steady-state distributions of first-order vector autoregressive models.
The parameter matrix of such a model encodes a directed graph whose vertices correspond to the components of the random vector. 
This combinatorial framework naturally allows for cycles in the graph structure. We focus on the fundamental problem of identifying the entries of the parameter matrix. In contrast to the classical setting, we assume non-Gaussian error terms, which allows us to use the higher-order cumulants of the model. In this setup, we show generic identifiability for  directed acyclic graphs with self-loops at each vertex and show how to express the parameters as a rational function of the cumulants. 
Furthermore, we establish local identifiability for {\em all}  directed graphs  containing self-loops at each vertex and no isolated vertices. Finally, we provide first results on the defining equations of the models, showing model equivalence for certain graphs and paving the way towards structure learning.
\bigskip
\end{abstract}

\section{Introduction}

Vector autoregressive models are widely used in control theory, the health sciences \cite{van2017temporal}, and econometrics \cite{sims1980macroeconomics}. Furthermore, they are applied to model gene regulatory networks in computational biology  \cite{Michailidis2013introduction,Rajapakse2011introduction}. In particular, the first-order vector autoregressive model, or VAR(1) model, is a common modeling choice. Its dynamics are given by
\begin{equation}
\label{eq::VAR(1)process}
    \begin{split}
    X_1 &=  \varepsilon_1, \\
    X_t &= A X_{t-1} + \varepsilon_t \quad \text{for } t > 1,
    \end{split}
\end{equation}
where $X_t$ is a $p$-dimensional random vector, $A \in \mathbb{R}^{p \times p}$ is a parameter matrix encoding the interactions among the components of $X_t$, and $\varepsilon_t$ is a $p$-dimensional random vector modeling independent errors with mean zero and diagonal covariance matrix $D$. 
Under certain conditions on the eigenvalues of $A$, the system \eqref{eq::VAR(1)process} has a unique {\em stationary solution}. Our primary object of interest is the steady-state distribution of the stationary solution. This is also called the equilibrium, limiting or simply stationary distribution in the literature. 
The covariance matrix $\Sigma$ of this steady-state distribution is parametrized by the \emph{discrete Lyapunov equation}
\begin{equation}
    \label{eq:DiscreteLyapunovEqSecondOrder}
    A \Sigma A^T + D = \Sigma.
\end{equation}
If the errors are Gaussian, then the steady-state distribution is also Gaussian and equation~\eqref{eq:DiscreteLyapunovEqSecondOrder} completely characterizes this distribution by its covariance.
In the case of non-Gaussian errors, which we study here, the steady-state distribution can further be characterized by its {\em higher-order cumulants} which satisfy higher-order analogues of equation~\eqref{eq:DiscreteLyapunovEqSecondOrder}. These equations are described in \cref{sec:higher-order}.

The VAR(1) model allows a natural graphical representation. Let $G = (V, E)$ be a directed graph whose vertices $V$ correspond to the $p$ coordinates of the random process $X_t$, and the edges $E \subseteq V \times V$ represent the sparsity pattern of the parameter matrix $A$. In other words, $a_{ji} = 0$ if $i \to j \notin E$. 
Note that this definition allows for self-loops at each node. To simplify indexing later on, we usually number the vertices as $V=\{0\}\cup [p-1]$.
Parallel to the continuous setting~\cite{Recke_Hansen},
we refer to the set of centered distributions satisfying the discrete Lyapunov equations up to order $n$ as the \emph{$n$th order discrete Lyapunov model of the directed graph $G$} (see \cref{def::model}).

Unlike structural causal models in which the random vector $X$ does not evolve with time~\cite{peters2017}, VAR(1) models naturally allow for \emph{cycles}, reflecting feedback loops in many real-world processes (e.g.,  stability of control systems \cite{passino2002lyapunov}, gene regulatory networks \cite{karlebach2008modelling, young2019identifying}, and macro-economic dynamics \cite{kozlov2009investigation}). In particular, it is common to allow the existence of self-loops where some of the variables depend on their own past, capturing the temporal and cyclic nature of the model. An interpretation of cycles in linear structural equation models is discussed in~\cite{JMLR:v13:hyttinen12a}, where it corresponds to the error in \eqref{eq::VAR(1)process} not depending on time. The authors of~\cite{JMLR:v13:hyttinen12a} discuss the limits of this deterministic view point and propose the set-up in the current paper as a possible extension. 

Since our primary interest lies in the steady-state distribution of the time series, each of our samples corresponds to a single snapshot of cross-sectional data once the process has stabilized. This perspective is critical in fields such as computational biology, where measuring the same system at multiple time points may be infeasible due to the destructive nature of RNA-sequencing technologies~\cite{young2019identifying}. Even though only one time point is observed per sample, the underlying dynamic structure is preserved by virtue of the model's stationarity. Furthermore, by keeping the interpretation of the data as arising from the steady state, there is a natural interpretation of cycles. 

Variations of this approach were recently established in \cite{young2019identifying} and \cite{varando2020graphical}. While \cite{young2019identifying} focuses on the Gaussian stationary VAR(1) model, \cite{varando2020graphical} proposes the equivalent in a continuous setting -- the graphical continuous Lyapunov model. This model parametrizes stationary distributions of diffusion processes, such as the Ornstein-Uhlenbeck process, via the continuous Lyapunov equation. 
Recent results include advances in parameter identifiability in the Gaussian and the non-Gaussian cases \cite{dettling2023identifiability,Recke_Hansen}, 
as well as new findings on structure identifiability and model equivalence~\cite{amendola2025structural}. Further work has addressed statistical estimation and structure learning \cite{varando2020graphical,dettling2024lasso}. 
In more general diffusion models with potentially nonlinear drift, recent contributions characterize the model's conditional independence structure \cite{boege2024conditional} and propose intervention-based learning procedures~\cite{lorch2024causal}.

While continuous Lyapunov models continue to be an active area of research, analogous results for discrete-time models are limited. 
In the case of continuous Lyapunov models with unknown errors, the strongest possible identifiability result can only ever be up to a scaling of the parameter matrix \cite{dettling2023identifiability}. However, as we will show, this is not the case for discrete Lyapunov models. 
To reliably {\em estimate the graph}~$G$ from data in the discrete setting, it is crucial to first provide theoretical guarantees for {\em identifiability} of the model parameters if the graph is known. This is precisely the question we address in the present work. To do so, we use tools from algebraic geometry, graph theory, and algebraic statistics. 

In the case of Gaussian error terms, the resulting stationary distribution is Gaussian as well, implying that all cumulants of order higher than two are zero. For general graphs, however, the  parameters are not identifiable
from only the first and second-order moments \cite{comon2010handbook} -- for instance, flipping the sign of the parameter matrix $A$ already results in the same covariance matrix.
Prior work by \cite{young2019identifying} mitigates these issues by assuming that the error covariance $D$ is known as well as imposing a particular sparsity pattern on $A$ and assuming the signs of all the diagonal elements of $A$ are known. In this specific setting, they are able to show identifiability of the parameters in $A$ from the covariance matrix as well as consistency and asymptotic efficiency of the maximum likelihood estimator of $A$. 
Recent work by \cite{liu2025identifiability} presents first results on structure identifiability from the covariance matrix. In particular, they consider model distinguishability based on the Jacobian matroid of the parametrization map. They provide sufficient conditions for generic structure distinguishability based on model dimension and graphical structures in special cases. 

In this work, in contrast, we consider the non-Gaussian setting which allows us to leverage higher-order cumulants as well. We show in \cref{thm:LocIdAllGraphs} and following corollaries that the entries of the parameter matrix $A$ together with the cumulants of the error terms are locally identifiable through an algebraic characterization of the second-, third-, and fourth-order cumulants for all graphs with all self-loops without an isolated node.
We also show in \cref{prop:IdAllSelfLoops} that when the sparsity pattern of $A$ specifies a DAG where all self-loops are present, the model parameters are generically (rationally) identifiable from second-, third- and fourth-order cumulants.
Furthermore, we investigate model equivalence and the polynomial equations satisfied by the cumulants for certain families of graphs.

The remainder of the paper is organized as follows. In \cref{sec::Preliminaries}, we derive the discrete Lyapunov equations for higher-order cumulants and use them to define higher-order discrete Lyapunov models. In \cref{sec::TrekParametrization}, we show how to parametrize the model using a trek rule and how this trek rule can be restricted to base treks without self-loops in a specific setting. \cref{sec::ParameterIdentifiabilityDAGs} addresses the question of parameter identifiability when the underlying graph is a DAG (with self-loops). In this case, we prove that the matrix $A$ and the cumulants of the error are generically identifiable from the second-, third-, and fourth-order cumulants of the stationary distribution of $X_t$. \cref{sec:local identifiability} extends these results by proving local identifiability for arbitrary directed graphs that may contain cycles while considering the Jacobian of only the second- and third-order cumulants. In \cref{sec::Equations}, we provide initial results on model equivalence as well as equations in the ideal of the model.  

\section{Preliminaries} \label{sec::Preliminaries}
In this section, we provide preliminary results and definitions for discrete Lyapunov models. In \cref{sec:higher-order}, we first derive the tensor equations for the cumulants of the steady state of a VAR(1) process; these equations are referred to as the \emph{discrete Lyapunov equations}. In \cref{sec:model-definition}, we define the $n$th-order discrete Lyapunov model for a given directed graph $G$. Finally, in \cref{sec:Identifiability-definition}, we introduce different types of identifiability and provide initial remarks on the types of identifiability that can be feasibly established for discrete Lyapunov models. 

\subsection{Higher-Order Cumulants and the Discrete Lyapunov Equations}\label{sec:higher-order}
A steady-state distribution of a VAR(1) process as described in \eqref{eq::VAR(1)process} exists in the general setting if $A$ is a Schur stable matrix, which is a matrix where all eigenvalues have absolute values strictly less than~1, and that the $\varepsilon_t$ are independent with mean 0 and have the same covariance. Under these condition, Example~8.4.1 of Brockwell and Davis \cite{brockwell2016introduction} (combined with Example~2.2.1) shows that there is a unique stationary solution to \eqref{eq::VAR(1)process}, given by
\begin{equation}
\label{eq::stationary-solution}
    X_t \;=\; \sum_{j=0}^{\infty} A^j \, \varepsilon_{t-j}.
\end{equation}
Intuitively, requiring $|\lambda_i(A)| < 1$ ensures that the powers $A^j$ decay as $j \to \infty$, making the above series converge absolutely. Hence, for large $t$, the distribution of $X_t$ stabilizes to a unique steady-state distribution, also sometimes called its limiting, equilibrium or stationary distribution.

We aim to characterize steady-state distributions of VAR(1) processes by their higher-order cumulants. Therefore, we will assume that the errors, $\varepsilon_t$, are independent and identically distributed so that all their higher-order cumulants are identical. 
We derive equations that describe the cumulants of the steady state based on the parameter matrix $A$ and the corresponding cumulant tensors of the errors $\varepsilon_t$. Denote by $\Omega^{(n)} = \text{cum}_n(\varepsilon_t)$ the $n$th-order cumulant tensor of the noise. Note that $\Omega^{(2)} = D$ as introduced in \eqref{eq::VAR(1)process}. 
As shown in \cite{young2019identifying}, the second-order cumulant of a stationary process $X_t$, i.e., $\mathrm{Var}(X_t)$, is described by the discrete Lyapunov equation in \eqref{eq:DiscreteLyapunovEqSecondOrder} or see remark after \cref{prop:nthorder}. We show that the higher-order cumulants of the steady-state distribution of $X_t$ follow higher-order tensor equivalents of the discrete Lyapunov equation.

In the following, we will denote the second-, third-, and fourth-order cumulants by $S = T_2$, $T = T_3$, and $R = T_4$, respectively. A notion of taking the product between a tensor and a matrix is needed to write the higher-order discrete Lyapunov equations, this is called the $k$-mode product. The definition of the $k$-mode product of a tensor $T \in \mathbb{R}^{I_1 \times I_2 \times \cdots \times I_N}$ with a matrix $A \in \mathbb{R}^{J \times I_k}$, denoted $T \times_k A$, 
is a $I_1 \times \cdots \times I_{k-1} \times J \times I_{k+1} \times \cdots \times I_N$ tensor with element-wise entries 
\begin{equation*}
    (T \times_n A)_{i_1 \cdots i_{k-1} j i_{k+1} \cdots i_N} = \sum_{i_k = 1}^{I_k} T_{i_1 \cdots i_{k-1} i_k i_{k+1} \cdots i_{N}} A_{ji_k}.
\end{equation*}

When performing the multiplication along every mode at once (as in equation \eqref{eq::CumulantInfiniteSum}), this is also called the Tucker product. 

\begin{proposition}[$n$th-Order Cumulants]
\label{prop:nthorder}
Let $A \in \mathbb{R}^{d \times d}$ be a Schur stable matrix, meaning it has all eigenvalues $|\lambda_i(A)| < 1$. Suppose $\varepsilon_t$ has finite $n$th order cumulant $\Omega^{(n)}$. Then the steady-state distribution of the VAR(1) model \eqref{eq::VAR(1)process} has finite $n$th cumulant $T_n$, and
\begin{align}
    \label{eq::CumulantInfiniteSum}
	T_n = \sum_{i = 0}^{\infty} \Omega^{(n)} \times_1 A^{i} \times_2 \cdots \times_n A^{i},
\end{align}
where $\times_k$ denotes the $k$-mode product. Two other equivalent equations for $T_n$ are the recursive formula
\begin{align}
\label{eq::DiscreteLyapunovRecursive}
 T_n \;=\; T_n \;\times_1 A \;\times_2 A \;\cdots \;\times_n A 
             \;+\; \Omega^{(n)},
\end{align}
and the vectorized product
\begin{align}
\label{eq::discreteLyapunovVec}
\mathrm{vec}(T_n) = (I - \underbrace{A\otimes\cdots\otimes A}_{n \text{ times}})^{-1}\mathrm{vec}(\Omega^{(n)}).
\end{align}
\end{proposition}

\begin{proof}
From \eqref{eq::VAR(1)process}, we can apply the properties of how cumulants behave under linear and affine transformations (see, e.g., \cite[Section 2.4]{mccullagh2018tensor})  and repeatedly use the linear relation $X_{t+1} = A X_t + \varepsilon_{t+1}$. For $n$th-order cumulants, one obtains the recursive identity
\begin{equation}
\label{eq::RecursionofCumulants}
\text{cum}_n(X_{t+1})
   \;=\; \bigl(\text{cum}_n(X_t)\bigr) \;\times_1 A \;\times_2 A \;\cdots \;\times_n A
         \;+\; \Omega^{(n)}.
\end{equation}

Under stationarity (and by \eqref{eq::stationary-solution}), $X_t$ converges in distribution to the limit $X_\infty$ whose $n$th order cumulant $T_n$ satisfies
\[
   T_n \;=\; \lim_{t\to\infty} \text{cum}_n(X_t).
\]
Unraveling the recursion yields the infinite series \eqref{eq::CumulantInfiniteSum}. Convergence follows because $A$ is Schur stable; so $A^i \to 0$ as $i \to \infty$ sufficiently fast to guarantee that the infinite sum is well-defined.

The vectorized equation \eqref{eq::discreteLyapunovVec} arises by vectorizing \eqref{eq::CumulantInfiniteSum} and using properties of the $k$-mode product and Kronecker product. 

Equation \eqref{eq::DiscreteLyapunovRecursive} also follows from equation \eqref{eq::RecursionofCumulants} since at stationarity the cumulant cannot change from time step to time step, so the cumulant of the steady-state distribution has to satisfy \eqref{eq::DiscreteLyapunovRecursive}. Alternatively, \cite[Corollary 3.2]{xu2022solutions} yields that, when $A$ is Schur stable, \eqref{eq::CumulantInfiniteSum} is the unique solution to the recursive equation \eqref{eq::DiscreteLyapunovRecursive} implying that the two equations are equivalent.
\end{proof}

The recursive equation \eqref{eq::DiscreteLyapunovRecursive} is the generalization of what has classically been known as the (second-order) discrete Lyapunov equation, \eqref{eq:DiscreteLyapunovEqSecondOrder}. Since they are all equivalent, we will refer to all the equations of the form \eqref{eq::CumulantInfiniteSum}, \eqref{eq::DiscreteLyapunovRecursive} and \eqref{eq::discreteLyapunovVec} as \emph{discrete Lyapunov equations}, to differentiate between them the order
of the cumulant should be specified. 

\begin{remark}[Gaussian special case] \label{rem::Gaussian_special_case}
If in addition $\varepsilon_t$ is Gaussian, one recovers the standard result: in the limit, $X_t$ follows a multivariate Gaussian distribution with mean $\mathbf{0}$ and covariance matrix $\Sigma$ satisfying the \emph{discrete Lyapunov equation}
\[
    \Sigma \;=\; A \,\Sigma\, A^T \;+\; D,
\]
where $D$ is the noise covariance. Moreover,
\[
    \Sigma \;=\; \sum_{j=0}^{\infty} A^j \, D\, (A^T)^j \;=\; (I - A \otimes A)^{-1}\vec(D),
\]
highlighting that the requirement $|\lambda_i(A)| < 1$ ensures invertibility of $(I - A \otimes A)$.
\end{remark}

\subsection{Defining the Model via Cumulant Equations} \label{sec:model-definition}

A directed graph $G$ is a pair $(V, E)$, where $V$ is the set of vertices, which we usually denote as $V = \{0\}\cup [p-1]$, and $E\subseteq V\times V$ is the set of edges, which are ordered pairs $(i,j)$, often denoted as $i\to j$. We denote by $\pa(i)$ the set of parents of a node $i\in V$, i.e., $\pa(i) = \{j\in V : j\to i\in E\}$. 

 Given a directed graph $G = (V,E)$, we define the (graphical) \emph{discrete Lyapunov model} of order $n$ as the set of all cumulant tensors satisfying the appropriate discrete Lyapunov equations, where $A$ is a stable matrix with sparsity pattern according to the edges of $G$, that is $a_{ji} = 0$ if $i \rightarrow j \not \in E$. We denote this set by $\mathbb{R}^{E}$. Furthermore, as explained in \cref{sec:higher-order}, the matrix $A$ needs to be Schur stable in order for the steady-state distribution to exist, so we will further assume that $A \in \mathbb{R}^{E}_{\text{stab}}$, the Schur stable matrices which have sparsity pattern according to $G$. In the following let $\text{Sym}_m(\mathbb{R}^l)$ denote the set of symmetric tensor of order $m$ on $\mathbb{R}^l$. 

 \begin{definition}
     \label{def::model}
     Let $G = (V, E)$ be a directed graph. The \emph{$n$th-order discrete Lyapunov cumulant model} of $G$ is the set
    \begin{align*}
    \mathcal{M}_G^{\leq n} = \{(T_2,\ldots, T_n)~:~&T_m = \sum_{i = 0}^{\infty} \Omega^{(m)} \times_1 A^{i} \times_2 \cdots \times_m A^{i} \text{ for all }2\leq m\leq n,\\
    & A \in \mathbb{R}^E_{\mathrm{stab}}, \ \Omega^{(m)} \in \mathrm{Sym}_m(\mathbb{R}^{|V|}) \text{ diagonal for all }2\leq m\leq n\}.  
    \end{align*}

    The \emph{$n$th-order cumulant ideal} of $G$ is the ideal $\mathcal{I}^{\leq n}
(G)$ of polynomials in the entries
$T_2,\ldots, T_n$ that vanish when $(T_2, \ldots, T_n) \in \mathcal{M}^{\leq n}_G$. 
 \end{definition}

Notice that $\mathcal{I}^{\leq n}
(G)$ is a homogeneous ideal. Usually, we will focus on the case where $n = 3$ or $4$, corresponding to Lyapunov models of order at most four.

\begin{example}\label{ex::VanishingIdealTwoNodegraph}
Consider the third-order discrete Lyapunov cumulant model of the graph on two nodes in \cref{fig::TwoNodes1self-loop}. 
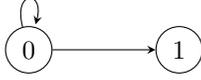
\begin{figure}  
\centering
\begin{tikzpicture}
\node[circle, draw, minimum size=0.5cm] (0) at (0,0) {0};
\node[circle, draw, minimum size=0.5cm] (1) at (2,0) {1};
  
\draw[->,  >=stealth] (0) edge[loop above] (0);
\draw[->,  >=stealth] (0) edge (1);
\end{tikzpicture} 
\caption{A directed graph on two nodes with a single self-loop at the source.}
\label{fig::TwoNodes1self-loop}
\end{figure}
The graph gives rise to the following $A$ matrix, and the error cumulants are assumed to be diagonal: 
\[
A = \begin{pmatrix} a_{00} & 0 \\ a_{10} & 0 \end{pmatrix}, \qquad
\Omega^{(2)} = \begin{pmatrix} \omega_{00} & 0 \\ 0 & \omega_{11} \end{pmatrix}, \qquad
\Omega^{(3)} = \operatorname{diag}(\omega_{000}, \omega_{111}).
\]
The second- and third-order cumulants are given by
\[
S = A S A^\top + \Omega^{(2)}, \qquad
T = T \times_1 A \times_2 A \times_3 A + \Omega^{(3)}.
\]
Writing these in entries yields
\begin{align*}
s_{00} &= a_{00}^2 s_{00} + \omega_{00}, & s_{01} &= a_{00} a_{10} s_{00}, & s_{11} &= a_{10}^2 s_{00} + \omega_{11},\\
t_{000} &= a_{00}^3 t_{000} + \omega_{000}, & t_{001} &= a_{00}^2 a_{10} t_{000}, & t_{011} &= a_{00} a_{10}^2 t_{000}, & t_{111} &= a_{10}^3 t_{000} + \omega_{111}.
\end{align*}
Eliminating the parameters $a_{00}, a_{10}, \omega_{000}$ leads to the polynomial relation, $
s_{01}^3 t_{000}^2 - s_{00}^3 t_{001} t_{011} = 0,$
which generates the third-order cumulant ideal
\[
\mathcal{I}^{\le 3}(G) = \left\langle s_{01}^3 t_{000}^2 - s_{00}^3 t_{001} t_{011} \right\rangle.
\]
\end{example}

\subsection{Identifiability} \label{sec:Identifiability-definition}
In this paper we study identifiability of the parameters in the discrete Lyapunov model for a known graph $G$. Identifiability is about determining injectivity of the map from the set of parameters of interest, $A$, and the noise parameters, $\Omega^{(n)}$, to the model $\mathcal M^{\leq n}_G$. In this section we introduce different notions of identifiability. In \cref{sec::ParameterIdentifiabilityDAGs} we prove generic identifiability of all DAGs (with self-loops) from second-, third- and some fourth-order cumulants, and in \cref{sec:local identifiability} we prove local identifiability for most directed graphs from only the second- and third-order cumulants.

If we let the error cumulants $\Omega^{(n)}$ be represented by their non-zero diagonal entries, then the parameters $ (A, \Omega^{(2)}, \Omega^{(3)}, \Omega^{(4)})$ of the fourth-order discrete Lyapunov model lie in the following set: 
\begin{equation*}
    \Theta_{G, \leq 4} =  \mathbb{R}^{E}_{\text{stab}} \times \mathbb{R}_{+}^p \times (\mathbb{R} \setminus \{ 0 \})^p \times \mathbb{R}_{+} ^p.
\end{equation*}

The question of identifiability for the fourth-order discrete Lyapunov models is then whether the (rational) map $$\varphi_{G, \leq4}: \Theta_{G, \leq4} \rightarrow PD_p \times \text{Sym}^3(\mathbb{R}^p) \times \text{Sym}^{4}(\mathbb{R}^p), \quad (A, \Omega^{(2)}, \Omega^{(3)}, \Omega^{(4)}) \mapsto (S, T, R)$$
is injective,
where  $(S, T, R)$ are the solutions to the corresponding second-, third- and fourth-order discrete Lyapunov equations. The identifiability results will initially be stated as identifying $A$ since it is clear by considering \eqref{eq::DiscreteLyapunovRecursive} that if $A$ can be identified from the cumulants $(S, T, R)$ then we can also identify the error cumulants. 

We provide definitions of several different types of identifiability for a parametrized statistical model, see for example \cite[Chapter 16]{sullivant2018algebraic} for an introduction. By a statistical model we mean a set of probability measures or a representation of such a set, e.g. via cumulants up to a certain order. 

\begin{definition}
\label{def::TypesofID}
    Let $f: \Theta \rightarrow \mathcal{M}_f$ be a rational map from a finite dimensional parameter space $\Theta \subseteq \mathbb{R}^k$ to a statistical model such that $\text{im}(f) = \mathcal{M}_f$. The model is said to be
    \begin{enumerate}[label = (\roman*),leftmargin=*, widest=iii]
        \item \hspace{-6pt} globally identifiable if $f^{-1}(f(\theta)) = \theta$ for all $\theta \in \Theta$;
        \item \hspace{-6pt} generically identifiable if \( f^{-1}(f(\theta)) = \theta \) for all $\theta \in \Theta \hspace{-1pt} \setminus \hspace{-1pt} A$ where $A$ is a proper algebraic subset of~$\Theta$;
        \item \hspace{-6pt} locally identifiable if $|f^{-1}(f(\theta))| < \infty$ for all $\theta \in \Theta \hspace{-1pt} \setminus \hspace{-1pt} A$ where $A$ is a proper algebraic subset of $\Theta$; 
        \item \hspace{-6pt} non-identifiable if $|f^{-1}(f(\theta))| = \infty$ for generic $\theta$.
    \end{enumerate}
    Note that for the second and third notion of identifiability the Zariski closure $\overline{\Theta}$ of $\Theta$ is assumed to be an irreducible variety.
\end{definition}
In some literature, generic (and local) identifiability is defined as identifiability away from a null set. \cref{def::TypesofID} (ii) is more specific than this since an algebraic set is also always a null set. Requiring the exceptional set to be algebraic has several justifications; general null sets can behave pathologically (for example, they may be dense in the ambient space) in ways algebraic sets cannot and in the following theorems it is the stronger version (\cref{def::TypesofID} (ii)) that we are able to prove in all cases. Local identifiability is also known as finite-to-one identifiability. In the case where $\Mcal_f$ is generically identifiable and $f^{-1}(p)$ is a rational function of the entries of $p$ for almost all $p \in \Mcal_f$, we say that $\Mcal_f$ is generically \emph{rationally} identifiable.

We seek to answer the question of identifiability for the discrete Lyapunov models (primarily up to fourth-order) so we will apply the definitions to the map $\varphi_{G, \leq 4}$. The parameter set $\Theta_{G, \leq 4}$ is an open subset (in the standard Euclidean topology) so its Zariski closure is therefore irreducible and \cref{def::TypesofID} (ii) and (iii) can be applied. If the discrete Lyapunov model of a certain order corresponding to a graph $G$ is identifiable it is sometimes denote this as the graph $G$ being identifiable.

Note that the discrete Lyapunov model with unknown errors can never be globally identifiable no matter the order of the cumulants. This is because the subset of the parameter space corresponding to isolating a node in the graph (or isolating all of them by taking $A$ to be diagonal) yields an under determined system (see \cref{rmk:isolatednode}). In the case where $A$ is diagonal, the higher order cumulants in the model are diagonal as well (see the treks rules in the following section) with entries only depending on the corresponding diagonal entries of $A$ and the respective error cumulants. This is why the results of \cref{sec::ParameterIdentifiabilityDAGs} and \cref{sec:local identifiability} focus on generic and local identifiability. 

Furthermore, if one were to consider only the second-order discrete Lyapunov model with unknown diagonal errors, it is not possible to prove generic identifiability without imposing sign conditions on several entries of $A$.
As discussed by \cite{young2019identifying}, this can be immediately seen for the second-order discrete Lyapunov model by observing that flipping the signs of a parameter matrix $A \in \mathbb{R}^{E}_{\text{stab}}$ to $-A \in \mathbb{R}^{E}_{\text{stab}}$ yields the same covariance matrix. This is why the results of \cref{sec::ParameterIdentifiabilityDAGs} and \cref{sec:local identifiability} use third- and fourth-order cumulants as well as the covariance.

\section{Trek Parametrization} \label{sec::TrekParametrization}

In this section we consider the individual entries of the cumulants $T_n$ and the expression for each such entry in terms of the parameters $A$ and $\Omega^{(n)}$ induced by the formulas in equation~\eqref{eq::CumulantInfiniteSum}. In \cref{sec:trek_rules} we give a combinatorial way, called the trek rule, to read off these expressions directly from the graph. In \cref{sec:restricted_trek} we show how in special cases one can simplify the expressions for the entries of $T_n$. Finally, in \cref{sec:marginal} we show how missing treks in the graph imply marginal independence between some of the random variables.

\subsection{Trek Rules}\label{sec:trek_rules}
Trek rules are commonly used to express the entries of the moments and cumulants in  linear structural equation models~\cite{sullivant2010trek, robeva2021multitrek, amendola2023third-order} and continuous Lyapunov models~\cite{boege2024conditional, hansen2025trekrulelyapunovequation}. We prove that similar rules hold for the expression of the entries of the cumulants $T_n$ in our discrete Lyapunov models, with the only difference that the paths of the treks need to have the same length. We call such treks {\em equitreks}. They are a specialization of the treks
introduced in \cite{sullivant2010trek, robeva2021multitrek} for linear structural equation models.

\begin{definition}
    Let $G = (V, E)$ be a directed graph. A $k$-equitrek between $k$ vertices $i_1, \dots, i_k$ is an ordered set of directed paths $(P_1, \dots, P_k)$ of the same length $l$ from a common source $v \in V$ to sinks $i_1, \dots, i_k$, respectively. We say that such a $k$-equitrek has length $l$. We let $\mathcal{T}(i_1, \dots, i_k)$ denote the set of all $k$-equitreks between the nodes $i_1, \dots, i_k$. 
\end{definition}

Rewriting the formula for $T_n$ obtained in \cref{prop:nthorder} yields the following trek rule with the proof given in \cref{sec:proofsTrekParametrization}. The cumulant $\Omega^{(n)}$ of the noise vector $\varepsilon_t$ is a diagonal tensor, so we only use its $p$ diagonal entries denoted by $w_1^{(n)},\ldots, w_p^{(n)}$ to simplify the notation.

\begin{proposition}
\label{prop::trekrule}
    Let $G = (V, E)$ be a directed graph on $p$ nodes, where $V = \{0\}\cup [p-1]$, and let the cumulant $T_n$ satisfy the  $n$th order discrete Lyapunov equation~\eqref{eq::CumulantInfiniteSum}. Then $T_n$ satisfies the trek rule
    \begin{equation*}
        (T_n)_{i_1, \dots i_n} 
      = \sum_{\tau = (\tau_1, \dots , \tau_n) \in \mathcal{T}(i_1, \dots, i_n)} w^{(n)}_{top(\tau)} a^{\tau_1} \cdots a^{\tau_n},
    \end{equation*}
    where $a^{P}$ denotes the path monomial associated to the path $P$ given as $a^{P} = \prod_{\alpha \rightarrow \beta \in P} a_{\beta \alpha}$.
\end{proposition}

\begin{remark}
    If the error cumulant $\Omega^{(n)}$ is not assumed to be diagonal, the top of a trek can be considered to be an $n$-fold blunt (or multidirected) edge corresponding to a non-zero entry in $\Omega^{(n)}$. In this way each of the paths could potentially start at different vertices. It is clear that the proof of~\cref{prop::trekrule} generalizes to this case. See, for example, \cite{varando2020graphical, foygel2012HTC} for a trek rule of this nature in the covariance case for linear structural equation models and continuous Lyapunov models, respectively. 
\end{remark}

The entries of the second- and third-order cumulants in a discrete Lyapunov model are then explicitly given by the following equations.

\begin{corollary}
\label{cor::trekrulefor2ndand3rd}
    Let $G = ([p], E)$ be a directed graph. 
    For $(S, T,\dots) \in \mathcal{M}^{\leq n}_G$ with $n\ge3$, we obtain the following trek rules: 
    \begin{equation*}
        s_{ij} = \sum_{\tau = (\tau_1, \tau_2) \in \mathcal{T}(i, j)} w^{(2)}_{top(\tau)} a^{\tau_1} a^{\tau_2}, \qquad t_{ijk} = \sum_{\tau = (\tau_1, \tau_2, \tau_3) \in \mathcal{T}(i, j, k)} w^{(3)}_{top(\tau)} a^{\tau_1} a^{\tau_2} a^{\tau_3}.
    \end{equation*}
\end{corollary}

The following example illustrates the trek rule.

\begin{example}
\label{ex::trekrule}
Let $G$ be the same graph on two nodes as in \cref{ex::VanishingIdealTwoNodegraph}, shown in \cref{fig::TwoNodes1self-loop}. 
Then, by using the formulas from \cref{cor::trekrulefor2ndand3rd}, we get $$s_{00} = \sum_{t=0}^{\infty} w^{(2)}_{0} a_{00}^{2t} = \frac{w^{(2)}_{00}}{1-a_{00}^2}, \qquad s_{01} = s_{10} = \sum_{t=0}^{\infty} w^{(2)}_{0} a_{00}a_{10} a_{00}^{2t} = \frac{w^{(2)}_{0}a_{00}a_{10}}{1-a_{00}^2}, \qquad s_{11}=w^{(2)}_{1}.$$

Analogously,
\begin{equation*}
t_{000} = \frac{w^{(3)}_{0}}{1-a^3_{00}}, \qquad 
t_{001}=\frac{w^{(3)}_0a_{00}^2a_{10}}{1-a^3_{00}}, \qquad t_{011}=\frac{w^{(3)}_0a_{00}a_{10}^2}{1-a^3_{00}}, \qquad t_{111} = w_1^{(3)}.
\end{equation*}
\end{example}

An intuitive explanation of our trek rule is as follows. If we think about the VAR(1) process, in order to calculate the cumulant for two or more variables from the same time point, we would need to trace all possible tuples of paths that start at a common source back in time and end at the variables of interest. Since the variables are at the same time point, the paths will have the same length, which means they form an equitrek. This is also why if there are no cycles (for example self-loops) there are very few equitreks in the graph and many entries of the cumulants will be zero. 

\begin{remark}
    \cite{liu2025identifiability} introduce \emph{maximal classes} as the set of all nodes reachable from a source strongly connected component. 
    The concept of treks offers a complementary viewpoint by interpreting these classes as sets of nodes connected by treks where all top nodes belong to the same source strongly connected component.
    Our trek rule allows us to reinterpret their structural identifiability results in the language of treks, as they relate the notion of maximal classes to entries of $\Sigma$ being non-zero.
\end{remark}

\subsection{A Restricted Trek Rule for DAGs}\label{sec:restricted_trek}
Since the infinite sums in the general trek rule may be hard to compute, we derive a trek rule for DAGs in a restriction of the model where we assume that the entries on the diagonal of the parameter matrix $A$ are constant, i.e., that $a_{ii} = t$ for all $i \in V$. This allows us to obtain a finite sum representation. 
Similar restrictions have been considered by \cite{boege2024conditional} and \cite{hansen2025trekrulelyapunovequation} in the continuous Lyapunov model based on second-order cumulants.

To formulate the trek rule in a concise way, we define the following notation. Let $\mathcal{T}^*(i,j)$ be the set of \emph{base treks} between $i$ and $j$, that is, the set of treks that visit each node on the trek only once on each path of the trek (i.e., no self-loops and cycles). Note that we do not require these treks to be equitreks. For a base trek $\tau=(\tau_i,\tau_j) \in \mathcal{T}^*(i,j)$, let $d(\tau_i)$ be the edge-length of the path $\tau_i$, i.e., the distance measured in number of edges between the top node, $\mathrm{top}(\tau)$, and the leaf $i$. Note that $\tau$ does not contain self-loops, so the path monomial $a^{\tau_i}$ only contains off-diagonal entries of $A$. 

The idea is to collect all terms that account for the edges of a base trek between $i$ and $j$ (a monomial in the off-diagonal entries of $A$) as well as collect the terms that encode the different ways of adding self-loops to this base trek to obtain all possible equitreks between $i$ and $j$ with this underlying base trek (a rational function in $t$). 
This rational function in $t$ is the product of a power of $t$ balancing out the potentially different lengths of both legs of a base trek and of a rational function whose numerator is a polynomial $p_{x,y}(t)$ and whose denominator is a power of $(1-t^2)$ accounting for all combinations in which arbitrary numbers of self-loops can be added along the trek while still obtaining an equitrek. 
We derive the following finite sum representation for the entries of the covariance matrix $S$.
\begin{proposition} \label{prop:trek_rule_dag}
    Let $G$ be a DAG with constant self-loop parameter $t$ and consider the corresponding discrete Lyapunov model. Then the entries of the covariance matrices $S$ in the model are parametrized by the following base trek rule
    \begin{equation} \label{eq:trek_rule_dag}
        s_{ij} = \sum_{\tau=(\tau_i,\tau_j) \in \mathcal{T}^*(i,j)} C(d(\tau_i),d(\tau_j);t) a^{\tau_i}a^{\tau_j} \omega_{\mathrm{top}(\tau)}^{(2)}, 
    \end{equation}
    where for distances $x,y \geq 0$ along the trek, the rational function $C(x,y;t)$ is defined as
    \begin{equation*}
        C(x,y;t) = t^{|x-y|} \frac{p_{x,y}(t)}{(1-t^2)^{x+y+1}},
    \end{equation*}
    and the polynomial $p_{x,y}(t)$  in the numerator  for distances $x, y \geq 0$ is given by
\begin{equation*}
    p_{x,y}(t) 
    = \sum_{l = 0}^{\min(x,y)} t^{2l} \binom{\max(x,y)}{\min(x,y)-l}\binom{\min(x,y)}{l} 
    = \sum_{l = 0}^{\min(x,y)} t^{2(l-\min(x,y))} \binom{x}{l}\binom{y}{l} .
\end{equation*}
\end{proposition}
Note that for example when $x \leq y$, we can simplify
\begin{equation*}
    p_{x,y}(t) = \sum_{l = 0}^{x} t^{2l} \binom{y}{x-l}\binom{x}{l}.
\end{equation*}

The proof can be found in \cref{sec:proofsTrekParametrization}. This result is parallel to the restricted trek rule for DAGs in the continuous Lyapunov model proposed by \cite{boege2024conditional}. However, while in the continuous setting the trek monomial is weighted by a binomial coefficient, here it is weighted by a sum of binomial coefficients together with a rational term accounting for the equitrek parametrization of the model.
The proof of \cref{prop:trek_rule_dag} is based on induction by employing the recursive formula for the entries of $S$ and can be found in \cref{app:Proof_Prop_3.7}.
The restricted trek rule simplifies significantly in the case of a tree or a directed path as demonstrated in the following examples.

\begin{example}
    In the case of a (poly)tree, two nodes $i$ and $j$ have at most one common source in the graph. 
    If $i$ and $j$ have a common source, there is a base trek $\tau = (\tau_i,\tau_j)$ with this source node as its top node. Every base trek between $i$ and $j$ is contained in the trek $\tau$ and has a common ancestor $k$ of $i$ and $j$ as its top node. We denote this base trek by $\tau^k = (\tau_i^k,\tau_j^k)$. Consequently, the set of base treks between $i$ and $j$ can directly be parametrized by the common ancestors of $i$ and $j$.
    Further let $z = \max (\mathrm{An}(i) \cap \mathrm{An}(j))$ be the common ancestor closest to $i$ and $j$ where $\mathrm{An}(i)$ denotes the ancestors of a node $i$ including $i$ itself. Then the trek rule simplifies further to 
    \[
    s_{ij} = t^{|d(\tau_i^z)-d(\tau_j^z)|}\sum_{k \in \mathrm{An}(i) \cap \mathrm{An}(j)}\frac{p_{d(\tau_i^k),d(\tau_j^k)}(t)}{(1-t^2)^{d(\tau_i^k)+d(\tau_j^k)+1}} a^{\tau_i^k}a^{\tau_j^k} w_k^{(2)}.
    \]
\end{example}
    
\begin{example}
    The case of a directed path yields an even simpler formula. Assuming a topological order, the distance between two nodes is given by the difference of their indices.
    The trek rule is then 
    \[
    s_{ij} = t^{|i-j|}\sum_{k = 0}^{\mathrm{min}(i,j)} \frac{p_{i-k,j-k}(t)}{(1-t^2)^{(i-k)+(j-k)+1}} a^{\tau_i^k} a^{\tau_j^k} w_k^{(2)}.
    \]
    This illustrates the role of each term in the formula: the first factor $t^{|i-j|}$ adds the minimum number of missing self-loops to turn a base trek into an equitrek. The index $k$ ranges over every possible top node of a trek between $i$ and $j$. The denominator accounts for the infinite number of additional self-loops that can be added while maintaining equal length on both sides. The numerator $p_{i-k,j-k}(t)$ accounts for all possible ways the self-loops can be positioned along the considered trek with top node $k$.
    For illustration, we compute some of these entries for the directed path on four nodes. The appearing polynomials $p_{x,y}(t)$ can be ordered into a symmetric square matrix for $x,y \in \{0,1,2,3\}$ as follows:
    \begin{align*}
    \left[p_{x,y}(t)\right]_{x,y=0,\dots,3} & =
    \begin{bmatrix}
        1 & 1 & 1 & 1 \\
        1 & 1+t^2 & 2+t^2 & 3+t^2 \\
        1 & 2+t^2 & 1+2\cdot 2 t^2+t^4 & 3+2\cdot 3t^2+t^4 \\
        1 & 3+t^2 & 3+2\cdot 3t^2+t^4 & 1+3 \cdot 3 t^2+3\cdot3 t^4+t^6
    \end{bmatrix}.
    \end{align*}
    Observe that the coefficients of $t^0$ correspond to Pascal's triangle (as they are given by $\binom{x}{y}$ for $x \geq y$) while the other coefficients are scaled and shifted versions thereof.
    Then we can for example write out
    \[
    s_{23} = \frac{t(3+6t^2+t^4)}{(1-t^2)^6} a_{10}^2a_{21}^2a_{32}w_0^{(2)} + \frac{t(2+t^2)}{(1-t^2)^4} a_{21}^2a_{32}w_1^{(2)} + \frac{t\cdot 1}{(1-t^2)^2} a_{32}w_2^{(2)}.
    \]
\end{example}
\bigskip

The restricted trek rule for DAGs can be extended to higher-order cumulants as 
\begin{align*}
    (T_n)_{i_1,\dots,i_n} &= \sum_{\tau = (\tau_1, \dots , \tau_n) \in \mathcal{T}^*(i_1, \dots, i_n)} C(d(\tau_1),\dots,d(\tau_n);t) a^{\tau_1} \cdots a^{\tau_n}w^{(n)}_{top(\tau)} \\
    \text{where } \quad C(x_1,\dots,x_n;t) &= t^{n\cdot\max(x_1,\dots,x_n)-\sum_{l=1}^n x_l} \frac{p_{x_1,\dots,x_n}(t)}{(1-t^n)^{\sum_{l=1}^n x_l +1}}.
\end{align*}

While the same induction strategy as for proving \cref{prop:trek_rule_dag} can be applied, indexing the terms and defining $p$ become much more involved. 
We conjecture that $p_{x,y}(t)$ generalizes to higher $n$ as
\[
p_{x_1,\dots,x_n}(t) = \hspace{-8pt} \sum_{l = 0}^{\sum_{i=1}^n x_i -\max(x_1,\dots,x_n)} \hspace{-7pt} t^{n(\sum_{i=1}^n x_i -\max(x_1,\dots,x_n)-l)} \sum_{k = 0}^l (-1)^{l-k} \binom{x_1+\cdots+x_1+1}{l-k}\prod_{i=1}^n\binom{x_i+k}{k}.
\]
Intuition on this based on generating functions of known integer sequences is given in \cref{app::dag_trek_rule_higher_cum}.

\subsection{Marginal Independence}\label{sec:marginal}

Similar to \cite{varando2020graphical} and \cite{boege2024conditional} in the continuous setting, we can derive marginal independencies from missing treks in the graph. The trek parametrization allows us to infer the marginal independence properties of the distributions in the discrete Lyapunov model based on results connecting higher-order cumulants and independence. 
Consider a multivariate distribution that is uniquely determined by its moments. If all mixed higher-order cumulants of $X_i$ and $X_j$ of such a distribution are zero, then $X_i$ and $X_j$ are marginally independent \cite{mccullagh2018tensor}. Combining this result with the parametrization of the cumulants via equitreks, we obtain the following result on marginal independence in the discrete Lyapunov model. Note that this statement can also be extended to a partition of the indices $I \cup J = V$ where there is no equitrek between the two sets of nodes. 

\begin{lemma}
    Let $G$ be a directed graph and consider a distribution in the discrete Lyapunov model $\mathcal{M}_G$ that is uniquely determined by its moments. If there is no equitrek between $i$ and $j$ in $G$, then all mixed higher-order cumulants of $X_i$ and $X_j$ are zero. 
    In particular, $X_i \indep X_j$ holds in the considered distribution.
\end{lemma}

\begin{proof}
    The trek rule allows us to write the $i\cdots ij\cdots j$-th entry of any higher order cumulant as a sum over all equitreks between $i$ and $j$ with the corresponding number of leafs.
    If there is no equitrek between $i$ and $j$, then all corresponding mixed higher-order cumulants are zero. As detailed above, this implies $X_i \indep X_j$ for all such distributions. 
\end{proof}

This is a more general statement than \cite[Cor. 2.3]{varando2020graphical}, as it includes non-Gaussian distributions. 
However, as in the continuous time setting, we can choose to restrict the error distributions of the underlying VAR(1) processes to be Gaussian. Then all cumulants of order higher than two are zero,
so it is enough to study the covariance structure. To emphasize this, we use the notation from \cref{rem::Gaussian_special_case}. 

\begin{corollary} \label{cor::marginal_independencies}
    Let $G$ be a directed graph and $\Sigma \in \mathcal{M}_G^{\leq 2}$ restricted to Gaussian distributions. If there is no equitrek between $i$ and $j$ in $G$, then 
    \[
    \Sigma_{ij} = 0. 
    \]
    In particular, $X_i \indep X_j$ holds in all distributions in $\mathcal{M}_G^{\leq 2}$.
\end{corollary}

Let $\hat{G} = (V,\hat{E})$ be the bidirected \emph{equitrek graph} derived from $G$ defined by including a bidirected edge between $i$ and $j$ in $\hat{E}$ if there is an equitrek between $i$ and $j$ in $G$. 
The same argument as in \cite{boege2024conditional} allows us  to describe the implied conditional independencies that hold for all distributions in the model in the Gaussian case by means of a graphical separation in $\hat{G}$ as follows.

\begin{lemma} \label{lem::implied_CIS}
    Let $G$ be a directed graph with corresponding equitrek graph $\hat{G}$, and let $I,J,K \subseteq V$ be disjoint index sets. Then
    \[
    V\setminus{(I \cup J \cup K)} \text{ separates } I \text{ and } J \text{ in } \hat{G} \quad \Rightarrow \quad X_I \indep X_J \mid X_K
    \]
    in the second-order discrete Lyapunov model $\mathcal{M}_G^{\leq 2}$ restricted to Gaussian distributions.
\end{lemma}

\begin{proof}
    See the argument in \cite{boege2024conditional} based on the connected set Markov property for bidirected graphs \cite{drton2008binary}.
\end{proof}

Similar to the continuous case, this statement is a special case of a setup with more general Markov processes. A general result in the case of discrete-time Markov processes is given by the second statement of \cite[Theorem 3.4]{niemiro2023local}, where it is shown that a certain graphical separation in the original graph implies a conditional independence statement holding for the corresponding stationary distributions. 
As conjectured in \cite{niemiro2023local}, this graphical separation might also be necessary for conditional independence, in the sense that if the separation does not hold, there exists a stationary distribution of a discrete-time Markov process with sparsity pattern given by the graph $G$ where the corresponding conditional independence does not hold. Similar to the approach for the continuous case in \cite{boege2024conditional}, it might be possible to construct these counterexample distributions directly in the discrete Lyapunov model as shown in the following example. 

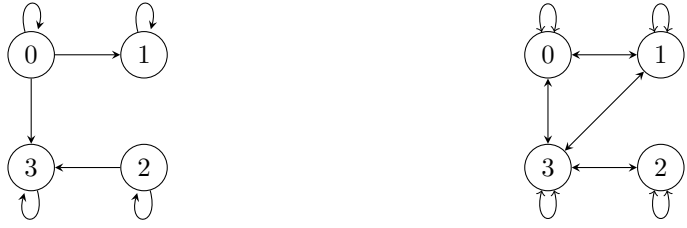
\begin{figure}
    \centering
    \begin{subfigure}[t]{0.4\textwidth}
    \centering
    \begin{tikzpicture}
        \node[circle, draw, minimum size=0.5cm] (0) at (0,1.5) {0};
        \node[circle, draw, minimum size=0.5cm] (1) at (1.5,1.5) {1};
        \node[circle, draw, minimum size=0.5cm] (2) at (1.5,0) {2};
        \node[circle, draw, minimum size=0.5cm] (3) at (0,0) {3};
        
        \draw[->,  >=stealth] (0) edge (1);
        \draw[->,  >=stealth] (0) edge (3);
        \draw[->,  >=stealth] (2) edge (3);
        \draw[->,  >=stealth] (0) edge[loop above] (0);
        \draw[->,  >=stealth] (1) edge[loop above] (1);
        \draw[->,  >=stealth] (2) edge[loop below] (2);
        \draw[->,  >=stealth] (3) edge[loop below] (3);
    \end{tikzpicture} 
    \caption{Directed graph $G$ on $p=4$ nodes including all self-loops.}
    \end{subfigure}
    \hspace{0.4cm}
    \begin{subfigure}[t]{0.4\textwidth}
    \centering
    \begin{tikzpicture}
        \node[circle, draw, minimum size=0.5cm] (0) at (0,1.5) {0};
        \node[circle, draw, minimum size=0.5cm] (1) at (1.5,1.5) {1};
        \node[circle, draw, minimum size=0.5cm] (2) at (1.5,0) {2};
        \node[circle, draw, minimum size=0.5cm] (3) at (0,0) {3};
        
        \draw[<->,  >=stealth] (0) edge (1);
        \draw[<->,  >=stealth] (0) edge (3);
        \draw[<->,  >=stealth] (2) edge (3);
        \draw[<->,  >=stealth] (1) edge (3);
        \draw[<->,  loop above] (0) to (0);
        \draw[<->,  loop above] (1) to (1);
        \draw[<->,  loop below] (2) to (2);
        \draw[<->,  loop below] (3) to (3);
    \end{tikzpicture} 
    \caption{Bidirected equitrek graph $\hat{G}$.}
    \end{subfigure}
    \caption{Directed graph with corresponding equitrek graph.}
    \label{fig::CI_example}
    \end{figure}

\begin{example}
    Consider the second-order discrete Lyapunov model $\mathcal{M}_G^{\leq2}$ restricted to Gaussian distributions for the graph on four nodes with all self-loops depicted in \cref{fig::CI_example}. Since there is no (equi)trek between $\{0,1\}$ and $\{2\}$ in $G$, we conclude from \cref{cor::marginal_independencies} that 
    \[
    X_{\{0,1\}} \indep X_2 
    \]
    holds for all distributions in $\mathcal{M}_G^{\leq2}$. Based on the resulting separations in the corresponding equitrek graph $\hat{G}$, \cref{lem::implied_CIS} further implies that 
    \[
    X_1 \indep X_2 \mid X_0 \text{ and } X_0 \indep X_2 \mid X_1 
    \]
    hold in all distributions in the model. 
    To show that no other conditional independencies hold for all distributions in the model, we need to construct a distribution in the model where, for example, 
    \[
    X_1 \indep X_3 \mid X_0
    \]
    does not hold. This is equivalent to the corresponding determinant 
    \[
    \det \Sigma_{\{1,0\}\times \{3,0\}} = \frac{a_{10} a_{30} (\omega^{(2)}_0)^2}{(1 - a_{00}^2) (1 - a_{00}a_{11}) (1 - a_{00}a_{33}) (1 - a_{11}a_{33})}
    \]
    being non-zero.
    It is clear that for almost all choices of the parameters of $A$ and $D$, this is fulfilled. Since we require the matrix $A$ to be stable, an easy choice to construct such a distribution is setting the edge weights to 1, the self-loop parameters to $\frac12$, and the error covariances to 1. Then we have $\det \Sigma_{\{1,0\}\times\{3,0\}} = \frac{256}{81}$, so $X_1 \not\!\indep X_3 \mid X_0$ in the corresponding distribution. Even though this approach works in small dimensions, a general way of constructing counterexample distributions to arbitrary conditional independence statements is required to prove the conjecture.
\end{example}

\begin{remark}
    Note that if all self-loops are present (or the graph is a DAG and at least all source nodes of the graph have a self-loop), the equitrek graph coincides with the trek graph known from the continuous setting. However, if some of the self-loops at the source nodes are missing, the equitrek graph is a sub-graph of the trek graph,
    thereby encoding additional independence statements.
\end{remark}

\section{Generic Identifiability}
\label{sec::ParameterIdentifiabilityDAGs}
In this section we study identifiability of the parameters in the discrete Lyapunov model for a known graph $G$.  
Recall that the parameters $ (A, \Omega^{(2)}, \Omega^{(3)}, \Omega^{(4)})$ of the fourth-order discrete Lyapunov model lie in the parameter space 
\begin{equation*}
    \Theta_{G, \leq 4} =  \mathbb{R}^{E}_{\text{stab}} \times \mathbb{R}_{+}^p \times (\mathbb{R} \setminus \{ 0 \})^p \times \mathbb{R}_{+} ^p.
\end{equation*}
The question of identifiability for fourth-order discrete Lyapunov models is whether the (rational) map 
$$\varphi_{G, \leq4}: \Theta_{G, \leq4} \rightarrow PD_p \times \text{Sym}^3(\mathbb{R}^p) \times \text{Sym}^{4}(\mathbb{R}^p), \quad (A, \Omega^{(2)}, \Omega^{(3)}, \Omega^{(4)}) \mapsto (S, T, R)$$
is injective,
where  $(S, T, R)$ denote the solutions to the corresponding second-, third- and fourth-order discrete Lyapunov equations. All of the identifiability results will initially be stated as identifying $A$ since it is clear by considering \eqref{eq::DiscreteLyapunovRecursive} that if $A$ can be identified from the cumulants $(S, T, R)$, then we can also identify the error cumulants.

In \cref{sec:identifiability_DAGs_self-loops} we establish generic (rational) identifiability for DAGs which have a self-loop at each vertex. In \cref{sec:identifiability_trees} we prove generic (rational) identifiability for polytrees which have self-loops at each source node. Finally, in \cref{sec:nonidentifiable_graphs} we present examples of some non-identifiable graphs.

\subsection{Identifiability for DAGs With All  Self-Loops}\label{sec:identifiability_DAGs_self-loops}
We first study generic identifiability for directed graphs $G$ that have a self-loop at each vertex and are otherwise acyclic. We refer to such graphs as \emph{DAGs with all self-loops}. Considering the definition of the VAR(1) process, the assumption of all the vertices having a self-loop is a reasonable one to make as the it corresponds to the value of the variable at time $t$ affecting its value at time $t+1$. This is why we initially consider this case.  

\begin{example} \label{ExampleIDTwoNodes}
Consider the graph $G$ in \cref{fig::TwoNodesTwoSelf-loops} on two nodes, which serves as the base case for the inductive proof of generic identifiability for DAGs with all self-loops in \cref{prop:IdAllSelfLoops}. 

\begin{figure}
    \centering
    \begin{tikzpicture}
    \node[circle, draw, minimum size=0.5cm] (0) at (0,0) {0};
    \node[circle, draw, minimum size=0.5cm] (1) at (2,0) {1};
  
    \draw[->,  >=stealth] (0) edge[loop above] (0);
    \draw[->,  >=stealth] (1) edge[loop above] (1);
    \draw[->,  >=stealth] (0) edge (1);
    \end{tikzpicture} 
    \caption{A graph on two vertices with two self-loops}
    \label{fig::TwoNodesTwoSelf-loops}
\end{figure}
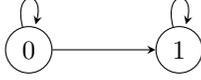
By employing the trek rule, we obtain the following equations for the second-, third- and fourth-order cumulants, denoted by $s$, $t$, and $r$ with respective subscripts, and the edge weight parameters $a_{00}, a_{10}$ and $a_{11}$:
$$\frac{s_{01}}{s_{00}} = \frac{a_{00}a_{10}}{1-a_{00}a_{11} }, \qquad \frac{t_{001}}{t_{000}} = \frac{a_{00}^2 a_{10}}{1 - a_{00}^2 a_{11} }, \qquad \frac{r_{0001}}{r_{0000}} =  \frac{a_{00}^3a_{10}}{1-a_{00}^3a_{11}}.$$
These equations can be solved uniquely for $a_{00}, a_{10}$ and $a_{11}$, yielding
\begin{align*}
    a_{00} &= \frac{-r_{0001}(s_{00}t_{001} - s_{01}t_{000})}{s_{01}(r_{0000}t_{001} - r_{0001}t_{000})}, \\
    a_{10} &= \frac{-s_{01}^2t_{001}(r_{0000}s_{01}t_{001} + r_{0001}s_{00}t_{001} - 2r_{0001}s_{01}t_{000})(r_{0000}t_{001} - r_{0001}t_{000})}{r_{0001}^2(s_{00}t_{001} - s_{01}t_{000})^3}, \\
    a_{11} &= \frac{(r_{0000}t_{001} - r_{0001}t_{000})s_{01}^2(r_{0000}s_{00}t_{001}^2 - r_{0001}s_{01}t_{000}^2)}{(s_{00}t_{001} - s_{01}t_{000})^3r_{0001}^2}.
\end{align*}

By employing the trek-rule once more, we see that the denominators above can never be zero if we assume the following: $A$ is Schur stable, $a_{00} \neq 0, a_{10} \neq 0$, and 
$\omega_{0}^{(2)} \neq 0, \omega_{0}^{(3)} \neq 0, \omega_{0}^{(4)} \neq 0$. Thus, we conclude that $G$ is generically identifiable from the second-, third- and fourth-order cumulants. 

Alternatively, it is possible to replace the fourth-order equation by another third-order equation involving $t_{011}/t_{000}$. In this case, there are two solutions, implying that $G$ is locally identifiable from the second- and third-order discrete Lyapunov equations alone. From experiments it furthermore seems to be the case that only one of these solutions yields a Schur stable matrix $A$. If this could be proven, then $A$ would also be generically identifiable from just the second-and third-order cumulants. 
\end{example}

\begin{remark}
    In the proofs of this section, we encounter the situation where we have several (potentially) different expressions for the same parameter $a_{ij}$ in terms of the cumulants of the model. However, all these equations are directly derived from the model definition and we further only obtain a single solution for $a_{ij}$ each time. Consequently, they all have to agree since otherwise, the corresponding original equation would not have been true in the model. 
\end{remark}

\begin{theorem} \label{prop:IdAllSelfLoops}
    Let $G = (V,E)$ be a DAG with $p \geq 2$ nodes that has no isolated nodes and a non-zero self-loop at each node. Consider the corresponding discrete Lyapunov model with edge weight matrix $A$. If $G$ is known (i.e., the zero pattern of $A$ is known), then $A$ is generically identifiable.
\end{theorem}

The proof can be found in \cref{sec:proofsParameterIdentifiabilty}. Once we have identified the matrix $A$, we can easily identify all cumulants of the noise, for instance by using the recursive equation~\eqref{eq::DiscreteLyapunovRecursive}.

\begin{remark}\label{rmk:isolatednode}
A graph which has an isolated node with a self-loop is neither generically nor locally identifiable using any number of cumulants since then the isolated node would itself have to be generically/locally identifiable.
This, however, is not possible since there will always be one more parameter than available equations.
\end{remark}

\begin{corollary} Let $G = (V,E)$ be a DAG with $p \geq 2$ nodes that has no isolated nodes and a non-zero self-loop at each node. Consider the corresponding discrete Lyapunov model with edge weight matrix $A$. If $G$ is known (i.e., the zero pattern of $A$ is known), then $A$, $\Omega^{(2)}$, $\Omega^{(3)}$, and $\Omega^{(4)}$ are generically identifiable.
\end{corollary}

\subsection{Other Types of DAGs}\label{sec:identifiability_trees}
We now consider graphs which might not have all self-loops present but are otherwise still acyclic. For a variable to not have a self-loop implies that its value at time $t$ does not affect it at time $t+1$, this could for example be the case of the time jumps are big compared to the speed of the process.

\begin{example} \label{ExampleIDTwoNodesOneLoop}
    Consider the graph on two nodes that has a directed edge from the source, $0$, to the other node, $1$, but only the source has a self-loop, as shown in \cref{fig::TwoNodes1self-loop}. 
By employing the trek rule, as well as directly considering the discrete Lyapunov equations, we obtain the following equations for the second- and third-order cumulants:
\begin{align*}
    &s_{00} = w_0^{(2)} \frac{1}{1-a_{00}^2},  &&t_{000} = w_0^{(3)} \frac{1}{1-a_{00}^3}, \\
    &s_{01} = a_{00}a_{10} s_{00}, &&t_{001} = a_{00}^2a_{10} t_{000}, \\
    &s_{11} = a_{10}^2 s_{00} + w_1^{(2)}, &&t_{111} = a_{10}^3 t_{000} + w_1^{(3)}.
\end{align*}
The third-order cumulant entry $t_{011}$ follows a similar pattern, but is not required for the following computation.
Combining these equations yields
\begin{align*}
    a_{00} &= \frac{s_{00} t_{001}}{s_{01} t_{000}}, 
    && \ a_{10} = \frac{s_{01}}{s_{00} a_{00}} = \frac{s_{01}^2 t_{000}}{s_{00}^2 t_{001}}, \\
    w^{(2)}_0 &= s_{00}(1-a_{00}^2) = \frac{s_{00} s_{01}^2 t_{000}^2 - s_{00}^3 t_{001}^2}{s_{01}^2 t_{000}^2}, 
    &&w^{(2)}_1 = s_{11} - a_{10}^2 s_{00} = \frac{s_{00}^3s_{11}t_{001}^2 - s_{01}^4 t_{000}^2}{s_{00}^4 t_{001}^2}, \\
    w^{(3)}_0 &= t_{000}(1-a_{00}^3) = \frac{s_{00} s_{01}^3 t_{000}^3 - s_{00}^4 t_{001}^3}{s_{01}^3 t_{000}^3}, 
    &&w^{(3)}_1 = t_{111} - a_{10}^3 t_{000} = \frac{t_{001}^3 t_{111} s_{00}^6 - s_{01}^6t_{000}^4}{s_{00}^6 t_{001}^3}.
\end{align*}
\end{example}

This example serves as the base case for the following theorem.

\begin{theorem}
\label{thm::GenericIDGeneral}
    Let $G = (V, E)$ be a polytree with $p \geq 2$ nodes that has no isolated nodes and non-zero self-loops at all sources. Consider the corresponding discrete Lyapunov model with edge weight matrix $A$. If $G$ is known (i.e., the zero pattern of $A$ is known), then $A$ is generically identifiable.
\end{theorem}

The proof can be found in \cref{sec:proofsParameterIdentifiabilty}. As in \cref{prop:IdAllSelfLoops}, it is easy to identify the cumulants of the noise once $A$ is identified. 

\begin{corollary} Let $G = (V,E)$ be a polytree with $p \geq 2$ nodes that has no isolated nodes and has a non-zero self-loop at all the sources. Consider the corresponding discrete Lyapunov model with edge weight matrix $A$. If $G$ is known (i.e. the zero pattern of $A$ is known), then $A$, $\Omega^{(2)}$, $\Omega^{(3)}$, and $\Omega^{(4)}$ are generically identifiable.
\end{corollary}
 
\begin{remark} 
\label{rmk:: onIdConditions}
The motivation behind \cref{thm::GenericIDGeneral} was to explore identifiability results in less restrictive cases than those in \cref{prop:IdAllSelfLoops}. One natural generalization is to consider graphs where some self-loops are missing. Note, that sources must still have self-loops, see \cref{ex:: nonId 1}. However, merely requiring this is not sufficient to guarantee identifiability, see \cref{ex:: nonId 4}. 

In the proofs of \cref{thm::GenericIDGeneral} and \cref{prop:IdAllSelfLoops}, we rely on the following: in the induction step, for each vertex, we are able to find linearly independent equations that identify the incoming edges. Having all the self-loops in a graph, or assuming the graph is a tree, ensures this property. However, there are alternative ways to formulate this condition.  

For example, the following statement could be proved using the same technique: Let $G=(V, E)$ be a DAG with $p\geq 2$ nodes, no isolated nodes, and self-loops at all sources. Further, for each non-source node $i$ with parents $\text{pa}(i) = \{v_1, \ldots, v_d\}$, suppose there exist $d$ nodes $\{u_1, \ldots, u_d\}$ that appear before $i$ in the topological order, such that $u_i$ is an ancestor of $v_i$ (possibly equal if $v_i$ has a self-loop), and $u_i$ is not an ancestor of $v_j$ for $i\neq j$, and $u_i\neq v_j$ for $j\neq i$. Consider the corresponding discrete Lyapunov model with edge weight matrix $A$. If $G$ is known, then $A$ is generically identifiable. In the proof, one would need to consider equations for $s_{iu_1}, \ldots, s_{i, u_d}$ to identify the edges $a_{iv_1}, \ldots, a_{iv_d}$. 

Finally, note that there exist graphs not satisfying the conditions above that are still identifiable, see \cref{ex:: exceptionId}. Thus, a natural direction for future research is to completely characterize the class of generically identifiable graphs.
\end{remark}

\begin{example} 
\label{ex:: exceptionId}
The graph in \cref{fig::GenericIDgraphNotsatisfyingPrevThm} does not satisfy the conditions from \cref{prop:IdAllSelfLoops}, \cref{thm::GenericIDGeneral}, or \cref{rmk:: onIdConditions}. Nevertheless, the parameters can still be generically identified using $s_{04}, t_{004}, s_{03}, s_{43}$.

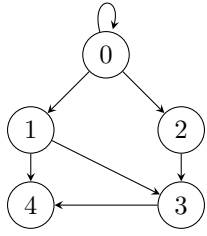
\begin{figure}
    \centering
    \begin{tikzpicture}
      \node[circle, draw, minimum size=0.5cm] (0) at (0,0) {0};
      \node[circle, draw, minimum size=0.5cm] (1) at (-1,-1) {1};
      \node[circle, draw, minimum size=0.5cm] (2) at (1,-1) {2};
      \node[circle, draw, minimum size=0.5cm] (4) at (-1,-2) {4};
      \node[circle, draw, minimum size=0.5cm] (3) at (1,-2) {3};
      
      \draw[->,  >=stealth] (0) edge[loop above] (0);
      \draw[->,  >=stealth] (0) edge (1);
      \draw[->,  >=stealth] (0) edge (2);
      \draw[->,  >=stealth] (1) edge (4);
      \draw[->,  >=stealth] (1) edge (3);
      \draw[->,  >=stealth] (2) edge (3);
      \draw[->,  >=stealth] (3) edge (4);
    \end{tikzpicture} 
    \caption{A generically identifiable graph.}
    \label{fig::GenericIDgraphNotsatisfyingPrevThm}
\end{figure}
\end{example}

\subsection{Nonidentifiable Graphs}\label{sec:nonidentifiable_graphs}

In this subsection, we present examples of small graphs that are not identifiable.

\subsubsection{Examples on Two Nodes}

We consider graphs on two nodes with various edge configurations, including graphs in which not all self-loops are present (in particular, the self-loop at the source) and graphs that contain a cycle, as shown in \cref{fig::non_id_2node_graphs}.

\begin{figure}[H]
    \centering
    \begin{subfigure}[t]{0.4\textwidth}
    \centering
    \begin{tikzpicture}
    \node[circle, draw, minimum size=0.5cm] (0) at (0,0) {0};
    \node[circle, draw, minimum size=0.5cm] (1) at (2,0) {1};
  
    \draw[->,  >=stealth] (1) edge[loop above] (1);
    \draw[->,  >=stealth] (0) edge (1);
    \end{tikzpicture} 
    \caption{A directed graph on two vertices with a single self-loop at the sink.}
    \label{fig::TwoNodes1self-loopsink}
    \end{subfigure}
    \hspace{1cm}
    ~ 
    \begin{subfigure}[t]{0.4\textwidth}
    \centering
    \begin{tikzpicture}
    \node[circle, draw, minimum size=0.5cm] (0) at (0,0) {0};
    \node[circle, draw, minimum size=0.5cm] (1) at (2,0) {1};
    
    \draw[->,  >=stealth] (1) edge[bend left] (0);
    \draw[->,  >=stealth] (0) edge[bend left] (1);
    \end{tikzpicture} 
    \caption{A cyclic directed graph on two vertices with no self-loops.}     \label{fig::CyclicGraph2nodes}
    \end{subfigure}
    \caption{Non-identifiable directed graphs on two nodes.}
    \label{fig::non_id_2node_graphs}
\end{figure}
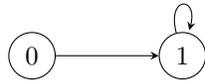
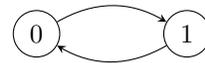

\begin{example}
\label{ex:: nonId 1}
    We consider the graph in \cref{fig::TwoNodes1self-loopsink}. 
    The unknown parameters are $a_{10}, a_{11}$, and $w^{(2)}_0$, $w^{(2)}_1, \dots, w^{(n)}_0, w^{(n)}_1$, when considering cumulants up to order $n$. We obtain the following equations:
    \begin{align*}
        &s_{00} = w^{(2)}_0, 
        &&t_{000} = w^{(3)}_0, 
        && 
        && q_{0\dots 00} = w^{(n)}_0, \\
        & 
        &&t_{001} = 0, 
        && 
        && q_{0\dots 01} = 0, \\
        &s_{01} = 0, 
        && 
        && \cdots 
        && \vdots \\
        & 
        &&t_{011} = 0, 
        &&
        && q_{01 \dots 1} = 0, \\
        &s_{11} = \frac{w^{(2)}_0 a_{10}^2 + w^{(2)}_1}{1-a_{11}^2}, 
        &&t_{111} = \frac{w^{(3)}_0 a_{10}^3 + w^{(3)}_1}{1-a_{11}^3}, 
        && 
        && q_{11 \dots 1} = \frac{w^{(n)}_0 a_{10}^n + w^{(n)}_1}{1-a_{11}^n}, 
    \end{align*}
    where the last equations describe the $n$th order cumulants. Note that there are no equitreks between 0 and 1, so the respective cumulant entries are zero. Consequently, there are $2(n-1)$ equations and $2(n-1) + 2$ unknowns. Therefore, the parameters cannot be identified using these equations for any~$n$. The same issue arises if the graph has one edge and no self-loops.
\end{example}

\begin{example}
    We consider the cyclic graph with two edges in \cref{fig::CyclicGraph2nodes}. 
    The unknown parameters are $a_{10}, a_{01}$, and $w^{(2)}_0, w^{(2)}_1, \dots, w^{(n)}_0, w^{(n)}_1$, when considering cumulants up to order $n$. 
    We obtain the following equations:
    \begin{align*}
        &s_{00} = \frac{w^{(2)}_0 + w^{(2)}_1 a_{01}^2}{1-a_{10}^2a_{01}^2}, 
        &&t_{000} = \frac{w^{(3)}_0 + w^{(3)}_1 a_{01}^3}{1-a_{10}^3a_{01}^3}, 
        && \dots
        &&q_{0\dots 0} = \frac{w^{(n)}_0 + w^{(n)}_1 a_{01}^n}{1-a_{10}^na_{01}^n}, \\
        &s_{11} = \frac{w^{(2)}_0 a_{10}^2 + w^{(2)}_1}{1-a_{10}^2a_{01}^2},
        &&t_{111} = \frac{w^{(3)}_0 a_{10}^3 + w^{(3)}_1}{1-a_{10}^3a_{01}^3},
        && \dots
        &&q_{1 \dots 1} = \frac{w^{(n)}_0 a_{10}^n + w^{(n)}_1}{1-a_{10}^na_{01}^n},
    \end{align*}
    where the last equations describe the $n$th order cumulants. Again, there are no equitreks between 0 and 1, so the respective cumulant entries are zero. Consequently, there are again $2(n-1)$ equations and $2(n-1) + 2$ unknowns, making this graph non-identifiable as well.
\end{example}

\begin{remark}
    These examples illustrate a more general observation: The number of $n$th order cumulants of a $p$-dimensional random vector is $\binom{p+n-1}{n}$, so in the $n$th order discrete Lyapunov model, we have at most $\sum_{i=2}^n\binom{i+p-1}{i} = \binom{p+n}{n}-(p+1)$ equations. Thus, if the model is generically identifiable, we have that $|E| + p(n-1) \leq \binom{p+n}{n}-(p+1)$. This bound can be further tightened by subtracting the number of zero cumulants due to missing $i$-equitreks for $i =2,\dots,n$ in the graph from the right-hand side (see, e.g., \cite{dettling2023identifiability}). In the examples above, this yields $|E| + 2(n-1) \leq 2(n-1)$. If there are no or only few equitreks missing, however, this bound does not present a notable restriction.
\end{remark}

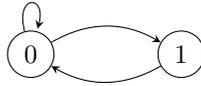
\begin{figure}[H]
    \centering
     \begin{tikzpicture}
      \node[circle, draw, minimum size=0.5cm] (0) at (0,0) {0};
      \node[circle, draw, minimum size=0.5cm] (1) at (2,0) {1};
      
      \draw[->,  >=stealth] (1) edge[bend left] (0);
      \draw[->,  >=stealth] (0) edge[bend left] (1);
      \draw[->,  >=stealth] (0) edge[loop above] (0);
    \end{tikzpicture} 
    \caption{A cyclic graph on two vertices with one self-loop.}
    \label{fig::CyclicTwoNodesThreeEdges}
\end{figure}

\begin{example}
    We consider the cyclic graph with three edges in \cref{fig::CyclicTwoNodesThreeEdges}. 
    The unknown parameters are $a_{00}, a_{10}, a_{01}$, and $w^{(2)}_0, w^{(2)}_1, \dots, w^{(n)}_0, w^{(n)}_1$, when considering the cumulants up to order~$n$. For this edge configuration, the parameters are locally identifiable from second- and third-order cumulants. 
\end{example}

\subsubsection{Diamond-type Graphs}
There are also examples of non-identifiable graphs that do not arise from a lack of equations compared to the number of parameters.
\begin{example}
\label{ex:: nonId 4}
We consider the graph on four vertices shown in \cref{fig::DiamondGraph}.
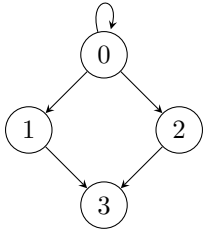
\begin{figure}[H]
    \centering
       \begin{tikzpicture}
      \node[circle, draw, minimum size=0.5cm] (0) at (2,3) {0};
      \node[circle, draw, minimum size=0.5cm] (1) at (1,2) {1};
      \node[circle, draw, minimum size=0.5cm] (2) at (3,2) {2};
      \node[circle, draw, minimum size=0.5cm] (3) at (2,1) {3};
      
      \draw[->,  >=stealth] (0) edge (1);
      \draw[->,  >=stealth] (0) edge (2);
      \draw[->,  >=stealth] (1) edge (3);
      \draw[->,  >=stealth] (2) edge (3);
      \draw[->,  >=stealth] (0) edge[loop above] (0);
    \end{tikzpicture} 
    \caption{A diamond-type graph.}
    \label{fig::DiamondGraph}
\end{figure}
Since the subgraph induced by $0,1,2$ satisfies the conditions of \cref{thm::GenericIDGeneral}, the parameters $a_{00}$, $a_{10}$ and $a_{20}$ are identifiable from the second-, third- and fourth-order cumulants (see \cref{ExampleIDTwoNodesOneLoop}). We also have
\[
s_{03} = w_0^{(2)} \frac{a_{10}a_{31} +a_{20}a_{32}}{1-a_{00}^2}.
\]
Hence, the quantity $y = a_{10}a_{31} +a_{20}a_{32}$ is identifiable from these cumulants. Note that this is a linear combination of the two unknown values $a_{31}$ and $a_{32}$ with known coefficients. However, we claim that $a_{31}$ and $a_{32}$ are not identifiable from any cumulants, regardless of how high the order is allowed to be. 

Let $q_{0^{k_0} 1^{k_1} 2^{k_2} 3^{k_3}}$ denote the $i$th order cumulant with $k_j$ copies of $j$ in its subscript for $j = 0,\dots,3$, where $i = k_0 + \cdots + k_3$.
Note that $a_{31}$ or $a_{32}$ appears in the trek parametrization of this cumulant if and only if $k_3 \neq 0$. In this case, for the off-diagonal cumulants (i.e., those with $k_3 \neq i$), we have
\[
q_{0^{k_0} 1^{k_1} 2^{k_2} 3^{k_3}} = w_0^{(i)} \frac{a_{00}^{2k_0}(a_{00}a_{10})^{k_1}(a_{00}a_{20})^{k_2}y^{k_3}}{1-a_{00}^{i}}.
\]
All $i$th order cumulants with $k_3 = 0$ are rational functions only of $a_{00}, a_{10}$ and $a_{20}$.

Let $A$ be any parameter matrix consistent with the prescribed cumulants of all orders, and let $\hat{A}$ be any parameter matrix which satisfies
$$\hat{a}_{00} = a_{00}, \qquad \hat{a}_{10} = a_{20}, \qquad \hat{a}_{20} = a_{20}, \qquad a_{10}\hat{a}_{31} +a_{20}\hat{a}_{32} = a_{10}a_{31} +a_{20}a_{32}.$$
Then $\hat{A}$ has the same prescribed cumulants as $A$. Therefore, there is a line of parameters that give the same cumulants, yielding this model non-identifiable.
Note that in the context of the proof of \cref{prop:IdAllSelfLoops}, this property results in the product of  submatrices of the $A$ and $S$ in \eqref{eq:DAGSubmatrices} having rank~1.
\end{example}

\section{Local Identifiability Using the Jacobian}\label{sec:local identifiability}
If the Jacobian of the parametrization has full rank, then the parametrization is locally identifiable \cite[Section 16.1]{sullivant2018algebraic}.
In this section, we derive an expression for the Jacobian of both the second- and third-order cumulants of the discrete Lyapunov model. We show that, for any graph where all connected components have at least size three and all self-loops are present, the joint second- and third-order Jacobian has full rank. This proves that 
all discrete Lyapunov models for such a graph are locally identifiable. By further including the Jacobian of the fourth-order cumulant we can prove that any graph with all self-loops and no isolated nodes are locally identifiable.  

As was seen in \cref{ExampleIDTwoNodes}, the graph on two nodes with both self-loops and an edge from $0$ to $2$ is generically identifiable from second-, third-, and fourth-order cumulants. However, it is already two-to-one identifiable (so locally identifiable) from the second- and third-order cumulants alone. It is this phenomenon which we will generalize in the section.

Fix a directed graph $G$ on $p$ vertices labeled from 0 to $p-1$. Let $\varphi_{G, \leq 3}$ be the map which sends a tuple of parameters $(A,\omegatwo,\omegathree)$ to the corresponding second- and third-order cumulants $(S,T)$ under the discrete Lyapunov model specified by $G$, so the third-order version of $\varphi_{G, \leq 4}$ from ~\cref{sec::ParameterIdentifiabilityDAGs}. Let $\mathrm{Jac}^G(S,T)$  denote the Jacobian of $\phi_{G, \leq 3}$.

We derive the following expression for the column of the Jacobian corresponding to the parameter $a_{\beta\alpha}$ using the Kronecker parametrization, equation \eqref{eq::discreteLyapunovVec}. By applying standard rules for differentiating a matrix inverse, which can be found for example in \cite[Section~2.2]{petersen2008matrix}, we obtain
\begin{align}
\label{eq::JacOfSigmawrtA}
    \text{vec} \left( \frac{\partial S}{\partial a_{\beta\alpha}} \right) &= -(I- A \otimes A)^{-1} \frac{\partial (I - A \otimes A) }{\partial a_{\beta\alpha}} (I- A \otimes A)^{-1} \text{vec}(\Omega^{(2)}) \nonumber \\
   &= (I- A \otimes A)^{-1} (E_{\beta\alpha} \otimes A + A \otimes E_{\beta\alpha}) (I- A \otimes A)^{-1} \text{vec}(\Omega^{(2)}), 
\end{align}
where $E_{\beta\alpha}$ is the matrix that is $1$ at the $(\beta,\alpha)$th position and 0 elsewhere. Note that the product of the last two matrices in~\eqref{eq::JacOfSigmawrtA} is exactly $\mathrm{vec}(S)$ as given by the Kronecker product parametrization. 

The column corresponding to the parameter $\omega^{(2)}_{i}$ is
\begin{align*}
    \vec \left( \frac{\partial S}{\partial \omega^{(2)}_{i}} \right)=(I - A \otimes A)^{-1} \frac{\partial \text{vec}(\Omega^{(2)})}{\partial \omega^{(2)}_{i}}  =(I - A \otimes A)^{-1}\text{vec}(E_{ii}).
\end{align*}

Similarly, we derive the following expression for the Jacobian of the third-order cumulant with respect to the entries of $A$:
\begin{align}
    \label{eq::JacofTwrtA}
   & \vec \left( \frac{\partial T}{\partial a_{\beta\alpha}} \right) \nonumber  = -(I- A \otimes A \otimes A)^{-1} \frac{\partial (I - A \otimes A \otimes A) }{\partial a_{\beta\alpha}} (I- A \otimes A \otimes A)^{-1} \text{vec}(\Omega^{(3)}) \nonumber \\
   &= (I- A \otimes A \otimes A)^{-1} (E_{\beta\alpha} \otimes A \otimes A + A \otimes E_{\beta\alpha} \otimes A + A \otimes A \otimes E_{\beta\alpha}) (I- A \otimes A \otimes A)^{-1} \text{vec}(\Omega^{(3)}),
\end{align}
and with respect to the $\omega^{(3)}_i$'s
\begin{align*}
    \vec \left( \frac{\partial T}{\partial \omega^{(3)}_{i}} \right)=(I - A \otimes A \otimes A)^{-1} \frac{\partial \text{vec}(\Omega^{(3)})}{\partial \omega^{(3)}_{i}} =  (I - A \otimes A \otimes A)^{-1}\text{vec}(E_{iii}). 
\end{align*}

From the above derivative computations, we observe that $J(S,T)$ can be written as a product of block matrices:
\begin{align}\label{eq::DecompOfJacobian}
\mathrm{Jac}^G(S,T) = \begin{pmatrix}
    (I - A \otimes A)^{-1} & 0 \\
    0 & (I - A \otimes A \otimes A)^{-1}
\end{pmatrix}
\begin{pmatrix}
    J^G_2(S,a) & J^G_2(S,\omega^{(2)}) & 0 \\
    J^G_3(T,a) & 0 & J^G_3(T,\omega^{(3)}) 
\end{pmatrix}.
\end{align}
We define the blocks of this second matrix as follows.
The matrix $J^G_2(S,a)$ is a $\binom{p+1}{2} \times \#E$ matrix whose $\beta\alpha$-th column (i.e., the column corresponding to the edge $\alpha \to \beta$) is 
\[(E_{\beta\alpha} \otimes A + A \otimes E_{\beta\alpha}) (I- A \otimes A)^{-1} \text{vec}(\Omega^{(2)}).\]
Similarly, $J^G_3(S,A)$ is the $\binom{p+2}{3} \times \#E$ matrix whose $\beta\alpha$-th column is 
\[
(E_{\beta\alpha} \otimes A \otimes A + A \otimes E_{\beta\alpha} \otimes A + A \otimes A \otimes E_{\beta\alpha}) (I- A \otimes A \otimes A)^{-1} \text{vec}(\Omega^{(3)}).
\]
The matrix $J^G_2(S,\omegatwo)$ is the $\binom{p+1}{2} \times p$ matrix whose $i$-th column is $\vec(E_{ii})$. Finally, $J^G_3(S,\omegathree)$ is the $\binom{p+2}{3} \times p$ matrix whose $i$-th column is $\vec(E_{iii})$. 
\begin{remark}
If we want the image of $\phi_G$ to contain cumulants of order higher than 3 as well, analogous formulas can be derived for the rows of the Jacobian  corresponding to these higher-order cumulants.
\end{remark}

\begin{definition}
    Let $G$ be a directed graph and let $A \in \mathbb{R}^{E}$. The \emph{modified Jacobian} of the second- and third-order cumulants of the discrete Lyapunov model, denoted by $\J^G(S,T)$, is defined by 
    \begin{equation*}
        \mathcal{J}^G(S, T) = \begin{pmatrix}
    J^G_2(S,a) & J^G_2(S,\omega^{(2)}) & 0 \\
    J^G_3(T,a) & 0 & J^G_3(T,\omega^{(3)}) 
\end{pmatrix},
    \end{equation*}
    as described in \eqref{eq::DecompOfJacobian}. 
\end{definition}

Since the original Jacobian is the modified Jacobian multiplied by an invertible block matrix from the left, their ranks are equal.  
Since the entries of the modified Jacobian will turn out to have a simpler description (see \cref{prop::J_2(Sigma)descriptionOfEntries} and \cref{prop::J_3(T)descriptionOfEntries}), we will determine its rank instead. 

We notice that the last $2p$ columns of the modified Jacobian are in fact unit vectors, so by reordering the rows of the modified Jacobian its rank decomposes as 
\begin{equation}
\label{eq::rankFullModifiedJacobian}
\mathrm{rank}(\mathcal{J}^G(S,T)) 
= \mathrm{rank}
\begin{pmatrix}
    J^G_2(S,a)_{\mathrm{off}} & 0 & 0 \\
    J^G_3(T,a)_{\mathrm{off}} & 0 & 0 \\
    J^G_2(S,a)_{\mathrm{diag}} & I_p & 0 \\
    J^G_3(T,a)_{\mathrm{diag}} & 0 & I_p \\
\end{pmatrix} 
= 2p + \mathrm{rank}\begin{pmatrix}
    J^G_2(S,a)_{\mathrm{off}} \\
    J^G_3(T,a)_{\mathrm{off}}
\end{pmatrix},
\end{equation}
where we denote the submatrix whose rows correspond to diagonal entries of the cumulants with '$\mathrm{diag}$' and the submatrix whose rows correspond to off-diagonal entries of the cumulants with `$\mathrm{off}$'. 
Thus, 
\begin{equation}
\label{eq::}
\mathcal{J}^G_{\mathrm{off}}(S, T) = 
    \begin{pmatrix}
    J^G_2(S,a)_{\mathrm{off}} \\
    J^G_3(T,a)_{\mathrm{off}}
\end{pmatrix},
\end{equation}
has full rank if and only if the modified Jacobian has full rank. We will refer to this matrix as the off-diagonal modified Jacobian. 
In order to be able to show that the modified Jacobian has full rank we first show the following two results describing its entries, with proofs provided in \cref{app::ProofsForLocalID}. 

\begin{proposition}
\label{prop::J_2(Sigma)descriptionOfEntries}
    The entries of $J^G_2(S,a)$ are given by
\begin{align*}
    J^G_2(S, a)_{(ij), \alpha \rightarrow \beta} = \delta_{j}(\beta) \sum_{l = 0}^{p-1} a_{il} s_{l \alpha} + \delta_{i}(\beta) \sum_{k = 0}^{p-1} a_{jk } s_{k \alpha}. 
\end{align*}
Furthermore, the description can be given by employing treks in $G$:
\begin{align*}
    J^G_2(S,a)_{(ij), \alpha \rightarrow \beta} = \delta_{j}(\beta) \sum_{\tau \in \mathcal{T}_{(1,0)}(i, \alpha)} m_{\tau} + \delta_{i}(\beta) \sum_{\tau \in \mathcal{T}_{(1,0)}(j, \alpha)} m_{\tau},
\end{align*}
where $\mathcal{T}_{(1,0)}(i,j)$ denotes the set of treks between $i$ and $j$ where the path going to $i$ is one longer than the path going to $j$, and $m_{\tau}$ is the trek monomial for a trek $\tau$.

\end{proposition}

\begin{proposition}
\label{prop::J_3(T)descriptionOfEntries}
The entries of $J^G_3(T,a)$ are given by
    \begin{equation*}
        J^G_3(T,a)_{(ijk), \alpha \rightarrow \beta} = \delta_{i}(\beta) \sum_{m = 0}^{p-1}\sum_{n = 0}^{p-1} a_{jm} a_{kn} t_{\alpha mn} + \delta_{j}(\beta) \sum_{l = 0}^{p-1} \sum_{n = 0}^{p-1} a_{il} a_{kn} t_{l \alpha n} + \delta_{k}(\beta) \sum_{l = 0}^{p-1} \sum_{m = 0 }^{p-1} a_{il} a_{jm} t_{l m \alpha}.
    \end{equation*}
Furthermore, the description can be given by employing treks in $G$:
    \begin{equation*}
        J^G_3(T, a)_{(ijk), \alpha \rightarrow \beta} = \delta_{i}(\beta) \sum_{\tau \in \mathcal{T}_{(1,1,0)}(j,k,\alpha)} m_{\tau} + \delta_{j}(\beta) \sum_{\tau \in \mathcal{T}_{(1,1,0)}(i,k,\alpha)} m_{\tau}  + \delta_{k}(\beta) \sum_{\tau \in \mathcal{T}_{(1,1,0)}(i,j,\alpha)} m_{\tau},
    \end{equation*}
where $\mathcal{T}_{(1,1,0)}(i,j,k)$ denotes the set of all treks between $i,j$ and $k$ where the paths going to $i$ and $j$ are one longer than the path going to $k$, and $m_{\tau}$ is the trek monomial for a trek $\tau$. 
\end{proposition}

\begin{remark}
    Similar formulas for the modified Jacobian can be derived analogously for higher-order cumulants.  
\end{remark}

We are now ready to state the main result of this section. 

\begin{theorem}
\label{thm:LocIdAllGraphs}
    Let $G = (V,E)$ be any connected directed graph with all self-loops on $p \geq 3$ nodes, then $\mathcal{J}_{\mathrm{off}}^{G}(S, T)$ has full rank ($= |E|$) generically. Therefore, the corresponding discrete Lyapunov model is locally identifiable from its second- and third-order cumulants. 
\end{theorem}
The proof of \cref{thm:LocIdAllGraphs} (and its modifications) proceed by using the local identifiability of models associated to subgraphs of $G$ to recursively prove that the model associated to $G$ itself is locally identifiable. In particular, we consider subgraphs obtained by edge-disconnecting $G$.

\begin{definition}
    Let $G = (V,E)$ be a connected graph. A subset $F \subset E$ \emph{edge-disconnects} $G$ if the graph $(V, E \setminus F)$ is disconnected. The set $\Xi(G)$ is the set of all graphs of the form $(V,E\setminus F)$ with exactly two connected components.
\end{definition}

Specifically, the proof of \cref{thm:LocIdAllGraphs} proceeds by induction on the number of nodes, where the induction step relies on being able to edge-disconnect the graph into two connected components of size at least three. Therefore, the base case includes all connected graphs on three, four or five nodes and all graphs on at least six nodes where it is not possible to disconnect the graph into two connected components of at least size three. 

Therefore, to understand the base, we must characterize the connected graphs where no matter how the graph is edge-disconnected a connected component of size one or two is always created. Note that the self-loops do not impact this property, so they will be omitted in the following discussion. Furthermore, this property is completely characterized by the skeleton of the graph and therefore the description will be given for undirected graphs.

It is possible to describe the base cases by considering the intersections of paths in the graph.
For example, let $G$ be a graph such that every graph in $\Xi(G)$ has an isolated node. Then every pair of paths in $G$ of vertex length two intersect at a single vertex. In other words, there is a vertex $v$ that is in every edge, so the skeleton of $G$ is a star, as pictured in \cref{fig::StarAndGen2star}.
Similarly, consider the graphs such that every graph in $\Xi(G)$ has a connected component with exactly two vertices. In this case, every pair of paths of vertex length three must intersect in at least one vertex. In \cref{lem::CharecterizeGen2star}, we show that if $G$ has six or more vertices, these graphs are the so-called \emph{generalized 2 stars}.

\begin{definition}
\label{def::Generalized2star}
    Let $G = (V,E)$ be a connected undirected graph with $|V| \geq 6$.  Then $G$ is a \emph{generalized two star} if all paths of vertex length 3 intersect in the same node. We call this node the \emph{center} of the graph. 
\end{definition}

\begin{figure}[H]
    \centering
\begin{center}
    \begin{tikzpicture}
      \node[circle, draw, minimum size=0.2cm, fill = gray] (1) at (0,1) {};
      \node[circle, draw, minimum size=0.2cm, fill = gray] (2) at (-2,-0.5) {};
      \node[circle, draw, minimum size=0.2cm, fill = gray] (3) at (-1,-0.5) {};
      \node[circle, draw, minimum size=0.2cm, fill = gray] (4) at (0,-0.5) {};
      \node[circle, draw, minimum size=0.2cm, fill = gray] (5) at (1,-0.5) {};
      \node[circle, draw, minimum size=0.2cm, fill = gray] (6) at (2,-0.5) {};

        
      \node[circle, draw, minimum size=0.2cm, fill = gray] (7) at (6,1) {};
      \node[circle, draw, minimum size=0.2cm, fill = gray] (8) at (5,-0.5) {};
      \node[circle, draw, minimum size=0.2cm, fill = gray] (9) at (7,-0.5) {};
      \node[circle, draw, minimum size=0.2cm, fill = gray] (10) at (5,2.5) {};
      \node[circle, draw, minimum size=0.2cm, fill = gray] (11) at (7,2.5) {};
      \node[circle, draw, minimum size=0.2cm, fill = gray] (12) at (8,1) {};

      \draw[-,  >=stealth] (1) edge  (2);
      \draw[-,  >=stealth] (1) edge  (3);
      \draw[-,  >=stealth] (1) edge  (4);
      \draw[-,  >=stealth] (1) edge  (5);
      \draw[-,  >=stealth] (1) edge  (6);

      \draw[-,  >=stealth] (7) edge  (8);
      \draw[-,  >=stealth] (7) edge  (9);
      \draw[-,  >=stealth] (7) edge  (10);
      \draw[-,  >=stealth] (7) edge  (11);
      \draw[-,  >=stealth] (8) edge  (9);
      \draw[-,  >=stealth] (11) edge  (12);

    \end{tikzpicture} 
    \end{center}
    \caption{Examples of a star (left) and a generalized two star (right).}
    \label{fig::StarAndGen2star}
\end{figure}
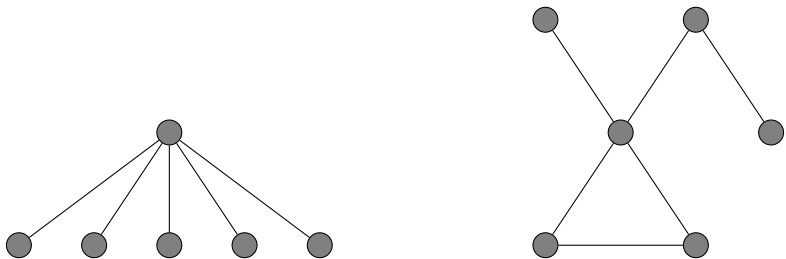

Due to the block structure of the Jacobian, the statement and proof of \cref{thm:LocIdAllGraphs} can be adapted to directed graphs with several connected components. Furthermore, connected components of size two are allowed as long as they are not the complete graph on two nodes, since they are also locally identifiable from the second- and third-order cumulants, see \cref{ExampleIDTwoNodes} or the calculations performed in Maple\footnote{\url{https://github.com/cecilie2424/Local-Identifiability-in-Non-Gaussian-Discrete-Lyapunov-Models}.}. 

\begin{corollary} \label{cor::LocalID_connected_components}
    Let $G = (V,E)$ be any directed graph with all self-loops. If all connected components of $G$ have at least size two and the components of size two are not the two-cyle then $G$ is generically locally identifiable from the second- and third-order cumulants. 
\end{corollary}

The restrictions on the size of the connected components of the graph are motivated by non-identifiability of the one-node graph (with a self-loop) and the two-cycle on two nodes (with two self-loops) from just the second- and third-order cumulants. 
A graph which has an isolated node with a self-loop is not locally identifiable using any number of cumulants as discussed in \Cref{rmk:isolatednode}. 
The two-cycle can only be locally identified by also employing some fourth-order cumulants. 

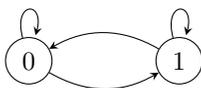
\begin{figure}[H]
    \centering
\begin{center}
    \begin{tikzpicture}
      \node[circle, draw, minimum size=0.5cm] (1) at (0,0) {0};
      \node[circle, draw, minimum size=0.5cm] (2) at (2,0) {1};

      \draw[->,  >=stealth] (1) edge[loop above] (1);
      \draw[->,  >=stealth] (2) edge[loop above] (2);
      \draw[->,  >=stealth] (1) edge [bend right] (2);
      \draw[->,  >=stealth] (2) edge[bend right] (1);
    \end{tikzpicture} 
    \end{center}
    \caption{The complete graph on two nodes.}
    \label{fig::TwoCycle}
\end{figure}

\begin{example}
\label{ex::twoCycleFinitetoone}
For the two-cycle on two nodes, the second- and third-order cumulants are not enough to obtain full rank of the modified Jacobian, since there will be one more parameter than there are equations. However, it is shown computationally in Maple\footnotemark[\value{footnote}] 
that, by adding the two diagonal fourth-order cumulants, the two $\omega^{(4)}$'s and one off-diagonal fourth-order cumulant (so in total using three fourth-order cumulants) we can prove that the modified Jacobian of both the second-, third- and the three fourth-order cumulants will have full rank. Therefore, the actual Jacobian will also have full~rank. 
\end{example}

From \cref{ex::twoCycleFinitetoone} it follows that the proof of \cref{thm:LocIdAllGraphs} could be modified to also allow for the use of fourth-order cumulants which would then yield local identifiability of discrete Lyapunov models for any graph which has all self-loops and no isolated nodes. Thus, it includes the few graphs not covered by \cref{cor::LocalID_connected_components}, i.e. the graphs which have a connected component which is the complete graph on two nodes. However, all other graphs were already included in \cref{cor::LocalID_connected_components}. We make this distinction of when the fourth-order cumulants are needed because in practice it is difficult to estimate higher-order cumulants accurately, and the difficulty grows with the order of the cumulant. Therefore, it is desirable to obtain the identifiability result from the lowest possible order of cumulants

\begin{corollary}
    Let $G = (V,E)$ be any directed graph with all self-loops on $p \geq 2$ nodes with no isolated nodes, then $\mathcal{J}_{\mathrm{off}}^{G}(S, T)$ has full rank ($= |E|$) generically. It follows that the corresponding discrete Lyapunov model is locally identifiable from its second-, third-, and fourth-order cumulants. 
\end{corollary}

\begin{proof}[Proof of \cref{thm:LocIdAllGraphs}]
    The proof proceeds by induction on the number of nodes. The base case is any graph on three, four, or five nodes, and all graphs with $|V| \geq 6$ whose skeleton is a star or generalized two star, as explained in the preceding paragraph. 
    The base cases are covered by \cref{lem::Case345,lem:generalized2Star}. 

    Let $G = (V,E)$ be a graph on $p\ge6$ nodes, where $V=\{0\}\cup [p-1]$. If $G$ is a star or generalized two star, then $G$ is covered by the base case. Therefore, we now assume this is not the case. This implies that $G$ can be edge-disconnected into exactly two connected components where each component has at least three vertices. Let $G' = (V,E')$ denote such an edge-disconnecting, and let $G_1 = (V_1, E_1)$ and $G_2 = (V_2, E_2)$ be the two connected components of $G'$. Further, let $|V_i| = n_i$ for $i = 1,2$, and write the vertices as $V_i = \{g^{i}_1, \dots, g^{i}_{n_i} \}$. We now let the sparsity pattern of $A$ be given by $G'$.

    Let $\mathcal{J}^{G_i}_{\mathrm{off}}(S, T)$ denote the columns of the modified Jacobian corresponding to $G_i$.
    The induction hypothesis then implies that $\mathcal{J}_{\mathrm{off}}^{G_1}(S, T)$ and $\mathcal{J}_{\mathrm{off}}^{G_2}(S, T)$ both have full rank generically. 

    We aim to show that $\mathcal{J}_{\mathrm{off}}^G(S, T)$ has generically full rank. 
    To simplify the notation, we construct a matrix $X$ that contains $\mathcal{J}_{\mathrm{off}}^G(S, T)$ as a submatrix. The only difference is that $X$ will (potentially) include columns corresponding to edges that are missing in $G$.

    The rows and columns of $X$ will be partitioned into four groups. 
    The first two groups correspond to submatrices of $\mathcal{J}_{\mathrm{off}}^{G_1}(S, T)$ and $\mathcal{J}_{\mathrm{off}}^{G_2}(S, T)$, respectively. The rows are chosen so that the submatrix has full rank, which is possible by the induction hypothesis. Specifically, the rows are
    \begin{align*}
        R_{G_1} = \{ (i,j) \; | \; i,j \in V_1, i \neq j \} \cup \{(i,j,k) \; i,j,k \in V_1, i \neq j \neq k \}, \\
        R_{G_2} = \{ (i,j) \; | \; i,j \in V_2, i \neq j \} \cup \{(i,j,k) \; i,j,k \in V_2, i \neq j \neq k \}, 
    \end{align*}
and the columns correspond to the edges in each connected component
    \begin{align*}
        C_{G_1} = \{ \alpha \rightarrow \beta \in E_{1} \} \qquad  \text{ and } \qquad 
        C_{G_2} = \{ \alpha \rightarrow \beta \in E_{2} \},
    \end{align*}
    respectively. We denote these two matrices by $\mathcal{J}^{G_1}$ and $\mathcal{J}^{G_2}$. The remaining blocks correspond to potential edges between the connected components. For $j = 1, \dots, n_2$, consider the rows and columns 
    \begin{align*}
        R_{G_1 \rightarrow g^{2}_j} = \{ (g^1_k g^1_k g^2_j) \; | \; k = 1, \dots, n_1 \}, \qquad C_{G_1 \rightarrow g^{2}_j} = \{ (g^1_k \rightarrow g^2_j) \; | \; k = 1, \dots, n_1 \}.
    \end{align*}
    Conversely, for $j = 1, \dots, n_1$, consider the rows and columns 
    \begin{align*}
        R_{G_2 \rightarrow g^{1}_j} = \{ (g^2_k g^2_k g^1_j) \; | \; k = 1, \dots, n_2 \}, \qquad C_{G_2 \rightarrow g^{1}_j} = \{ (g^2_k \rightarrow g^1_j) \; | \; k = 1, \dots, n_2 \}. 
    \end{align*}
    With this ordering of rows and columns, the matrix $X$ has the following block structure: 
    \begin{equation*}
    \begin{pNiceArray}{ccccccccc}[first-row,first-col, nullify-dots]
        & C_{G_1} & C_{G_2} & C_{G_1 \rightarrow g^{2}_1} & \dots & C_{G_1 \rightarrow g^{2}_{n_2}} & C_{G_2 \rightarrow g^{1}_1} & \dots & C_{G_2 \rightarrow g^{1}_{n_1}} \\ 
        R_{G_1} \quad\ & \mathcal{J}^{G_1} & & & & & & & \\ 
        R_{G_2} \quad\ & & \mathcal{J}^{G_2} & & & & & & \\ 
        R_{G_1 \rightarrow g^{2}_1} \;\,\ & & & M_{G_1 \rightarrow g^{2}_1} & & & \text{\Large 0} & & \\ 
        \vdots \qquad\ & & & & \rotatebox[origin=c]{20}{$\ddots$} & & & & \\
        R_{G_1 \rightarrow g^{2}_{n_2}} & & & & & M_{G_1 \rightarrow g^{2}_{n_2}} & & & \\ 
        R_{G_2 \rightarrow g^{1}_2} \;\,\ & & & \text{\Large 0} & & & M_{G_2 \rightarrow g^{1}_2} & &
        \\ 
        \vdots \qquad\ & & & & & & & \rotatebox[origin=c]{20}{$\ddots$} & \\
        R_{G_2 \rightarrow g^{1}_{n_1}} & & & & & & & & M_{G_2 \rightarrow g^{1}_{n_1}}
    \end{pNiceArray}. 
\end{equation*}

    We see that $X$ is block diagonal by using \cref{prop::J_2(Sigma)descriptionOfEntries,prop::J_3(T)descriptionOfEntries} to infer the zero entries as follows. Generally, the zero entries arise either due to the sink of an edge not matching the indices of the considered row or due to the missing treks between the two connected components.

    \textbf{Case 1:} Zeros due to sinks not matching row indices.

    First, in order for there to be a potential non-zero entry at position $(lm), \alpha \rightarrow \beta$ or $(klm), \alpha \rightarrow \beta$ in this matrix, $\beta$ needs to equal either $l$ or $m$, or $k,l$ or $m$, respectively. 
    This accounts for the zeros between rows $R_{G_1}$ and columns with sink in $G_2$, namely columns in $C_{G_2}$ and $C_{G_1 \rightarrow g^{2}_{j}}$. 
    The same argument holds for rows $R_{G_2}$ and columns in $C_{G_1}$ and $C_{G_2 \rightarrow g^{1}_{j}}$.

    Now consider the columns $C_{G_1 \to g_j^{2}}$ for $j=1,\dots,n_2$. For fixed $j$, entries in $C_{G_1 \to g_j^{2}}$ can only be non-zero if they correspond to $R_{G_1 \to g_j^{2}}$, indicated as blocks $M_{G_1 \to g_j^{2}}$ in the matrix. For all $i \neq j$, however, the entries of $C_{G_1 \to g_j^{2}}$ in $R_{G_1 \to g_i^{2}}$ are zero, since $g_j^{2} \notin V_1 \cup \{g^2_i\}$. 
    The same argument holds, if we flip the indices 1 and 2. 
    
    \textbf{Case 2:} Zeros due to missing treks.

    We are now left with the cases where the sink $\beta$ of an edge $\alpha \to \beta$ might match one of the indices $(lm)$ or $(klm)$. Here we use the existence (or non-existence) of treks to deduce the zero pattern. Specifically, we use the characterization of the entries of $J^G_2(S)$ and $J^G_3(T)$ based on treks where the length of the legs differ by one. If there is no match between $\beta$ and $(lm)$ or $(klm)$, the first case applies.

    Assume without loss of generality that $\beta = m$. For the corresponding entry in $J^G_2(S)$ to be non-zero there needs to exist a base trek between $l$ and $\alpha$ (we can use self-loops to obtain treks where the length of one of the legs is off by one). Since the two connected components are disconnected in $G'$, this is only possible if $l$ and $\alpha$ are in the same connected component. Similarly, $J^G_3(T)$ can only be non-zero if, again without loss of generality, $\beta = m$ and there exists a base trek between $k$, $l$, and $\alpha$. Thus, $k$,~$l$, and $\alpha$ all need to belong to the same connected component. 

    Consequently, the remaining entries in the columns $C_{G_2 \to g^1_j}$ for each $j=1,\dots,n_1$ are zero in the rows $R_{G_1}$, $R_{G_1 \to g_1^2}, \dots,R_{G_1 \to g_{n_2}^2}$, since $\alpha \in V_2$ is never in the same connected component as $k,l \in V_1$. 
    The same argument holds for each $j=1 ,\dots,n_2$ for the remaining entries of the columns $C_{G_1 \to g^2_j}$ in the rows $R_{G_2}$ and $R_{G_2 \to g_1^1}, \dots,R_{G_2 \to g_{n_1}^1}$.

    The remaining entries to consider are in the columns in $C_{G_1}$ and $C_{G_2}$. 
    In the columns in $C_{G_1}$, again the entries in $R_{G_2 \to g_1^1}, \dots,R_{G_2 \to g_{n_1}^1}$ are zero, since $\alpha \in V_1$ and $k,l \in V_2$. Similarly, in the columns in $C_{G_2}$, the entries in $R_{G_1 \to g_1^2}, \dots,R_{G_1 \to g_{n_2}^2}$ are zero, since $\alpha \in V_2$ and $k,l \in V_1$.
    Further, for the columns $C_{G_1}$ and the rows $R_{G_1 \rightarrow g_j^{2}}$, the sink $\beta \in V_1$ will always coincide with at least one index of the rows. However, since one of the indices in $(klm)$ will always be $g_j^{2}$, there would have to be a trek between $\beta, g_j^{2}$ and another element in $G_1$ for the entry to be non-zero, which is not the case. For all other indices, the first case again applies. The argument is analogous for the columns in  $C_{G_2}$.

    To show that $X$  generically has full rank, it suffices to show that each block on the diagonal generically has full rank. The matrices $\mathcal{J}^{G_1}$ and  $\mathcal{J}^{G_2}$ generically have full rank by the induction hypothesis. Employing \cref{lem::FullrankMmatrix}, we further obtain that that each of the $M$ matrices generically has full rank as well. Since the generic choices ensuring the $M$ matrices have generically full rank do not depend on $G$, the submatrix $\mathcal{J}^G_{\mathrm{off}}(S,T)$ will also have full rank. 
\end{proof}


\section{Defining Equations} \label{sec::Equations}
In this section, we explore the vanishing ideal of the discrete Lyapunov model for specific families of graphs. First, in \cref{sec::equations_polytree_single_source}, we consider the case where the graph $G$ is a polytree with a single source and only one self-loop. In this setting, the vanishing ideal is toric. In \cref{sec::VanishingDeterminants}, we present several determinantal results which give rise to polynomials in the vanishing ideal for other types of graphs. In \cref{sec:BirationalImplicitization} we explore a different approach for obtaining polynomials in the vanishing ideal.

\subsection{Polytrees With One Source} \label{sec::equations_polytree_single_source}
In this subsection, we focus on the case of $G$ being a polytree with one source node (i.e., a directed tree) and a single self-loop at the source. For such a graph, the vanishing ideal $\mathcal{I}^{\leq n}(G)$ is toric for any $n \geq 2$. The model admits a monomial parametrization given by the \textit{shortest} equitreks. In a  directred tree with a single self-loop at the source, there is a unique shortest equitrek between any set of nodes.

Let $G=(V, E)$ be such a polytree, where $V = \{0\}\cup [p-1]$ and 0 is the unique source node. For the $n$th order model, we introduce new parameters $v_i^{(2)}, v_i^{(3)}, \ldots, v_i^{(n)}$ for each $i\in V$, such that
$$ v_i^{(2)} = w_0^{(2)}\frac{\left(a^{\lambda_i}\right)^2}{1-a_{00}^2}, \qquad v_i^{(3)} = w_0^{(3)}\frac{\left(a^{\lambda_i}\right)^3}{1-a_{00}^3}, \qquad \ldots \qquad v_i^{(n)} = w_0^{(n)}\frac{\left(a^{\lambda_i}\right)^n}{1-a_{00}^n},$$
where $\lambda_i$ is the unique shortest path from the source to vertex $i$, and $a^{\lambda_i}$ is the corresponding path polynomial. This invertible change of parameters allows us to derive the \emph{shortest equitrek rule}. 

For vertices $i_1, \ldots, i_l$, let $\tau(i_1, \ldots, i_l)$ denote the unique shortest $l$-equitrek between $i_1, \ldots, i_l$ consisting of paths $\tau_1, \ldots, \tau_l$. Let $top(\tau)$ denote the top vertex of the trek $\tau$. Using the new $v$ parameters, we can write the second- and third-order cumulants as
$$s_{i_1,i_2} = v^{(2)}_{top(\tau(i_1,i_2))}a^{\tau_1}a^{\tau_2}, \qquad t_{i_1,i_2, i_3} = v^{(3)}_{top(\tau(i_1,i_2, i_3))}a^{\tau_1}a^{\tau_2}a^{\tau_3}.$$
More generally,
$$(T_n)_{i_1, \ldots, i_n} = v^{(n)}_{top(\tau(i_1, \ldots, i_n))}a^{\tau_1}\ldots a^{\tau_n}.$$

Notice that the new parametrization is monomial, so the Zariski closure of its image is a toric variety~\cite{sullivant2018algebraic}. We can determine the associated parametrization matrix $P$. Its rows are indexed by the parameters and its columns are indexed by the cumulants of the model.

\begin{example}
    Let $G$ be the graph in \cref{fig::Polytree with on source}.
\begin{figure}[H]
    \centering
     \begin{tikzpicture}
    \node[circle, draw, minimum size=0.5cm] (0) at (0,0) {0};
    \node[circle, draw, minimum size=0.5cm] (1) at (0,-1) {1};
    \node[circle, draw, minimum size=0.5cm] (2) at (1,-2) {2};
    \node[circle, draw, minimum size=0.5cm] (3) at (-1,-2) {3};
    \draw[->,  >=stealth] (0) edge[loop above] (0);
    \draw[->,  >=stealth] (0) edge (1);
    \draw[->,  >=stealth] (1) edge (2);
    \draw[->,  >=stealth] (1) edge (3);
\end{tikzpicture} 
    \caption{A polytree on four vertices with a single self-loop at the source.}
    \label{fig::Polytree with on source}
\end{figure}
For the third-order cumulant model, the parameters are $v^{(2)}_0, \ldots, v^{(2)}_3$, $v^{(3)}_0, \ldots, v^{(3)}_3$ and $a_{00}, \ldots, a_{31}$.  
The associated parametrization matrix $P$ is given by
\begin{equation*}
\scalebox{0.9}{$\begin{pmatrix}
1 & 1 & 1 & 1 & 0 & 1 & 1 & 0 & 1 & 0 & 0 & 0 & 0 & 0 & 0 & 0 & 0 & 0 & 0 & 0 & 0 & 0 & 0 & 0 & 0 & 0 & 0 & 0 & 0 & 0\\
0 & 0 & 0 & 0 & 1 & 0 & 0 & 0 & 1 & 0 & 0 & 0 & 0 & 0 & 0 & 0 & 0 & 0 & 0 & 0 & 0 & 0 & 0 & 0 & 0 & 0 & 0 & 0 & 0 & 0\\
0 & 0 & 0 & 0 & 0 & 0 & 0 & 1 & 0 & 0 & 0 & 0 & 0 & 0 & 0 & 0 & 0 & 0 & 0 & 0 & 0 & 0 & 0 & 0 & 0 & 0 & 0 & 0 & 0 & 0\\
0 & 0 & 0 & 0 & 0 & 0 & 0 & 0 & 0 & 1 & 0 & 0 & 0 & 0 & 0 & 0 & 0 & 0 & 0 & 0 & 0 & 0 & 0 & 0 & 0 & 0 & 0 & 0 & 0 & 0\\
0 & 0 & 0 & 0 & 0 & 0 & 0 & 0 & 0 & 0 & 1 & 1 & 1 & 1 & 1 & 1 & 1 & 1 & 1 & 1 & 0 & 1 & 1 & 1 & 1 & 1 & 0 & 0 & 0 & 0\\
0 & 0 & 0 & 0 & 0 & 0 & 0 & 0 & 0 & 0 & 0 & 0 & 0 & 0 & 0 & 0 & 0 & 0 & 0 & 0 & 1 & 0 & 0 & 0 & 0 & 0 & 0 & 1 & 1 & 0\\
0 & 0 & 0 & 0 & 0 & 0 & 0 & 0 & 0 & 0 & 0 & 0 & 0 & 0 & 0 & 0 & 0 & 0 & 0 & 0 & 0 & 0 & 0 & 0 & 0 & 0 & 1 & 0 & 0 & 0\\
0 & 0 & 0 & 0 & 0 & 0 & 0 & 0 & 0 & 0 & 0 & 0 & 0 & 0 & 0 & 0 & 0 & 0 & 0 & 0 & 0 & 0 & 0 & 0 & 0 & 0 & 0 & 0 & 0 & 1\\
0 & 1 & 2 & 2 & 0 & 1 & 1 & 0 & 0 & 0 & 0 & 2 & 4 & 4 & 1 & 3 & 3 & 2 & 2 & 2 & 0 & 2 & 2 & 1 & 1 & 1 & 0 & 0 & 0 & 0\\
0 & 1 & 1 & 1 & 0 & 2 & 2 & 0 & 0 & 0 & 0 & 1 & 1 & 1 & 2 & 2 & 2 & 2 & 2 & 2 & 0 & 3 & 3 & 3 & 3 & 3 & 0 & 0 & 0 & 0\\
0 & 0 & 1 & 0 & 0 & 1 & 0 & 0 & 1 & 0 & 0 & 0 & 1 & 0 & 0 & 1 & 0 & 2 & 1 & 0 & 0 & 1 & 0 & 2 & 1 & 0 & 0 & 2 & 1 & 0\\
0 & 0 & 0 & 1 & 0 & 0 & 1 & 0 & 1 & 0 & 0 & 0 & 0 & 1 & 0 & 0 & 1 & 0 & 1 & 2 & 0 & 0 & 1 & 0 & 1 & 2 & 0 & 1 & 2 & 0
\end{pmatrix}$}.
\end{equation*}
The associated toric ideal lives in a polynomial ring with 30 variables $s_{ij}$ and $t_{i j k}$ for $i, j, k\in \{0, 1, 2, 3\}$, and has degree $1374$. Its Gröbner basis contains 640 generators.
\end{example}

Even though the ideal itself is intractable, we can still recover certain information about the graph from the toric ideal. 

\begin{definition}
 Define the \emph{level} of a vertex $v$ in a directed tree to be the length of the shortest path from the source node to $v$. For each $\ell \geq 0$, define $V_\ell(G)$ to be the set of all vertices in $G$ at level $\ell$.
\end{definition}

The following result shows how the levels of vertices can be recovered from the vanishing ideal. 

\begin{proposition}
\label{prop:: levels recovery}
Let $G=(V, E)$ be a directed tree with a single self-loop at the source, where $V=\{0\}\cup [p-1]$, and let $\mathcal I^{\leq n}(G)$ be the vanishing ideal of the $n$th order cumulant model associated to $G$ for $n\geq 3$. Then,
\begin{enumerate}
    \item For $i\in V$, $s_{ij}^3t_{iii}^2 - s_{ii}^3t_{iij}t_{ijj}\in \mathcal{I}^{\leq n}(G)$ for all $j\in V$ if and only if $i$ is the source node. 
    \item Let 0 be the source node. Then, for $i, j\in V$, $s_{0i}t_{00j}-s_{0j}t_{00i} \in \mathcal{I}^{\leq n}(G)$ if and only if $i$ and $j$ are on the same level.
    \item For $i, j\in V$ on different levels, $s_{ij}t_{00j}-s_{0j}t_{0ij} \in \mathcal{I}^{\leq n}(G)$ if and only if the level of $j$ is greater than level of $i$ (i.e., $j$ is further from the source node than $i$). 
\end{enumerate}
\end{proposition}
The proof can be found in \cref{appendix::Defining equations}.

\begin{example} \label{ex::ExtraPolynomial}
    For the polytree in \cref{fig:: polytree on five nodes}, the polynomial $s_{02}s_{34}t_{224}-s_{03}s_{22}t_{244}$ lies in the vanishing ideal. Note that this polynomial does not satisfy any of the conditions from \cref{prop:: levels recovery}. Here, vertex 2 is the top of the equitrek between vertices 3 and 4.

    \begin{figure}[H]
        \centering
            \begin{tikzpicture}
            \node[circle, draw, minimum size=0.5cm] (0) at (0,0) {0};
            \node[circle, draw, minimum size=0.5cm] (1) at (-1,-1) {1};
            \node[circle, draw, minimum size=0.5cm] (2) at (1,-1) {2};
            \node[circle, draw, minimum size=0.5cm] (3) at (0, -2) {3};
            \node[circle, draw, minimum size=0.5cm] (4) at (2, -2) {4};

            \draw[->,  >=stealth] (0) edge[loop above] (0);
            \draw[->,  >=stealth] (0) edge (1);
            \draw[->,  >=stealth] (0) edge (2);
            \draw[->,  >=stealth] (2) edge (3);
            \draw[->,  >=stealth] (2) edge (4);
        \end{tikzpicture}
        \caption{A polytree on five vertices with a single self-loop at the source.}
        \label{fig:: polytree on five nodes}
    \end{figure}
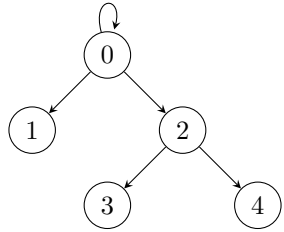

\end{example}

The extra polynomial found in \cref{ex::ExtraPolynomial}, corresponding to a trek which does not have a top node equal to the source, is an instance of a more general result, as described in \cref{prop: top trek recovery} below.

\begin{proposition}
\label{prop: top trek recovery}
    Let $G=(V, E)$ be a directed tree with a single self-loop at the source, where $V=\{0\}\cup [p-1]$, and let $I^{\leq n}(G)$ be the vanishing ideal of the $k$th order cumulant model associated to~$G$ for $n\geq 3$. Let $i, j, l\in V$. Then,
    $$s_{0l}s_{ij}t_{llj}-s_{0i}s_{ll}t_{ljj}\in \mathcal I^{\leq n}(G)$$
    if and only if $l$ is the top of the shortest equitrek between $i$ and $j$.
\end{proposition}

Consider two directed trees $G$ and $H$ with self-loops at their respective sources and the same number of nodes. It follows from \cref{prop:: levels recovery} and \cref{prop: top trek recovery} that if the levels of the graphs differ, or if there are two vertices for which the tops of the shortest equitreks are different, then the ideals are different. It turns out that these are also necessary conditions for the ideals to be different.

\begin{theorem}\label{thm:equivalent_trees}
    Two directed trees $G$ and $H$ on the same vertex set with self-loops at their respective sources define the same ideal if and only if 
    \begin{enumerate}
    \item[(i)] for all $\ell \geq 0$, the vertices at level $\ell$ in $G$ are the same as the vertices at level $\ell$ in $H$, i.e. $V_\ell(G) = V_\ell(H)$, and 
    \item[(ii)] for each pair of vertices, the tops of the shortest equitreks between them are the same in $G$ and~$H$.
    \end{enumerate}
\end{theorem}

This result will follow from \cref{lem: swapping} and \cref{lem: row equiv}, both of which are proven in \cref{appendix::Defining equations}. The main idea is that if the graphs $G$ and $H$ satisfy the conditions of \cref{thm:equivalent_trees}, then their corresponding parametrization matrices are row-equivalent, which implies that the vanishing ideal coincide.

\begin{lemma}
\label{lem: swapping}
    Let $G = (V, E)$ be a directed tree with $V= \{0\}\cup [p-1]$, where $0$ is the source with a self-loop. Let $i$ and $j$ be two vertices from the same level, each having at most one outgoing edge. Suppose there is no vertex $i'$ such that the shortest equitrek between $i$ and $i'$ or between $j$ and $i'$ is shorter than the equitrek between $i$ and $j$. Let $H$ be the graph constructed by swapping $i$ and $j$ in $G$. Then, for any order $n$ of cumulants, the vanishing ideals coincide: $\mathcal{I}^{\leq n}(H)=\mathcal{I}^{\leq n}(G)$.
\end{lemma}

\begin{lemma}
\label{lem: row equiv}
    Let $G$ and $H$ be directed trees with self-loops at their respective sources, and with the same number of vertices $p$, labeled from 0 to $p-1$. Then $H$ can be constructed from $G$ using the swapping operations from~\cref{lem: swapping} if and only if the following conditions hold:
    \begin{enumerate}
        \item For all $\ell \geq 0$, the vertices at level $\ell$ in $G$ and $H$ coincide, i.e., $V_\ell(G) = V_\ell(H)$.
        \item For each pair of vertices $i$ and $j$, the top vertices of the shortest equitreks are the same in $G$ and in~$H$. (Equivalently, the first $p$ rows of the base trek parametrization matrix are the same.)
    \end{enumerate}
\end{lemma}

\cref{lem: row equiv} immediately implies~\cref{thm:equivalent_trees} and constructively shows that any two graphs satisfying the conditions in the theorem have row-equivalent parametrization matrices.

\subsection{Defining Polynomials via the Vanishing of Determinants} \label{sec::VanishingDeterminants}
In this section, we consider the problem of finding elements of the vanishing ideal of the discrete Lyapunov model in cases where the defining graph is not a polytree with a single source. We do this by showing that certain graphical structures imply the vanishing of certain determinants. We further show that this is not a complete characterization and it remains a topic for future research.

We now show that certain matrices have to drop rank under certain constraints on the parents and ancestors of a set of vertices. The proofs are given in \cref{appendix::Defining equations}. To state these results, we define the set of parents of a subset $U = \{v_1,\dots,v_k\} \subset V$ by
\[
\mathrm{pa}(U) = \bigcup_{i=1}^k \mathrm{pa}(v_i).
\]

\begin{proposition}
\label{prop::Determinant_Parents}
    Let $G=(V, E)$ be a graph on $p$ vertices, and let $U= \{v_1, \ldots, v_k\}$ be a subset of~$V$. 
    Denote by $S'$ the $(p -k)\times k$ submatrix of $S$, formed by columns $v_1, \ldots, v_k$ and not containing any diagonal entries. Then,
    $$rk(S')\leq |pa(U)|.$$
The same result holds for all slices of the tensor $T$. Here, for a fixed $i\in V$, a slice $T_i$ is a $p\times p$ matrix $T_{ijk}$ for $0\leq j, k\leq p-1$. 
\end{proposition}

Furthermore, the statement of~\cref{prop::Determinant_Parents} holds for the matrix formed by gluing $S$ and all slices of $T$ below each other, forming a $p\times (p+p^2)$ matrix.

\begin{proposition}
\label{prop::Determinant_Parents_full}
     Let $G=(V, E)$ be a graph on $p$ vertices, and let $U= \{v_1, \ldots, v_k\}$ be a subset of~$V$. 
     Form a $p(p+1)\times p$ matrix by stacking $S$, and $T_i$, $i\in \{0\}\cup [p-1]$. Denote by $Q$ the submatrix formed by the columns $v_1, \ldots, v_k$ and not containing any diagonal entries $s_{ii}$ or $t_{iii}$. Then, 
     $$rk(Q)\leq |pa(U)|.$$
\end{proposition}

Note that this is a generalization of the previous proposition since $S'$ is a submatrix of $Q$. It follows that if we have a subset of vertices whose total number of parents is small, we can obtain polynomials in the vanishing ideal of the model by taking determinants of submatrices of the matrix $Q$.
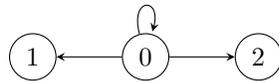
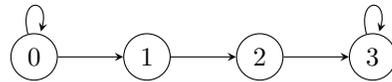
\begin{figure}[H]
    \centering
    \begin{subfigure}[t]{0.4\textwidth}
    \centering
    \begin{tikzpicture}
  \node[circle, draw, minimum size=0.5cm] (0) at (0,0) {0};
  \node[circle, draw, minimum size=0.5cm] (1) at (-1.5,0) {1};
  \node[circle, draw, minimum size=0.5cm] (2) at (1.5,0) {2};
  
  \draw[->,  >=stealth] (0) edge[loop above] (0);
  \draw[->,  >=stealth] (0) edge (1);
  \draw[->,  >=stealth] (0) edge (2);
\end{tikzpicture}
    \caption{Vertices 1 and 2 have one parent.}
    \label{fig:: 2 children}
    \end{subfigure}%
    \hspace{1cm}
    ~ 
    \begin{subfigure}[t]{0.4\textwidth}
    \centering
    \begin{tikzpicture}
  \node[circle, draw, minimum size=0.5cm] (0) at (0,0) {0};
  \node[circle, draw, minimum size=0.5cm] (1) at (1.5,0) {1};
  \node[circle, draw, minimum size=0.5cm] (2) at (3,0) {2};
  \node[circle, draw, minimum size=0.5cm] (3) at (4.5,0) {3};
  
  \draw[->,  >=stealth] (0) edge[loop above] (0);
  \draw[->,  >=stealth] (3) edge[loop above] (3);
  \draw[->,  >=stealth] (0) edge (1);
  \draw[->,  >=stealth] (1) edge (2);
  \draw[->,  >=stealth] (2) edge (3);
\end{tikzpicture} 
    \caption{Vertices 1 and 2 have one grandparent.} \label{fig:: 2 grandchildren} 
    \end{subfigure}
    \caption{Graphs with vanishing determinants.}
\end{figure}

\begin{example}
   Let $G$ be the graph in \cref{fig:: 2 children}, and let $U=\{1, 2\}$. Note that $|pa(U)|=1$. 
Consider the matrix $Q$ as defined in \cref{prop::Determinant_Parents_full}: 
\begin{equation*}
    Q^t = \begin{pmatrix}
        s_{01} & t_{001} & t_{011} & t_{012} & t_{011} & t_{112} & t_{012} & t_{112} \\
        s_{02} & t_{002} & t_{012} & t_{022} & t_{012} & t_{122} & t_{022} & t_{012}
    \end{pmatrix}.
\end{equation*}
Observe that some columns of the matrix above repeat due to the symmetries of the cumulant tensor. By the proposition, all $2\times2$ minors of $Q$ must vanish. However, these polynomials alone do not generate the whole ideal $\mathcal{I}^{\leq 3}_G$. 
\end{example}

\begin{proposition}
\label{prop::Determinant_Ancestors}
    Let $G=(V, E)$ be a graph on $p$ vertices, and let $U= \{v_1, \ldots, v_k\}$ be a subset of~$V$. Let $an_2(U)$ be the set of all grandparents of $U$, and $sib(U)$ to be the set of siblings: $$an_2(U)=\bigcup_{w\in pa(U)}pa(w), \quad sib(U) = \{w\in V: pa(w)\in pa(U)\}.$$
    Denote by $S'$ the $(p-|sib(U)|)\times k$ submatrix of $S$, formed by taking columns $v_1, \ldots, v_k$ and removing all rows containing entries $s_{ij}$ with $i, j\in sib(U)$. Then,
    $$rk(S')\leq |an_2(U)|.$$
    The same statement holds for all slices of the tensor $T$ as well. 
    Moreover, the statement holds for a matrix, formed by gluing $S$ and all slices of $T$ below each other (to form a $p\times (p+p^2)$ matrix).
\end{proposition}

\begin{example}
   Let $G$ be the graph in \cref{fig:: 2 grandchildren}, and let $U = \{1, 2\}$. Here, $|an_2(U)|=1$, and $|sib(U)|=2$. For the slice $T_3$, consider the submatrix $T_3'$ introduced in \cref{prop::Determinant_Ancestors}:
\begin{equation*}
    T_3' = \begin{pmatrix}
        t_{013} & t_{023}\\
        t_{113} & t_{123} \\
        t_{123} & t_{223} \\
        t_{133} & t_{233}
    \end{pmatrix}.
\end{equation*}
In this case, the matrix $T_3'$ has the same number of rows as $T_3$, as none of them need to be deleted. By \cref{prop::Determinant_Ancestors}, all $2\times 2$ minors of $T_3'$ vanish.
\end{example}

Finally, we show that the problem of computing the vanishing ideal of the model can be simplified by removing entries of the cumulants  indexed only by a sink. Here we call a vertex $v$ a \textit{sink} if it has no outgoing edges directed to other vertices.
\begin{proposition}
    Let $G = (V,E)$ be a directed graph, and consider the third-order vanishing ideal $\mathcal I^{\leq 3}_G \subseteq \mathbb{R}[s_{ij}, t_{ijk}\mid i, j, k\in V]$ of the corresponding discrete Lyapunov model. If $v$ is a sink in $G$, then there exists a generating set of $\mathcal I^{\leq 3}_G$ such that the variables $s_{vv}$ and $t_{vvv}$ do not appear in any generator.
\end{proposition}
\begin{proof}
    Suppose that $f=\sum_{t=0}^k s_{vv}^tf_t$ belongs to a generating set, and assume that no terms of $f_t$ can be divided by $s_{vv}$. Using the trek rule, $f$ can be written as an element of $\mathbb{R}(a_{ij}, \omega^{(2)}_i, \omega^{(3)}_i,\, i,j\in V)$. Now assign arbitrary values to $a_{ij}$ and $\omega^{(3)}_i$ for all $i,j\in V$ and to $\omega^{(2)}_i$ for all $i\in V\setminus \{v\}$ from the parameter space $\Theta$. Denote the resulting polynomial as $g\in \mathbb{R}[\omega^{(2)}_v]$. Note that $g$ has degree $k$. Since $f\in \mathcal I^{\leq 3}_G$, it must vanish for all assignments of the variables $a_{ij}, \omega^{(2)}_i$, $\omega^{(3)}_i$. Thus, $g(\omega^{(2)}_v)=0$ for all $\omega^{(2)}_v\in \mathbb{R}$. It follows that $g=0$, and hence all its coefficients are 0. The coefficients are evaluations of $f_t$ or $\frac{f_t}{(1-a_{vv})^t}$, depending on whether there is a self-loop at $v$. Thus, $f_t \in \mathcal I^{\leq 3}_G$, and we can add them to the generating set instead of $f$. We can perform a similar process to get generators not involving~$t_{vvv}$.
\end{proof}

\begin{corollary}
    The elimination of $s_{vv}$ and $t_{vvv}$ does not change the vanishing ideal corresponding to $G$ in case $v$ is a sink of $G$.
\end{corollary}

\subsection{Defining Polynomials via Birational Implicitization} \label{sec:BirationalImplicitization}

In \cite{boege2024real}, the authors develop a strategy for efficiently computing the polynomials and semialgebraic constraints that define a statistical model under strict assumptions on its parametrization. In particular, their approach requires that the parametrization be \emph{ambirational}. 

\begin{definition}
    Let $\alpha: \R^m \rightarrow \R^n$ be a rational function, $\Theta \subset \R^m$ a parameter space, and $\Mcal = \alpha(\Theta)$ a parametric statistical model. If $\alpha$ admits a rational inverse $\beta: \R^n \dashrightarrow \R^m$, then $\Mcal$ is \emph{ambirational}.
\end{definition}

In particular, this means that $\Theta$ and $\Mcal$ are birationally equivalent, and the rational functions realizing this equivalence are also isomorphisms of their respective ambient spaces. In this setting, one can apply \cite[Theorem~3.10]{boege2024real} to use the polynomial constraints on the parameter space $\Theta$ to derive polynomial constraints on $\Mcal$. We have shown in  \cref{prop:IdAllSelfLoops} that when $G$ is a DAG with all self-loops, the map from the parameter space consisting of the matrix $A$ and cumulants $K_2, K_3$ and $K_4$ to the space of 2nd, 3rd and 4th order cumulants admits a rational inverse. We have not shown that this map is ambirational; indeed,  \cref{prop:IdAllSelfLoops} shows that for a DAG $G$ with all the self loops, for generic cumulants in $\Mcal_G^{\leq 4}$ the parameters $A$ and $\Omega^{(m)}$ can be found as a rational function of the cumulants. This does not rule out that possibility that the map from parameter space to the cumulants is somewhere many-to-one; in general, these models are generically identifiable but not globally identifiable. 
Nevertheless, we can attempt to apply the methods of \cite{boege2024real} to compute candidates for defining polynomials of the model without the guarantee that the resulting polynomials will actually belong to the vanishing ideal.

\begin{figure}
\centering
    \begin{tikzpicture}
  \node[circle, draw, minimum size=0.5cm] (0) at (-1.5,0) {0};
  \node[circle, draw, minimum size=0.5cm] (1) at (0,0) {1};
  \node[circle, draw, minimum size=0.5cm] (2) at (1.5,0) {2};
  
  \draw[->,  >=stealth] (0) edge[loop above] (0);
  \draw[->,  >=stealth] (1) edge[loop above] (1);
  \draw[->,  >=stealth] (2) edge[loop above] (2);
  \draw[->,  >=stealth] (0) edge (1);
  \draw[->,  >=stealth] (1) edge (2);
\end{tikzpicture} \caption{A directed path with all self-loops}\label{fig:BirationalImplicitization}
\end{figure}
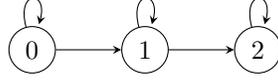
Consider the directed path $G$ with  all self-loops pictured in \cref{fig:BirationalImplicitization}. Given generic second-, third- and fourth-order cumulants in $\Mcal_G^{\leq 4}$,  \cref{prop:IdAllSelfLoops} allows us to solve for the adjacency matrix $A$ as a rational function of these cumulants. This also gives an expression for $a_{20}$ in terms of the cumulants; but since $A$ is a weighted adjacency matrix of $G$, the $(2,0)$ entry of $A$ should be zero. If the map which sends $\Mcal_G^{\leq 4}$ to the matrix $A$ were ambirational, then by \cite[Theorem~3.10]{boege2024real}, the numerator of $a_{20}$ as a function of the second-, third- and fourth-order cumulants belongs to the vanishing ideal $\mathcal{I}^{\leq 4}(G).$ Using Macaulay2, we compute that this numerator has two irreducible factors, $a_{20} = f_1f_2$. If we assume that $a_{20} \in \mathcal{I}^{\leq 4}(G)$, then since this ideal is prime, one of $f_1$ or $f_2$ must belong to $\mathcal{I}^{\leq 4}(G).$ By generating random cumulants in $\Mcal^{\leq 4}_G$ and substituting them into these polynomials, we find that the one that vanishes on these random cumulants is the following polynomial, $f_1$: 

\small{
\begin{align*}
&s_{00}s_{01}^3s_{02}s_{11}t_{001}^4t_{002}r_{0000}^3&-&
s_{00}s_{01}^4s_{12}t_{001}^4t_{002}r_{0000}^3&-&
s_{01}^4s_{02}^2t_{000}t_{001}^4r_{0000}^2r_{0001} & \\
+&s_{00}s_{01}^3s_{02}s_{12}t_{000}t_{001}^4r_{0000}^2r_{0001}&+&
s_{00}s_{01}^3s_{02}^2t_{001}^5r_{0000}^2r_{0001}&-&
s_{00}^2s_{01}^2s_{02}s_{12}t_{001}^5r_{0000}^2r_{0001}& \\
-&s_{01}^4s_{02}s_{11}t_{000}^2t_{001}^2t_{002}r_{0000}^2r_{0001}&+&
s_{01}^5s_{12}t_{000}^2t_{001}^2t_{002}r_{0000}^2r_{0001}&+&
s_{01}^5s_{02}t_{000}t_{001}^3t_{002}r_{0000}^2r_{0001}& \\
-&2s_{00}s_{01}^3s_{02}s_{11}t_{000}t_{001}^3t_{002}r_{0000}^2r_{0001}&+&
s_{00}s_{01}^4s_{12}t_{000}t_{001}^3t_{002}r_{0000}^2r_{0001}&-&
s_{00}s_{01}^4s_{02}t_{001}^4t_{002}r_{0000}^2r_{0001}& \\
+&s_{00}^2s_{01}^3s_{12}t_{001}^4t_{002}r_{0000}^2r_{0001}&-&
s_{01}^4s_{02}s_{12}t_{000}^3t_{001}^2r_{0000}r_{0001}^2&+&
2s_{01}^4s_{02}^2t_{000}^2t_{001}^3r_{0000}r_{0001}^2& \\
+&s_{00}s_{01}^3s_{02}s_{12}t_{000}^2t_{001}^3r_{0000}r_{0001}^2&-&
2\,s_{00}s_{01}^3s_{02}^2t_{000}t_{001}^4r_{0000}r_{0001}^2&-&
s_{00}^2s_{01}^2s_{02}s_{12}t_{000}t_{001}^4r_{0000}r_{0001}^2&\\
+&s_{00}^3s_{01}s_{02}s_{12}t_{001}^5r_{0000}r_{0001}^2&+&
2s_{01}^4s_{02}s_{11}t_{000}^3t_{001}t_{002}r_{0000}r_{0001}^2&-&
s_{01}^5s_{12}t_{000}^3t_{001}t_{002}r_{0000}r_{0001}^2&\\
-&2s_{01}^5s_{02}t_{000}^2t_{001}^2t_{002}r_{0000}r_{0001}^2&+&
s_{00}s_{01}^3s_{02}s_{11}t_{000}^2t_{001}^2t_{002}r_{0000}r_{0001}^2&-&
s_{00}s_{01}^4s_{12}t_{000}^2t_{001}^2t_{002}r_{0000}r_{0001}^2& \\
+&2s_{00}s_{01}^4s_{02}t_{000}t_{001}^3t_{002}r_{0000}r_{0001}^2&+&
s_{00}^2s_{01}^3s_{12}t_{000}t_{001}^3t_{002}r_{0000}r_{0001}^2&-&
s_{00}^3s_{01}^2s_{12}t_{001}^4t_{002}r_{0000}r_{0001}^2& \\
-&s_{01}^4s_{02}^2t_{000}^3t_{001}^2r_{0001}^3&+&
2s_{00}s_{01}^3s_{02}s_{12}t_{000}^3t_{001}^2r_{0001}^3&+&
s_{00}s_{01}^3s_{02}^2t_{000}^2t_{001}^3r_{0001}^3& \\
-&4s_{00}^2s_{01}^2s_{02}s_{12}t_{000}^2t_{001}^3r_{0001}^3&+&
3s_{00}^3s_{01}s_{02}s_{12}t_{000}t_{001}^4r_{0001}^3&-&
s_{00}^4s_{02}s_{12}t_{001}^5r_{0001}^3& \\
-&s_{01}^4s_{02}s_{11}t_{000}^4t_{002}r_{0001}^3&+&
s_{01}^5s_{12}t_{000}^4t_{002}r_{0001}^3&+&
s_{01}^5s_{02}t_{000}^3t_{001}t_{002}r_{0001}^3& \\
-&2s_{00}s_{01}^4s_{12}t_{000}^3t_{001}t_{002}r_{0001}^3&-&
s_{00}s_{01}^4s_{02}t_{000}^2t_{001}^2t_{002}r_{0001}^3&+&
4s_{00}^2s_{01}^3s_{12}t_{000}^2t_{001}^2t_{002}r_{0001}^3& \\
-&3s_{00}^3s_{01}^2s_{12}t_{000}t_{001}^3t_{002}r_{0001}^3&+&
s_{00}^4s_{01}s_{12}t_{001}^4t_{002}r_{0001}^3 &&&&
\end{align*}}

\normalsize

The fact that $f_1$ vanishes on cumulants in the model generated from random values of the entries of $A$ is strong evidence that $f_1$ belongs to the vanishing ideal of the model. Indeed, if $f_1$ were not in $\mathcal{I}^{\leq 4}(G)$, then $f_1$ would be nonzero on almost all cumulants in $\Mcal^{\leq 4}_G$. We leave both a statistical interpretation of this polynomial and a proof that it belongs to $\mathcal{I}^{\leq 4}(G)$ as directions for future research. This computation also leads us to the question: under what circumstances can the requirement of ambirationality in \cite{boege2024real} be relaxed?

\section{Discussion}
In this paper, we presented a first study of discrete Lyapunov models with non-Gaussian errors. The non-Gaussianity gives rise to non-Gaussian equilibrium distributions, allowing us to consider not only its covariance matrix but also its higher-order cumulants. We showed that the entries of the cumulants of the equilibrium distribution can be expressed combinatorially via {\em equitreks} in the graph. 

This combinatorial interpretation allowed us to derive several parameter identifiability results. In particular, we provided generic identifiability from second-, third-, and fourth-order cumulants results for DAGs with self-loops at each node (Section~\ref{sec::ParameterIdentifiabilityDAGs}). Furthermore, we showed local identifiability for all directed graphs containing a self-loop at each node and no isolated noted (Section~\ref{sec:local identifiability}). Finally, we described some of the equations characterizing the implicit description of the model in Section~\ref{sec::Equations}.

This work is only the first study of such models, and numerous questions remain open. Most notably, these models can be used in causal discovery, and, therefore, one of the main problems is to design algorithms for learning the graph from samples of the equilibrium distribution. One of the ways this could be approached is by discovering more equations that vanish on the model for each graph, and then testing such equations on the sample cumulants obtained from data. We here have only discovered some of the equations that vanish on the model for some special graphs, and it is still an open problem to characterize the full vanishing ideal of the model.

\bigskip
\textbf{Funding.} This project originated at the Workshop for Women in Algebraic Statistics, held at St John's College, Oxford from July 8 to July 18, 2024. The workshop was supported by St John's College, Oxford, the L'Oreal-UNESCO For Women in Science UK and Ireland Rising Talent Award in Mathematics and Computer Science (awarded to Jane Coons), the Heilbronn Institute for Mathematical Research, and the UKRI/EPSRC Additional Funding Programme for the Mathematical Sciences.

Elina  Robeva was also supported by a Canada CIFAR AI Chair. Cecilie Olesen Recke was also supported by Novo Nordisk Foundation Grant NNF20OC0062897. Sarah Lumpp was also supported by the European Research Council (ERC) under the European Union’s Horizon 2020 research and innovation programme (grant agreement No 883818) as well as the DAAD programme Konrad Zuse Schools of Excellence in Artificial Intelligence, sponsored by the Federal Ministry of Research, Technology and Space.

\bibliographystyle{alpha}
\bibliography{DiscreteLyapunov}

\newpage

\appendix

\section{Proofs for \texorpdfstring{\cref{sec::TrekParametrization}}{Section 3}}\label{sec:proofsTrekParametrization}

\subsection{Proof of \texorpdfstring{\cref{prop::trekrule}}{Proposition 3.2}}\label{app:Proof_Prop_3.2}

\begin{proof}
Spelling out the $l$-th mode products from  \cref{prop:nthorder}, we obtain 
\begin{align*}
    (T_n)_{i_1\ldots i_n} & = \sum_{l=0}^{\infty} \sum_{1\leq r_1,\ldots, r_n\leq p} (\Omega^{(n)})_{r_1, \ldots, r_n}(A^l)_{i_1r_1}\cdots(A^l)_{i_nr_n}\\
    & = \sum_{l=0}^{\infty} \sum_{r=1}^p w^{(n)}_{r}(A^l)_{i_1r}\cdots(A^l)_{i_nr} = \sum_{r=1}^p w^{(n)}_{r} \sum_{l=0}^{\infty}(A^l)_{i_1r}\cdots(A^l)_{i_nr},
\end{align*}
where in the second equality we used that $\Omega^{(n)}$ is diagonal.
Further note that $$(A^l)_{ij} = \sum_{P\in P_l(j, i)} a^P,$$ where $P_l(j, i)$ denotes the set of all paths from $j$ to $i$ of length $l$.
Substituting this representation yields 
\begin{align*}
    (T_n)_{i_1\ldots i_n} & = \sum_{r=1}^p w^{(n)}_{r} \sum_{l=0}^{\infty}\left( \sum_{P_1\in P_l(r, i_1)} a^P_1\right)\ldots \left(\sum_{P_n\in P_l(r, i_n)} a^P_n\right) \\
    & = \sum_{\tau = (\tau_1, \ldots, \tau_n) \in \mathcal{T}(i_1,\dots,i_n)} w^{(n)}_{top(\tau)} a^{\tau_1} \ldots a^{\tau_n}. \qedhere
\end{align*}
\end{proof}

\subsection{Proof of \texorpdfstring{\cref{prop:trek_rule_dag}}{Proposition 3.7}}\label{app:Proof_Prop_3.7}

The proof of \cref{prop:trek_rule_dag} is based on induction by employing the recursive formula for the entries of $S$. Induction then basically means adding an edge on one or both sides of the considered base trek. Thus, the main challenge is to correctly account for the added possibilities of adding self-loops along the the trek, which is expressed in the following recursive properties of the rational coefficients $C(x,y;t)$ and the polynomials $p_{x,y}(t)$. 

\begin{lemma} \label{lem:recursion_trek_rule_dag}
    Let $0 \leq x \leq y$. The polynomial $p$ and the rational function $C$ satisfy the following recursive relations:
    \begin{enumerate}
        \item[(i)] $p_{x+1,y+1}(t) = t^2 p_{x,y+1}(t) 
        + (1+(t^2-1)\delta_{xy}) p_{x+1,y}(t) 
        + (1-t^2)p_{x,y}(t)$;
        \item[(ii)] $C(x+1,y+1;t) = \frac{1}{1-t^2}\left(t(C(x,y+1;t) + C(x+1,y;t)) + C(x,y;t)\right)$. 
    \end{enumerate}
\end{lemma}

\begin{proof}
    $(i)$ We prove this using a similar strategy as in the proof of the Chu-Vandermonde identity by looking at coefficients in the expansion of the binomial formula.
    
    First observe that $p_{x,y}(t)$ is equal to the coefficient of $s^x$ in the polynomial $(1+s)^{y}(1+st^2)^x$ by writing out the appropriate binomial expansion. Thus, if we can prove the stronger claim of a corresponding recursive equality for these polynomials, it follows in particular that the coefficients of $s^x$ are equal. This strategy works only when $x < y$, so the proof is now split into the two cases $x < y$ and $x = y$. 

    \textbf{Case 1:} For $x < y$, the recursive formula simplifies to
    \begin{equation*}
        p_{x+1,y+1}(t) = t^2 p_{x,y+1}(t) 
        +  p_{x+1,y}(t) 
        + (1-t^2)p_{x,y}(t). 
    \end{equation*}
    
    By the explained correspondence, $p_{x,y+1}(t)$ is the coefficient of $s^x$ in $(1+s)^{y+1}(1+st^2)^x$, $p_{x+1,y}(t)$ is the coefficient of $s^{x+1}$ in $(1+s)^{y}(1+st^2)^{x+1}$, $p_{x,y}(t)$ is the coefficient of $s^{x}$ in $(1+s)^{y}(1+st^2)^x$, and $p_{x+1,y+1}(t)$ is the coefficient of $s^{x+1}$ in $(1+s)^{y+1}(1+st^2)^{x+1}$. Therefore, it suffices to prove the polynomial identity
    \begin{equation*}
         (1+s)^{y+1}(1+st^2)^{x+1} = s t^2  (1+s)^{y+1}(1+st^2)^x
        +  (1+s)^{y}(1+st^2)^{x+1}
        + s (1-t^2) (1+s)^{y}(1+st^2)^x.
    \end{equation*}
    Rewriting the right-hand side yields 
    \begin{align*}
        &s t^2  (1+s)^{y+1}(1+st^2)^x
        +  (1+s)^{y}(1+st^2)^{x+1}
        + s (1-t^2) (1+s)^{y}(1+st^2)^x \\ 
        &= (1+s)^{y}(1+st^2)^{x} (st^2(1+s) + (1+st^2) + s(1-t^2)) \\
        &=(1+s)^y(1+st^2)^x(1 + s + st^2 +  s^2t^2) =  (1+s)^{y+1}(1+st^2)^{x+1},
    \end{align*}
    which proves the required polynomial equality, concluding this case. 

    \textbf{Case 2:} 
    For $x = y$, we write out the definition of $p$ directly to prove the result. Since $p_{x,y}$ is symmetric in $x$ and $y$, the recursive formula to prove becomes 
    \begin{equation*}
        p_{x+1,x+1}(t) = 2 t^2 p_{x,x+1}(t) 
        + (1-t^2)p_{x,x}(t).
    \end{equation*}
    Using the definition to write out the right-hand side yields
    \begin{align*}
         2 t^2 p_{x,x+1}(t) 
        + (1-t^2)p_{x,x}(t) &=\sum_{l = 0}^{x} t^{2(l+1)} \left( 2\binom{x+1}{x-l}\binom{x}{l} -  \binom{x}{l}^2\right) + \sum_{l = 0}^{x} t^{2l}\binom{x}{l}^2 \\ 
         &= \sum_{k = 1}^{x+1} t^{2k} \left( 2\binom{x+1}{x-(k-1)}\binom{x}{k-1} - \binom{x}{k-1}^2\right) + 
         \sum_{l = 0}^{x} t^{2l}\binom{x}{l}^2 \\ 
         &= \sum_{k = 1}^{x} t^{2k} \left( 2\binom{x+1}{x-(k-1)}\binom{x}{k-1} -  \binom{x}{k-1}^2 + \binom{x}{k}^2 \right) + t^{2(x+1)} + 1, 
    \end{align*}
    by shifting the index of the first sum. Note that the last two terms outside of the sum correspond to the extreme terms in $p_{x+1,x+1}(t)$. In order to prove the recursive identity, the only thing left to compute is  
    \begin{align*}
        2\binom{x+1}{k}\binom{x}{k-1} - \binom{x}{k-1}^2 + \binom{x}{k}^2 &= \binom{x+1}{k}^2 \frac{2k}{x+1} -  \left[\binom{x+1}{k} \frac{k}{x+1}\right]^2 + \left[\binom{x+1}{k}\frac{x+1-k}{x+1}\right]^2 \\
        &= \binom{x+1}{k}^2 
        \left[\frac{2k(x+1)}{(x+1)^2} - \frac{k^2}{(x+1)^2} + \frac{(x+1-k)^2}{(x+1)^2} \right] \\
        &= \binom{x+1}{k}^2 
    \end{align*}
    to obtain the formula for $p_{x+1,x+1}(t)$.
    
    $(ii)$ Writing out the definitions on the right-hand side, we obtain
    \begin{align*}
        & \frac{1}{1-t^2}\left(t(C(x,y+1;t) + C(x+1,y;t)) + C(x,y;t)\right) \\
        &=   t^{|x-(y+1)|+1} \frac{p_{x,y+1}(t)}{(1-t^2)^{x+y+3}} + t^{|(x+1)-y|+1} \frac{p_{x+1,y}(t)}{(1-t^2)^{x+y+3}} + t^{|x-y|} \frac{(1-t^2) p_{x,y}(t)}{(1-t^2)^{x+y+3}} \\
        &= \frac{t^{|(x+1)-(y+1)|}}{(1-t^2)^{(x+1)+(y+1)+1}} 
        \left(t^2 p_{x,y+1}(t) 
        + (1+(t^2-1)\delta_{xy}) p_{x+1,y}(t) 
        + (1-t^2)p_{x,y}(t) \right) \\
        &= t^{|(x+1)-(y+1)|} \frac{p_{x+1,y+1}(t)}{(1-t^2)^{(x+1)+(y+1)+1}} = C(x+1,y+1;t).
    \end{align*}
    In the second equality, we use that $x \leq y$, so we have 
    \[
    x < y+1, \text { so } |x-(y+1)| + 1 = y + 1 - x +1 = |x-y| + 2.
    \]
    If we consider $x+1$, the power of $t$ depends on whether $x=y$ or not. In the case $x < y$, we have
    \[
    x+1 \leq y, \text { so } |(x+1)-y| + 1 = y - (x+1) + 1 = |x-y|.
    \]
    In the case $x = y$, however, we have,   
    \[
    x+1 > y, \text { so } |(x+1)-y| + 1 = x-y + 2 = |x-y| + 2,
    \]
    inducing the factor $(1+(t^2-1)\delta_{xy})$ to account for both cases in the equation.
    The result follows from the recursive property of the polynomial $p_{x,y}(t)$ shown in $(i)$.
\end{proof}

\begin{proof}[Proof of \cref{prop:trek_rule_dag}]
    Writing out the recursive formula for $s_{ij}$, we obtain
    \begin{align*}
    s_{ij} 
    & = \sum_{\substack{k \in \mathrm{pa}(i), \\ l \in \mathrm{pa}(j)}} a_{ik}a_{jl} s_{kl} + \delta_{ij} \omega_i^{(2)} 
    = t^2 s_{ij} +  \sum_{\substack{k \in \mathrm{pa}(i), \\ l \in \mathrm{pa}(j), \\ (k,l) \neq (i,j)}} a_{ik}a_{jl} s_{kl} + \delta_{ij} \omega_i^{(2)}.
    \end{align*}
    Solving for $s_{ij}$ yields
    \begin{align*}
    s_{ij}
    &= \sum_{\substack{k \in \mathrm{pa}(i), \\ l \in \mathrm{pa}(j), \\ (k,l) \neq (i,j)}} \frac{1}{1-t^2} a_{ik}a_{jl} s_{kl} + \delta_{ij} \frac{1}{1-t^2} \omega_i^{(2)} \\
    &= \sum_{l \in \mathrm{pa}(j) \setminus \{j\}} \frac{1}{1-t^2} t a_{jl} s_{il} + \sum_{k \in \mathrm{pa}(i) \setminus \{i\}} \frac{1}{1-t^2} t a_{ik} s_{kj} + \sum_{\substack{k \in \mathrm{pa}(i)\setminus \{i\}, \\ l \in \mathrm{pa}(j)\setminus \{j\}}} \frac{1}{1-t^2} a_{ik}a_{jl} s_{kl} + \delta_{ij} \frac{1}{1-t^2} \omega_i^{(2)}.
    \end{align*}
    This formula motivates an inductive proof similar to the proof of \cite[Prop. 4.3]{boege2024conditional}. We assume that \eqref{eq:trek_rule_dag} holds for $s_{il}$, $s_{kj}$, and $s_{kl}$ as the induction hypothesis. Fix a topological ordering of the nodes and the induced lexicographic ordering of the entries $s_{ij}$ of $S$, where $s_{kl} < s_{ij}$ if and only if $l < j $, or $l = j$ and $k < i$.

    The smallest $S$-entries are given by $s_{ii}$, where $i$ is a source node of $G$. In this case, we have 
    \[
    s_{ii} = \frac{1}{1-t^2} \omega_i^{(2)} = C(0,0;t) \omega_i^{(2)}, 
    \]
    which corresponds to the sum over the empty trek in \eqref{eq:trek_rule_dag}.
    Note that if $i$ and $j$ are two distinct source nodes, we have $s_{ij} = 0$ corresponding to the claim in \eqref{eq:trek_rule_dag} as well. 

    Now consider a non-empty base trek (i.e., with no cycles and self-loops) between two nodes $i$ and $l$, denoted by $\Tilde{\tau} \in \mathcal{T}^*(i,l)$. Let $k$ be the unique parent of $i$ on this trek. Then the trek $\Tilde{\tau}$ can be decomposed uniquely into the remaining trek $\tau$ between $k$ and $l$ and the edge $k \rightarrow i$. Hence, we have a bijection between $\mathcal{T}^*(i,l)$ and $\bigcup_{k \in \mathrm{pa}(i)\setminus \{i\}} \mathcal{T}^*(k,l)$. Moreover, $d(\Tilde{\tau}_i) = d(\tau_k)+1$ and $a^{\Tilde{\tau}_i} = a^{\tau_k} \cdot a_{ik}$. A similar decomposition can be applied to the path to the second leaf of the treak or to both at the same time.
    Applying the induction hypothesis together with these observations and \cref{lem:recursion_trek_rule_dag}, we obtain
    \begin{align*}
    s_{ij} 
    &= \sum_{l \in \mathrm{pa}(j) \setminus \{j\}} \sum_{\Tilde{\tau}=(\Tilde{\tau}_i,\Tilde{\tau}_l) \in \mathcal{T}^*(i,l)} \frac{1}{1-t^2} t C(d(\Tilde{\tau}_i),d(\Tilde{\tau}_l);t) a^{\Tilde{\tau}_i}a^{\Tilde{\tau}_l}a_{jl} w_{\mathrm{top}(\Tilde{\tau})}^{(2)} \\
    & \quad + \sum_{k \in \mathrm{pa}(i) \setminus \{i\}} \sum_{\Tilde{\tau}=(\Tilde{\tau}_k,\Tilde{\tau}_j) \in \mathcal{T}^*(k,j)} \frac{1}{1-t^2} t C(d(\Tilde{\tau}_k),d(\Tilde{\tau}_j);t) a^{\Tilde{\tau}_k} a_{ik} a^{\Tilde{\tau}_j} w_{\mathrm{top}(\Tilde{\tau})}^{(2)} \\
    & \quad + \sum_{\substack{k \in \mathrm{pa}(i)\setminus \{i\}, \\ l \in \mathrm{pa}(j)\setminus \{j\}}} \sum_{\tau=(\tau_k,\tau_l) \in \mathcal{T}^*(k,l)} \frac{1}{1-t^2} C(d(\tau_k),d(\tau_l);t) a^{\tau_k} a_{ik} a^{\tau_l} a_{jl} w_{\mathrm{top}(\tau)}^{(2)} + \delta_{ij} \frac{1}{1-t^2} \omega_i^{(2)} \\
    &= \sum_{\substack{k \in \mathrm{pa}(i)\setminus \{i\}, \\ l \in \mathrm{pa}(j)\setminus \{j\}}} \sum_{\tau=(\tau_k,\tau_l) \in \mathcal{T}^*(k,l)} \frac{1}{1-t^2} \left[t(C(d(\tau_k)+1,d(\tau_l);t) + C(d(\tau_k),d(\tau_l)+1;t)) + C(d(\tau_k),d(\tau_l);t) \right] \\
    & \quad \cdot a^{\tau_k} a_{ik} a^{\tau_l} a_{jl} w_{\mathrm{top}(\tau)}^{(2)} + \delta_{ij} \frac{1}{1-t^2} \omega_i^{(2)} \\
    &= \sum_{\substack{k \in \mathrm{pa}(i)\setminus \{i\}, \\ l \in \mathrm{pa}(j)\setminus \{j\}}} \sum_{\tau=(\tau_k,\tau_l) \in \mathcal{T}^*(k,l)} C(d(\tau_k)+1,d(\tau_l)+1;t) a^{\tau_k} a_{ik} a^{\tau_l} a_{jl} w_{\mathrm{top}(\tau)}^{(2)} + \delta_{ij} \frac{1}{1-t^2} \omega_i^{(2)} \\
    &= \sum_{\tau=(\tau_i,\tau_j) \in \mathcal{T}^*(i,j)} C(d(\tau_i),d(\tau_j);t) a^{\tau_i} a^{\tau_j} w_{\mathrm{top}(\tau)}^{(2)}.
    \end{align*}
    In the case $i = j$, the last term $\frac{1}{1-t^2} \omega_i^{(2)}$ accounts for the empty trek that does not involve true parents of~$i$. 
\end{proof}

\begin{remark}
    The proof of the restricted trek rule for third-order cumulant entries is analogous to the argument above. However, it requires substantially more bookkeeping of indices during the induction, as well as proving a more involved recursion for the appearing polynomials $p_{x,y,z}(t)$. 
\end{remark}

\subsection{Restricted trek rule for higher-order cumulants} \label{app::dag_trek_rule_higher_cum}

The general form of an entry of the $n$th order cumulant in the discrete Lyapunov model of a DAG is given as 
\[
(T_n)_{i_1,\dots,i_n} = \sum_{\tau = (\tau_1, \dots , \tau_n) \in \mathcal{T}^*(i_1, \dots, i_n)} C(d(\tau_1),\dots,d(\tau_n);t) a^{\tau_1} \cdots a^{\tau_n}w^{(n)}_{top(\tau)},
\]
where $C$ is the rational function in the edge lengths of all $n$ components of the considered base trek as well as the self-loop parameter $t$ given by 
\[
C(x_1,\dots,x_n;t) = t^{n\cdot\max(x_1,\dots,x_n)-\sum_{l=1}^n x_l} \frac{p_{x_1,\dots,x_n}(t)}{(1-t^n)^{\sum_{l=1}^n x_l +1}}.
\]
For instance, in the case $n=3$, the rational function is given as
\[
C(x,y,z;t) = t^{3\cdot\max(x,y,z)-(x+y+z)} \frac{p_{x,y,z}(t)}{(1-t^3)^{x+y+z+1}}.
\]
While it is clear how to compute the powers of $t$ and $(1-t^3)$, the polynomial $p_{x,y,z}(t)$ is much less tractable.

In the case $n=2$, the polynomial $p_{x,y}(t)$ seems simple, however it contains a more intricate structure that allows us to extend the definition to higher $n$. 
First, observe that in the case $x=y$, i.e., a base trek with equal length paths, the formula collapses to 
\begin{equation} \label{eq::dag_trek_polynomial_2}
    p_{x,x}(t)=\sum_{l = 0}^x t^{2l} \binom{x}{l}^2.
\end{equation}
The coefficients appearing in $p_{x,x}(t)$ are the squared entries of the $x$-th row of Pascal's triangle. The resulting integer sequence is known as \href{https://oeis.org/A008459}{OEIS A008459}. Combinatorially, this sequence admits several interpretations; for example, it counts the lattice paths from $(0, 0)$ to $(x, x)$ with steps $(1, 0)$ and $(0, 1)$, having $l$ right turns.
 
In the case $n=3$, the coefficients of the polynomial $p_{x,x,x}(t)$ are given by a known integer sequence as well, namely \href{https://oeis.org/A181544}{OEIS A181544}. This sequence can be defined via the row generating function of row~$x$:
\begin{align*}
    B(x) & = \left(\sum_{i \geq 0} \binom{x+i}{i}^3s^i\right) (1-s)^{3x+1} 
    = \sum_{l \geq 0} \left( \sum_{k = 0}^l (-1)^k \binom{3x+1}{k}\binom{x+l-k}{l-k}^3 \right) s^l.
\end{align*}
This representation allows us to infer the closed formula for each coefficient in that row, yielding
\[
p_{x,x,x}(t) 
= \sum_{l = 0}^{2x} t^{3(2x-l)} \sum_{k = 0}^l (-1)^{l-k} \binom{3x+1}{l-k}\binom{x+k}{k}^3.
\]

Interestingly, a similar formula can be derived from the row generating function for $n=2$, namely 
\[
p_{x,x}(t) = \sum_{l = 0}^{x} t^{2(x-l)} \sum_{k = 0}^l (-1)^{l-k} \binom{2x+1}{l-k}\binom{x+k}{k}^2.
\]
It can be shown that the coefficients of this polynomial coincide with the squared binomial coefficients appearing in~\eqref{eq::dag_trek_polynomial_2}. Generalizing this to the product of two different binomial coefficients reveals the underlying combinatorial pattern of the coefficients in $p_{x,y}(t)$ as 
\[
\binom{x}{l}\binom{y}{l} = \sum_{k=0}^l (-1)^{l-k}\binom{x+y+1}{l-k} \binom{x+k}{k}\binom{y+k}{k}
\]
holds (by rewriting the binomial coefficients in the sum as a part independent of $i$ and a part dependent on $i$ that can further be rewritten as an application of the hypergeometric function $\prescript{}{3}{F}_2$ -- which can be rewritten again in terms of a fraction of binomial coefficients). 
When generalizing this pattern to the case $n=3$, we obtain 
\[
p_{x,y,z}(t)= \hspace{-5pt} \sum_{l = 0}^{x+y+z-\max(x,y,z)} \hspace{-4pt} t^{3(x+y+z-\max(x,y,z)-l)} \sum_{k = 0}^l (-1)^{l-k} \binom{x+y+z+1}{l-k}\binom{x+k}{k}\binom{y+k}{k}\binom{z+k}{k}.
\]

Consequently, we conjecture that this formula generalizes to higher $n$ as
$$p_{x_1,\dots,x_n}(t) = \hspace{-8pt} \sum_{l = 0}^{\sum_{i=1}^n x_i -\max(x_1,\dots,x_n)} \hspace{-7pt} t^{n(\sum_{i=1}^n x_i -\max(x_1,\dots,x_n)-l)} \sum_{k = 0}^l (-1)^{l-k} \binom{x_1+\cdots+x_1+1}{l-k}\prod_{i=1}^n\binom{x_i+k}{k}.$$

\section{Proofs for \texorpdfstring{\cref{sec::ParameterIdentifiabilityDAGs}}{Section 4}}\label{sec:proofsParameterIdentifiabilty}
\subsection{Proof of \texorpdfstring{\cref{prop:IdAllSelfLoops}}{Theorem 4.4}}

\begin{proof}
    Let the node set $V$ of $G$ be topologically ordered as $0 \leq \cdots \leq p-1$. The proof of generic identifiability proceeds by induction. At each step $j$, we show that we can generically identify all parameters $a$ corresponding to edges going into node $j$, i.e., $a_{ji}$ for $i \leq j$. 

    Since $G$ is a DAG, the first node $0$ in the topological order is a source. If we consider the subgraph of $G$ consisting of $0$ and its first child $i$ according to the topological order, we are in the case of \cref{ExampleIDTwoNodes} and can generically identify $a_{00}$. 

    We now proceed with the induction step. Assume we have generically identified all edges going into the nodes up until $j-1$, i.e., all $a_{kl}$ such that $k \leq j-1$ and $l \leq j-1$. Now consider node $j$. There are two cases, either it is or is not a source. If it is a source proceed similarly to the base case to identify $a_{jj}$. 

    Otherwise, $j$ is a non-source node. In that case, we obtain the following equations by using the recursive formulas for $S$ and $T$, 
    when $i \neq j$. Without loss of generality, we assume that the first source in the topological order on $G$ that is an ancestor of $j$ is node $0$. Then,
    \begin{equation*}
        s_{ij} = \sum_{k \in \text{pa}(i), l \in \text{pa}(j)} a_{ik} a_{jl} s_{kl}, \qquad t_{00j} = \sum_{m \in \text{pa}(j)} a_{00}^2 a_{jm} t_{00m}. 
    \end{equation*}

    Let $\{i_1, \dots, i_d \}$ denote the parents of $j$ not including $j$ itself in topological order. We collect the equations for $s_{i_1j},\ldots, s_{i_d j}$ together with the equation for $t_{00j}$ as follows:  
    
   \begin{align*}
        \begin{bmatrix}
        s_{i_1 j} \\ 
        \vdots \\ 
        s_{i_d j} \\ 
        t_{00j}
        \end{bmatrix} 
        &= 
        \begin{bmatrix}
             A_{ \{ i_1, \dots, i_d \} \times ( \{0\} \cup [j-1] )  } & 0 \\
             0 & a_{00}^2 \\ 
        \end{bmatrix}
        \begin{bmatrix}
            S_{( \{0 \} \cup [j-1] ) \times \{ i_1, \dots, i_d, j \}}  \\
            \begin{matrix}
                t_{00 i_1} &  \cdots &  t_{00 i_d} &  t_{00j} 
            \end{matrix}
        \end{bmatrix} \cdot 
        \begin{bmatrix}
        a_{j i_1 }  \\
        \vdots \\ 
        a_{j i_d} \\
        a_{jj} \\
        \end{bmatrix}  \\
        &=
        \begin{bmatrix}
           A_{ \{ i_1, \dots, i_d \} \times ( \{0\} \cup [j-1] )  } \cdot S_{( \{0 \} \cup [j-1] ) \times \{ i_1, \dots, i_d \}} &  A_{ \{ i_1, \dots, i_d \} \times ( \{0\} \cup [j-1] )  } S_{ (\{0 \} \cup [j-1]) \times j} \\
          \begin{matrix}
              a_{00}^2 t_{00 i_1} & & & & \cdots & & & & a_{00}^2 t_{00 i_d}
          \end{matrix} & a_{00}^2 t_{00j} \\
        \end{bmatrix} \cdot 
        \begin{bmatrix}
        a_{ j i_1}  \\
        \vdots \\ 
        a_{j i_d} \\
        a_{jj} \\
        \end{bmatrix}.
    \end{align*}

We denote the matrix above by $P^j$ with rows indexed by $\{i_1,\dots,i_d,0\}$ and columns indexed by $\{i_1,\dots,i_d,j\}$ in the indicated order. Our goal is to show that $P^j$ is generically invertible. This is equivalent to the upper left block $P^j_{\{i_1,\dots,i_d\}\times\{i_1,\dots,i_d\}}$
as well as its Schur complement in $P^j$ being generically invertible. 

To show this, we pick a lower triangular matrix $A$ with non-zero diagonal entries such that there is only one path from $0$ to $j$ with non-zero edge weights. All other edge weights that are not self-loops are set to zero. Without loss of generality, assume that this path to $j$ passes through the first parent $i_1$. We denote the unique nodes on the path by $(0,k_1,\dots,k_r,i_1,j)$. 
As a consequence, we have that $t_{00i_1} \neq 0$ and $t_{00j} \neq 0$, while $t_{00i_2},\dots,t_{00i_d} = 0$. Furthermore, we have in particular that $s_{0j},s_{k_1j},\dots,s_{k_rj},s_{i_1j} \neq 0$, while $s_{lj} = 0$ for all nodes $l$ that are not on that path.

Computing the upper left block of $P^j$, we obtain
\begin{align*}
    P^j_{\{i_1,\dots,i_d\}\times\{i_1,\dots,i_d\}}
    = & \begin{bmatrix}
        \cdots & 0 & a_{i_1k_r} & a_{i_1i_1} & 0 & \cdots & \cdots &  & \\
        \cdots & 0 & 0 & 0 & a_{i_2i_2} & 0 & \cdots & &  \\
        \ddots & \vdots & \vdots & \vdots & \ddots & \ddots & \ddots & & \\
        \cdots & 0 & 0 & 0 & \cdots & 0 & a_{i_di_d} & 0 & \cdots \\
    \end{bmatrix} 
    \begin{bmatrix}
        s_{0 i_1} & 0 & \cdots & \cdots & 0 \\
        \vdots & \vdots & & & \vdots \\
        s_{k_{r-1}i_1} & \vdots & & & \vdots \\
        s_{k_ri_1} & \vdots & & & \vdots \\
        s_{i_1i_1} & 0 & \cdots & \cdots & 0 \\
        0 & s_{i_2i_2} & 0 & \cdots & 0 \\
        \vdots & \ddots & \ddots & \ddots & \vdots \\
        0 & \cdots & 0 & s_{i_di_d} & 0 
    \end{bmatrix} \\
    = & \begin{bmatrix}
        a_{i_1k_r} s_{k_ri_1} + a_{i_1i_1} s_{i_1i_1} & & & \\
        & a_{i_2i_2} s_{i_2i_2} & & \\
        & & \ddots & \\
        & & & a_{i_di_d} s_{i_di_d}
    \end{bmatrix},
\end{align*}
so it is generically invertible by taking the inverse of the diagonal elements.
Note that we depicted the rows and columns indexed by $k_r$, $i_1, \dots,i_d$ to be adjacent in the matrix, however they do not have to be consecutive nodes in the topological order. In that case, there would be additional zero rows and columns in the respective matrices.

Its Schur complement in $P^j$ is of the form
\begin{align*}
    P^j_{0j\cdot \{i_1,\dots,i_d\}} 
    = a_{00}^2t_{00j} - 
    \begin{bmatrix}
        a_{00}^2t_{00i_1} & \dots & a_{00}^2t_{00i_d} & a_{00}^2t_{00j}
    \end{bmatrix} 
    \cdot
    \left(P^j_{\{i_1,\dots,i_d\}\times\{i_1,\dots,i_d\}}\right)^{-1} 
    \cdot 
    P^j_{\{i_1,\dots,i_d\}\times j}.
\end{align*}
We further have 
\begin{align*}
    P^j_{\{i_1,\dots,i_d\}\times j} 
    & = A_{\{i_1,\dots,i_d\} \times (\{0\} \cup [j-1])} \cdot S_{(\{0\} \cup [j-1]) \times j} \\
    & = A_{\{i_1,\dots,i_d\} \times (\{0\} \cup [j-1])} 
    \cdot 
    \begin{bmatrix}
    * \\ \vdots \\ * \\ s_{k_{r-1}j} \\ s_{k_r j} \\ s_{i_1j} \\ 0 \\ \vdots \\ 0 
    \end{bmatrix} 
    = \begin{bmatrix}
    a_{i_1k_r} s_{k_rj} + a_{i_1i_1} s_{i_1j} \\
    0 \\ \vdots\\ 0
    \end{bmatrix}.
\end{align*}
Combining these expressions yields 
\begin{align*}
    P^j_{0j\cdot \{i_1,\dots,i_d\}} = a_{00}^2t_{00j} - a_{00}^2t_{00i_1} \frac{a_{i_1k_r} s_{k_rj} + a_{i_1i_1} s_{i_1j}}{a_{i_1k_r} s_{k_ri_1} + a_{i_1i_1} s_{i_1i_1}}.
\end{align*}
Thus, we now have to choose the values of the non-zero edge weights in $A$ such that 
\begin{align}\label{eq:schur_complement}
t_{00j}\left(a_{i_1k_r} s_{k_ri_1} + a_{i_1i_1} s_{i_1i_1}\right) - t_{00i_1} \left(a_{i_1k_r} s_{k_rj} + a_{i_1i_1} s_{i_1j}\right) \neq 0.
\end{align}
Note that 
\[
t_{00j} = \frac{a_{00}^2a_{ji_1}}{1-a_{00}^2a_{jj}} t_{00i_1},
\]
so the expression above simplifies to
\[
a_{00}^2a_{i_1j}\left(a_{i_1k_r} s_{k_ri_1} + a_{i_1i_1} s_{i_1i_1}\right) - (1-a_{00}^2a_{jj}) \left(a_{i_1k_r} s_{k_rj} + a_{i_1i_1} s_{i_1j}\right) \neq 0.
\]

Observe that not all terms in the above expression depend on $\omega^{(2)}_{i_1}$, as there are, for example, no treks between $k_r$ and either $i_1$ or $j$ starting at $i_1$. The only terms that depend on $\omega^{(2)}_{i_1}$ are $s_{i_1i_1}$ and $s_{i_1j}$, and we~have 
\begin{align*}
    s_{i_1i_1} & = \dots + \omega^{(2)}_{i_1}\frac{1}{1-a_{i_1i_1}^2}, \\
    s_{i_1j} & = \frac{1}{1-a_{i_1i_1}a_{jj}}\left(a_{i_1k_r} a_{ji_1} s_{k_ri_1} + a_{i_1k_r} a_{jj} s_{k_rj} + a_{i_1i_1} a_{ji_1} s_{i_1i_1}\right)
    = \dots + \omega^{(2)}_{i_1}\frac{a_{i_1i_1} a_{ji_1}}{(1-a_{i_1i_1}^2)(1-a_{i_1i_1}a_{jj})}
\end{align*}
due to the structure of the chosen $A$.
Consequently, the coefficient of $\omega^{(2)}_{i_1}$ in the numerator of the Schur complement \eqref{eq:schur_complement} is 
\begin{align*}
   & a_{00}^2a_{i_1j} a_{i_1i_1} \frac{1}{1-a_{i_1i_1}^2} - (1-a_{00}^2a_{jj}) a_{i_1i_1} \frac{a_{i_1i_1} a_{ji_1}}{(1-a_{i_1i_1}^2)(1-a_{i_1i_1}a_{jj})} \\
   = \ &  a_{i_1j} a_{i_1i_1} \frac{a_{00}^2(1-a_{i_1i_1}a_{jj}) - (1-a_{00}^2a_{jj})a_{i_1i_1}}{(1-a_{i_1i_1}^2)(1-a_{i_1i_1}a_{jj})} \\
   = \ & a_{i_1j} a_{i_1i_1} \frac{a_{00}^2 - a_{i_1i_1}}{(1-a_{i_1i_1}^2)(1-a_{i_1i_1}a_{jj})}.
\end{align*}
Finally, if we set all non-zero entries of $A$ to $\frac12$, for instance, the coefficient evaluates to $-\frac19 \neq 0$.
\end{proof}

\subsection{Proof of \texorpdfstring{\cref{thm::GenericIDGeneral}}{Theorem 4.8}}
\begin{proof}
    Let the node set $V$ of $G$ be topologically ordered as $0\leq \ldots\leq p-1$. The proof proceeds by induction. 

    The base case of the induction is to identify the self-loop of the first source. We can identify the self-loop of the source by \cref{ExampleIDTwoNodes} or \cref{ExampleIDTwoNodesOneLoop}, depending on whether the child has a self-loop. 
    
    We now proceed with the induction step. Consider node $j$. If it is a source node, repeat the base case. Otherwise, let $i_1, \ldots, i_d$ be the parents of $j$. There are two cases, depending on whether $j$ has a self-loop or~not. 
    
    Assume first that $j$ does not have a self-loop. In this case, we need to identify $d$ incoming edges at this step. For each parent $i_k$, let $e_k$ denote one of it's ancestors that is a source. Notice that all $e_1,\ldots, e_d$ are distinct; otherwise, $G$ would not be a polytree. Then, by collecting the recursive formulas for $s_{e_1 j}, \dots s_{e_d j}$, we~have
   \begin{align}\label{eq:DAGSubmatrices}
        \begin{bmatrix}
        s_{e_1 j} \\ 
        \vdots \\ 
        s_{e_d j} 
        \end{bmatrix} 
        &= 
        \begin{bmatrix}
             A_{ \{ e_1, \dots, e_d \} \times ( \{0\} \cup [j-1] )  }  
        \end{bmatrix}
        \begin{bmatrix}
            S_{( \{0 \} \cup [j-1] ) \times \{ i_1, \dots, i_d\}}  
        \end{bmatrix} \cdot 
        \begin{bmatrix}
        a_{j i_1 }  \\
        \vdots \\ 
        a_{j i_d} \\
        \end{bmatrix}.
    \end{align}

    Notice that, since $e_1,\ldots, e_d$ are sources, they have no incoming edges except for the self-loops. Consequently, each row $e_l$ has only one non-zero entry in column $e_l$. Therefore, the product of the first two matrices above can be written as the product of two square matrices:
    \begin{align*}
        \begin{bmatrix}
             A_{ \{ e_1, \dots, e_d \} \times ( \{0\} \cup [j-1] )  }  
        \end{bmatrix}
        \begin{bmatrix}
            S_{( \{0 \} \cup [j-1] ) \times \{ i_1, \dots, i_d\}}  
        \end{bmatrix} = \begin{bmatrix}
             A_{ \{ e_1, \dots, e_d \} \times \{ e_1, \dots, e_d \} }  
        \end{bmatrix}
        \begin{bmatrix}
            S_{\{ e_1, \dots, e_d \}\times \{ i_1, \dots, i_d\}}  
        \end{bmatrix}. 
    \end{align*}
    
    The product is diagonal, since $s_{e_l i_k}=0$ unless $l=k$. Hence, we can write
     \begin{align*}
        \begin{bmatrix}
        s_{e_1 j} \\ 
        \vdots \\ 
        s_{e_d j} 
        \end{bmatrix} 
        &= 
        \begin{bmatrix}
            a_{e_1 e_1}s_{e_1 i_1} & & &\\
             & a_{e_2 e_2}s_{e_2 i_2} & &\\
            & & \ddots &\\
             & & & a_{e_d e_d}s_{e_d i_d}
        \end{bmatrix} \cdot 
        \begin{bmatrix}
        a_{j i_1 }  \\
        \vdots \\ 
        a_{j i_d} \\
        \end{bmatrix}. 
    \end{align*}
    Now, let $A$ satisfy $a_{kk} \neq 0$ for all $k$ and $a_{ji_1} \neq 0$. Then, by the equitrek rule, we have $s_{e_k i_k} \neq 0$ for all $1 \leq k \leq d$, since there exists an equitrek between $e_k$ and $i_k$ for each $k$. For this choice of parameters, the matrix is invertible and therefore, the parameters are generically identifiable.

Now assume that $j$ has a self-loop. Without loss of generality, assume $0=e_1$. It follows that $t_{00i_1} \neq 0$, while $t_{00i_k}=0$ for all $2\leq k\leq d$. Consider the system
   \begin{align*}
        \begin{bmatrix}
        s_{0 j} \\ 
        \vdots \\ 
        s_{e_d j} \\ 
        t_{00j}
        \end{bmatrix} 
        &= 
        \begin{bmatrix}
             A_{ \{ 0, e_2, \dots, e_d \} \times ( \{0\} \cup [j-1] )  } & 0 \\
             0 & a_{00}^2 \\ 
        \end{bmatrix}
        \begin{bmatrix}
            S_{( \{0 \} \cup [j-1] ) \times \{ i_1, \dots, i_d, j \}}  \\
            \begin{matrix}
                t_{00 i_1} & 0 \cdots &  0 &  t_{00j} 
            \end{matrix}
        \end{bmatrix} \cdot 
        \begin{bmatrix}
        a_{j i_1 }  \\
        \vdots \\ 
        a_{j i_d} \\
        a_{jj} \\
        \end{bmatrix}  \\
        &=
        \begin{bmatrix}
            a_{00 }s_{0 i_1} & 0 & \ldots & 0  & a_{00}s_{0j}\\
            0 & a_{e_2 e_2}s_{e_2 i_2} & \ldots & 0 & a_{e_2e_2}s_{e_2j}\\
            & \vdots\\
            0 & \ldots & 0 & a_{e_d e_d}s_{e_d i_d} & a_{e_de_d}s_{e_dj}\\
            a_{00}^2t_{00i_1} & 0 & \ldots &0 & a_{00}^2t_{00j}
        \end{bmatrix} \cdot 
        \begin{bmatrix}
        a_{j i_1 }  \\
        \vdots \\ 
        a_{j i_d} \\
        a_{jj} \\
        \end{bmatrix}. 
    \end{align*}
    As before, let $A$ satisfy $a_{kk} \neq 0,1$ for all $k$ and $a_{ji_1} \neq 0$. Then, by the equitrek rule, we have $s_{e_k i_k}, t_{e_k e_k i_k} \neq 0$ for all $1 \leq k \leq d$ and $s_{0j}, t_{00j} \neq 0$, since there is an equitrek between $e_k$ and $i_k$ for each $k$.
    
    Hence, the upper-left $d \times d$ block of this matrix is invertible for this choice of $A$. To prove that the entire matrix is invertible, it suffices to check that the Schur complement with respect to this block, which is a $1 \times 1$ matrix, is nonzero under these conditions.  The Schur complement is
    \[
    a_{00}^2 t_{00j} - \frac{a_{00}^2 t_{00i_1} s_{0j}}{s_{0i_1}}.
    \]
    For the sake of contradiction, suppose that the Schur complement is zero. Clearing denominators and dividing through by $a_{00}$, which is nonzero by assumption, yields
    \begin{equation}\label{eq:TreeSchur}
    s_{0i_1} t_{00j} - s_{0j}t_{00i_1} = 0.
    \end{equation}
    Since $G$ is a polytree, the only path from $0$ to $j$ passes through $i_1$. Therefore, we have
    \[
    s_{0j} = \frac{a_{00}a_{ji_1} s_{0i_1}}{1-a_{00}a_{jj}} \qquad \text{and} \qquad
    t_{00j} = \frac{a_{00}^2 a_{ji_1} t_{00i_1}}{1-a_{00}^2a_{jj}}.
    \]
    Plugging these into \eqref{eq:TreeSchur}, we obtain
    \[
    s_{0i_1}t_{00i_1} \left(\frac{a_{00}^2 a_{ji_1}}{1-a_{00}^2a_{jj}} - \frac{a_{00}a_{ji_1} }{1-a_{00}a_{jj}}\right) =0.
    \]
    Clearing denominators and dividing through by $s_{0i_1} t_{00i_1}$, which is assumed to be nonzero, gives
    $$0 = (1-a_{00} a_{jj})a_{00}^2 a_{ji_1} - (1-a_{00}^2a_{jj}) a_{00} a_{ji_1} = a_{00}^2a_{ji_1} - a_{00}a_{ji_1} = a_{ji_1} a_{00} (a_{00}-1).$$
    This contradicts our assumption that $a_{ji_1} \neq 0$ and $a_{00} \neq 0,1$. Hence, the Schur complement is nonzero, and the matrix is invertible, as required. 
\end{proof}

\section{Proofs for \texorpdfstring{\cref{sec:local identifiability}}{Section 5}} \label{app::ProofsForLocalID}

We provide the proof for the entries of the modified Jacobians, described in \cref{prop::J_2(Sigma)descriptionOfEntries} and \cref{prop::J_3(T)descriptionOfEntries}. 

\begin{proof}[Proof of \cref{prop::J_2(Sigma)descriptionOfEntries}]
   By equation \eqref{eq::JacOfSigmawrtA} and the fact that $(I- A \otimes A)^{-1} \text{vec}(\Omega^{(2)}) = \text{vec}(S)$, the entry of the Jacobian indexed by $(ij), (\alpha \rightarrow \beta)$ is given by the following sum: 
    \begin{align*}
        J^G_2(S, a)_{(ij), \alpha \rightarrow \beta} &= \sum_{l = 0}^{p-1} \sum_{k = 0}^{p-1} (A \otimes E_{\beta \alpha} + E_{\beta \alpha } \otimes A)_{ip + j,lp+k}\cdot s_{lk} \\ 
        &= \sum_{l = 0}^{p-1} \sum_{k = 0}^{p-1}
         s_{lk} \left( a_{il}\delta_{(j,k)}(\beta, \alpha) + \delta_{(i,l)}(\beta, \alpha) a_{jk} \right) \\ 
        &= \delta_{j}(\beta) \sum_{l = 0}^{p-1} a_{il} s_{l \alpha} + \delta_{i}(\beta) \sum_{k = 0}^{p-1} a_{jk } s_{k \alpha},
    \end{align*}
    as stated in the statement of the proposition. To get to the trek description, note that $s_{l \alpha} = \sum_{\tau \in \mathcal{T}(l, \alpha)} m_{\tau}$, where the sum is over all equitreks between $l$ and $\alpha$ as given by the trek rule (\cref{prop::trekrule}). Multiplying this by $a_{il}$ corresponds to adding the edge $l \rightarrow i$ to the trek and thereby transforming the equitrek between $l$ and $\alpha$ to a trek between $i$ and $\alpha$ where the path going to $i$ is one longer than the path going to $\alpha$. 
\end{proof}

\begin{proof}[Proof of \cref{prop::J_3(T)descriptionOfEntries}]
    By equation \eqref{eq::JacofTwrtA} and the fact that $(I - A \otimes A \otimes A)^{-1} \text{vec}(\Omega^{(3)}) = \mathrm{vec}(T)$, the entry of the Jacobian indexed by $(ijk),(\alpha \rightarrow \beta)$ is given by the following sum:
    \begin{align*}
        &J^G_3(T,a)_{(ijk), \alpha \rightarrow \beta}  \\ 
 &= \sum_{l = 0}^{p-1} \sum_{m = 0}^{p-1} \sum_{n = 0}^{p-1} ((E_{\beta \alpha} \otimes A ) \otimes A + (A \otimes E_{\beta \alpha}) \otimes A + (A \otimes A )\otimes E_{\beta \alpha})_{(ip+j)p + k,(lp + m)p + n} t_{lmn} \\ 
 &= \sum_{l = 0}^{p-1} \sum_{m = 0}^{p-1} \sum_{n = 0}^{p-1}  t_{lmn} \left( (E_{\beta \alpha} \otimes A)_{ip + j, lp+m} a_{kn} + (A \otimes E_{\beta \alpha})_{ip + j, lp+m} a_{kn} + (A \otimes A)_{ip + j, lp+m} \delta_{(k,n)}(\beta, \alpha) \right) \\ 
 &= \sum_{l = 0}^{p-1} \sum_{m = 0}^{p-1} \sum_{n = 0}^{p-1}  t_{lmn} \left( \delta_{(i,l)}(\beta, \alpha) a_{jm} a_{kn} + a_{il} \delta_{(j,m)}(\beta, \alpha) a_{kn} + a_{il} a_{jm} \delta_{(k,n)}(\beta, \alpha) \right) \\ 
 &= \delta_{i}(\beta) \sum_{m = 0}^{p-1}\sum_{n = 0}^{p-1} a_{jm} a_{kn} t_{\alpha mn} + \delta_{j}(\beta) \sum_{l = 0}^{p-1} \sum_{n = 0}^{p-1} a_{il} a_{kn} t_{l \alpha n} + \delta_{k}(\beta) \sum_{l = 0}^{p-1} \sum_{m = 0 }^{p-1} a_{il} a_{jm} t_{l m \alpha},
    \end{align*}
    which matches the formula stated in the proposition. Arriving at the trek description is entirely analogous to the covariance case (\cref{prop::J_2(Sigma)descriptionOfEntries}).  
\end{proof}

Below we provide the lemmas required for the main proof of local identifiability using the Jacobian matrix. \cref{lem::CharecterizeGen2star} provides a sufficient condition for a graph to be a generalized two star. The base cases of \cref{thm:LocIdAllGraphs}
are covered by \cref{lem::Case345} and \cref{lem:generalized2Star}. 

\begin{lemma}
\label{lem::CharecterizeGen2star}
    Let $G = (V,E)$ be a connected undirected graph on at least six nodes where every pair of paths of vertex length $3$ intersect by at least one vertex. Then $G$ is a generalized two star.
\end{lemma}

\begin{proof}
    Assume that there are at least three different vertex paths of length $3$; otherwise the statement is trivial. 

    Let $n$ be the length of the longest path in $G$ with no repeated vertices. The proof proceeds by checking possible values for $n$, which are $n = 3,4$, or $5$. If $n$ was larger, this path could be split into two non-intersecting paths of vertex length at least three. 

    If $n = 3$, any two paths must share the middle vertex. Indeed, if they shared only the beginning or end vertex, this would result in a path of vertex length $4$, which is a contradiction. Thus, all distinct vertex paths must share the same center vertex. In this case, the graph is a star. 

    If $n = 4$, consider a path of vertex length $4$, $0 - 1 - 2 - 3$ in the graph. Then any other vertex in the graph cannot be connected to $0$ or $3$ in this path, since otherwise there would be a path of vertex length $5$. Thus, any remaining vertices in the graph have to be connected to vertex $1$ or $2$. If there is a remaining vertex $4$ that is connected to $1$, there cannot be another vertex $5$ connected to $2$, since otherwise $0-1-4$ and $3-2-5$ would be two paths of vertex length $3$ which do not intersect. Furthermore, the vertex $4$ could not be connected to both $1$ and $2$ because then there would be a path of vertex length $5$. Thus, any additional vertices in the graph have to be connected to $1$ or $2$ but only one of them, making this one the center of the graph. The center is unique because in order for the graph to have at least six vertices there must be at least two nodes aside from the given vertex path of length four. Without loss of generality, we let $1$ be the center. If $4$ is connected to $1$ then arguing as before any additional vertex can be connected to $1$ or a single vertex can be connected to $4$, but at most one since otherwise two vertex disjoint paths of vertex length three are created.

    If $n = 5$, consider a path of vertex length $5$, $0 - 1 - 2 - 3 - 4$. Arguing as in the $n = 4$ case, it is only possible for any remaining vertices to be connected to the graph by vertex $2$ (so it is possible as above to add a single vertex or a path of two (potentially connected) vertices going out from $2$) since otherwise a path of length $6$ is created or two vertex disjoint paths of length $3$ are created. Thus, any two paths of vertex length three in this graph must intersect at node $2$. 
\end{proof}

The calculations for the following Lemma were carried out in Maple, with code available at \url{https://github.com/cecilie2424/Local-Identifiability-in-Non-Gaussian-Discrete-Lyapunov-Models}. 

\begin{lemma}
    \label{lem::Case345}
     Let $G = (V,E)$ be any directed graph on three, four or five nodes with all self-loops. Then, the matrix $\mathcal{J}^G_{\mathrm{off}}(S,T)$ has full rank generically, so $G$ is locally identifiable. 
\end{lemma}

\begin{proof}

We consider all connected polytrees on $3$, $4$ and $5$ nodes. We show that when $A$ has the sparsity pattern of such a graph, the off-diagonal modified Jacobian for the complete graph has full column rank $(=p^2)$. 
This shows that for any graph 
$G'$ on $p$ nodes which has a given polytree on $p$ nodes as a subgraph, its off diagonal modified Jacobian, $\mathcal{J}^{G'}_{\mathrm{off}}(S,T)$, will have full column rank. Therefore, $G'$ is locally identifiable. 
To see this, remove the columns from the complete graph which are not in $G'$. Since the off-diagonal modified Jacobian of the entire graph had full column rank ($=p^2$), it will just drop by 1 in column rank when removing a column until we are only left with columns corresponding to edges in $G'$, which then has full column rank equal to the number of edges in $E'$. The chosen $A$ is valid for any graph which has the specific polytree as a subgraph. In this way we are guaranteed to check any connected graph if we check every polytree since any connected graph has at least one polytree as a subgraph.  

In order to ensure that all possible polytrees are checked (some possibly more than once), we consider all undirected unlabeled trees on three, four and five nodes. There are one, two and three respectively (consider \cref{fig::UndirectedTrees345} without labeling). Then all unlabeled polytrees on three, four and fives nodes will be included in the polytrees created by choosing all possible combinations of directions of the edges with skeleton equal to the undirected trees. This creates $4,16$ and $48$ different $A$'s to check to make sure that all possible directed trees are checked at least once. We choose a labeling, see \cref{fig::UndirectedTrees345}, to be able to pick specific $A$'s. 
The matrices are created and checked in the maple code. They are all chosen such that they have $a_{ii}=1/2,\ i=1,\dots,p$ on the diagonal and the off-diagonal entries corresponding to existing edges in the given polytree set to one. See \cref{ex::TwoNodeJacobian} to see such an $A$ in the case of a graph on two nodes. 
\end{proof}

    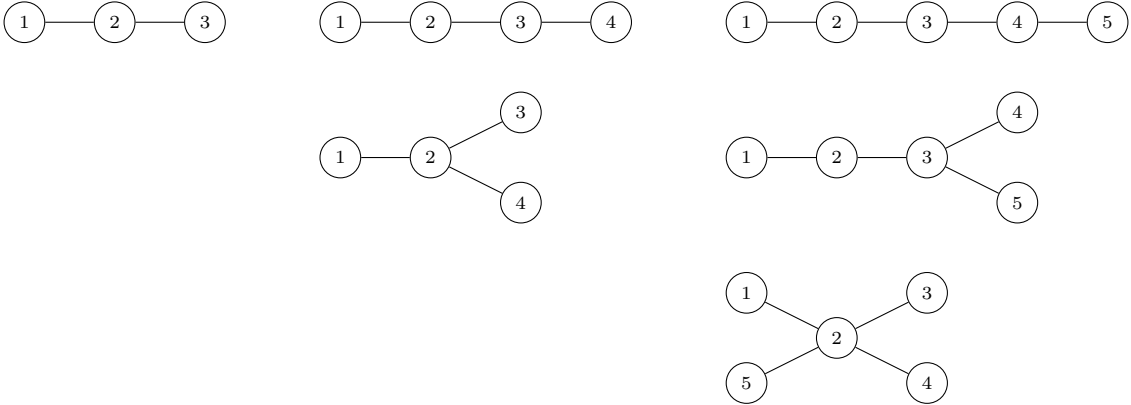
\begin{figure}[H]
    \centering
\begin{center}
    \begin{tikzpicture}[scale = 0.6]
      \tikzstyle{every node}=[font=\scriptsize]
      \node[circle, draw, minimum size=0.4cm] (0) at (0,0) {1};
      \node[circle, draw, minimum size=0.4cm] (1) at (2,0) {2};
      \node[circle, draw, minimum size=0.4cm] (2) at (4,0) {3};
      
      \node[circle, draw, minimum size=0.4cm] (3) at (7,0) {1};
      \node[circle, draw, minimum size=0.4cm] (4) at (9,0) {2};
      \node[circle, draw, minimum size=0.4cm] (5) at (11,0) {3};
      \node[circle, draw, minimum size=0.4cm] (6) at (13,0) {4};

      \node[circle, draw, minimum size=0.4cm] (7) at (7,-3) {1};
      \node[circle, draw, minimum size=0.4cm] (8) at (9,-3) {2};
      \node[circle, draw, minimum size=0.4cm] (9) at (11,-2) {3};
      \node[circle, draw, minimum size=0.4cm] (10) at (11,-4) {4};

      \node[circle, draw, minimum size=0.4cm] (11) at (16,0) {1};
      \node[circle, draw, minimum size=0.4cm] (12) at (18,0) {2};
      \node[circle, draw, minimum size=0.4cm] (13) at (20,0) {3};
      \node[circle, draw, minimum size=0.4cm] (14) at (22,0) {4};
      \node[circle, draw, minimum size=0.4cm] (15) at (24,0) {5};

      \node[circle, draw, minimum size=0.4cm] (16) at (16,-3) {1};
      \node[circle, draw, minimum size=0.4cm] (17) at (18,-3) {2};
      \node[circle, draw, minimum size=0.4cm] (18) at (20,-3) {3};
      \node[circle, draw, minimum size=0.4cm] (19) at (22,-2) {4};
      \node[circle, draw, minimum size=0.4cm] (20) at (22,-4) {5};

      \node[circle, draw, minimum size=0.4cm] (21) at (16,-6) {1};
      \node[circle, draw, minimum size=0.4cm] (22) at (18,-7) {2};
      \node[circle, draw, minimum size=0.4cm] (23) at (20,-6) {3};
      \node[circle, draw, minimum size=0.4cm] (24) at (20,-8) {4};
      \node[circle, draw, minimum size=0.4cm] (25) at (16,-8) {5};

      \draw[-,  >=stealth] (0) edge  (1);
      \draw[-,  >=stealth] (1) edge  (2);

      \draw[-,  >=stealth] (3) edge  (4);
      \draw[-,  >=stealth] (4) edge  (5);
      \draw[-,  >=stealth] (5) edge  (6);

      \draw[-,  >=stealth] (7) edge  (8);
      \draw[-,  >=stealth] (8) edge  (9);
      \draw[-,  >=stealth] (8) edge  (10);

      \draw[-,  >=stealth] (11) edge  (12);
      \draw[-,  >=stealth] (12) edge  (13);
      \draw[-,  >=stealth] (13) edge  (14);
      \draw[-,  >=stealth] (14) edge  (15);

      \draw[-,  >=stealth] (16) edge  (17);
      \draw[-,  >=stealth] (17) edge  (18);
      \draw[-,  >=stealth] (18) edge  (19);
      \draw[-,  >=stealth] (18) edge  (20);

      \draw[-,  >=stealth] (21) edge  (22);
      \draw[-,  >=stealth] (22) edge  (23);
      \draw[-,  >=stealth] (22) edge  (24);
      \draw[-,  >=stealth] (22) edge  (25);

    \end{tikzpicture} 
    \end{center}
    \caption{The different structures for undirected trees (with labels) on three, four and five nodes.}
    \label{fig::UndirectedTrees345}
\end{figure}

\begin{lemma}
    \label{lem::ColliderTwostar}
    Let $G = (\{0\}\cup [p-1],E)$ be a directed graph on at least three nodes whose undirected skeleton is a generalized two star. Furthermore, assume all nodes not equal to the center either have an edge from it to the center node or is connected (direction irrelevant) to another node that has an edge from it to the center. Then $\mathcal{J}^{G}_{\mathrm{off}}(S,T)$ generically has full rank, and the graph is locally identifiable. 
\end{lemma}

\begin{proof}
    The proof proceeds by picking a matrix $A$ and showing that the modified Jacobian has full rank for this choice. We choose a zero pattern for $A$ such that all outgoing edges from $0$ are set to $0$. This implies that the only treks in the graph are between $0$ and any other node $i \in [p]$, between a node $i$ and itself, and lastly in a generalized two star there can exist an edge between a pair of non-center nodes $i,j$ as long as their only other edges (other than the self-loops) are edges to the center node. Therefore, there can also exist pairs $i,j \in [p]$, $i,j \neq 0$, such that there is a trek between $i$ and $j$, but no treks between $i$ (resp. $j$) with any other node not equal to $i,j,0$. 
    
    Assume that there are $k$ such pairs of nodes, and assume, without loss of generality, that they are the first $1, \dots, 2k$ nodes such that $i$ is paired with $i + 1$ for $i$ an odd number less than $2k$. We refer to these as the \emph{pair nodes}. The remaining $m = p-1-(2k)$ non-center nodes are not connected to any other nodes except the center; we refer to these as the \emph{pure star nodes}. 

    We make the following partition of the rows and columns.
    For each odd integer $i$ between $1$ and $2k-1$, let $E_i$ denote the set of edges involving at least one node from the $i$th pair, that is,
    \begin{equation*}
        E_i = \left\{ x \rightarrow y \in E \; | \; x,y \in \{0, i, i+1 \} \right\} \setminus \{ 0 \rightarrow 0 \}.
    \end{equation*}
    For every pair except the first, we select a subset of rows 
    \begin{align*}
        R_{i} = &\{(0i), (0(i+1)), (i (i+1)), \\
    &(00i), (00(i+1)), (0ii), (0 i (i+1)), (0 (i+1) (i+1)), (ii (i+1)), (i (i+1) (i+1))  \}
    \end{align*}
    of cardinality $|E_i|$ such that the modified Jacobian on three nodes corresponding to the three node graph on $0,i$ and $i+1$ has full rank, as guaranteed to exist by \cref{lem::Case345}. The corresponding column set is
    \begin{align*}
        C_i = \{ x \rightarrow y \in E_i \},
    \end{align*}
    for $i$ taking odd values from $3$ to $2k-1$. 

    For the first pair, we include the edge $0 \rightarrow 0$ in its columns and therefore need to include one additional row. Thus, it will be a subset of    
    \begin{align*}
        R_{1} = &\{(01), (02), (1 2), (001), (002), (011), (0 1 2), (0 2 2), (11 2), (1 2 2)  \}
    \end{align*}
    of cardinality $|E_i| + 1$ such that the modified Jacobian on three nodes corresponding to the three node graph on $0,1$ and $2$ has full rank, as guaranteed to exist by \cref{lem::Case345}. The corresponding column set is
    \begin{align*}
        C_1 = \{ x \rightarrow y \in E_1 \} \cup \{0 \rightarrow 0 \}. 
    \end{align*}

    For the pure star part, namely the nodes $(p-1)-(2k), \dots, p-1$, we define a row and column set for each node $j$ as follows. We choose any subset of
    \begin{align*}
        R_j = \{ (0j), (00j), (0jj) \},
    \end{align*}
    of cardinality equal to the number of edges involving $j$ (at least two and possibly three). For the column set, we pick all the edges involving $j$. By assumption, this set always includes the self-loop and the incoming edge to $0$, and may additionaly include the outgoing edge from $0$. Thus,
    \begin{align*}
        C_j = \{ j \rightarrow j, j \rightarrow 0 \} \cup (E \cap \{ 0 \rightarrow j \}).
    \end{align*}

    If there is at least one pair, then the specified submatrix of the modified Jacobian will be block triangular:
    \begin{equation*}
        \begin{pNiceArray}{cccccccc}[first-row,first-col, nullify-dots]
            & C_{p-1} & \dots & C_{2k+1} & C_{2k-1} & C_{2k-3} & \dots & C_1  \\ 
            R_{p-1} \;\,\  & J^S_{R_{p-1}, C_{p-1}} & & & & & &  \\ 
            \vdots \quad\ & & \rotatebox[origin=c]{20}{$\ddots$} & & & & &  \\ 
            R_{2k+1} & & &   J^S_{R_{2k+1}, C_{2k+1}}   & & &  &  \\ 
            R_{2k-1} & & & &  J^P_{R_{2k-1}, C_{2k-1}} & & &  \\
            R_{2k-3} & & & & & J^P_{R_{2k-3}, C_{2k-3}} & &  \\ 
            \vdots \quad\ & &  \text{\Large 0} & & & & \rotatebox[origin=c]{20}{$\ddots$} & 
            \\ 
            R_1 \;\,\ & & & & & & & J^P_{R_1, C_1}  & \\
        \end{pNiceArray}
    \end{equation*}

    The zero-pattern follows from  \cref{prop::J_2(Sigma)descriptionOfEntries} and \cref{prop::J_3(T)descriptionOfEntries} as in the proof of \cref{thm:LocIdAllGraphs}. Moreover, above the diagonal, nonzero entries can occur only in column $C_1$. 
    By \cref{lem::Case345}, all the pair matrices $J^P_{R_i, C_i}$ for $i$ all the odd numbers from $1$ to $2k-1$ have full rank. For all the pure star matrices  $J^S_{R_i, C_i}$, it follows by direct computation in \cref{ex::TwoNodeJacobian} that they have full rank, since each corresponds to a specific submatrix (with at most three columns) of the joint Jacobian for a connected graph on two nodes with both self-loops. This concludes the proof in the case where there is at least one pair. 

    If there are no pairs, so that the graph has a star skeleton, we need to place the $0 \rightarrow 0$ column in one of the column sets corresponding to a single node, and correspondingly add a row. We choose the nodes set, which now corresponds to node $1$, since there were no pairs. We select a subset of 
    \begin{align*}
        R_1 = \{ (01), (001), (011), (122) \}
    \end{align*}
    of cardinality equal to the number of edges involving node $1$ (at least two and possibly three) plus one, where at least one of the rows is $(122)$. The corresponding column set is
    \begin{align*}
        C_1 = \{ 1 \rightarrow 1, 1 \rightarrow 0 , 0 \rightarrow 0\} \cup (E \cap \{ 0 \rightarrow j \}).
    \end{align*}

    In this case, the specified submatrix of the modified Jacobian will be block triangular and takes the form
    \begin{equation*}
        \begin{pNiceArray}{ccccc}[first-row,first-col, nullify-dots]
            & C_{p-1} & \dots & C_{2} & C_{1} \\ 
            R_{p-1} \;\,\  & J^S_{R_{p-1}, C_{p-1}} & & &   \\ 
            \vdots \quad\ & & \rotatebox[origin=c]{20}{$\ddots$} & &   \\ 
            R_{2} \;\,\  & \text{\Large 0}  &  &   J^S_{R_{2}, C_{2}}   &   \\ 
            R_{1} \;\,\  & & & &  J_{{R_1}, C_{1}}   \\
        \end{pNiceArray}
    \end{equation*}

    Again, the zero-pattern follows from \cref{prop::J_2(Sigma)descriptionOfEntries} and \cref{prop::J_3(T)descriptionOfEntries}. All the matrices on the block diagonal except for $J_{R_1, C_1}$ were already shown to have full rank in the previous case. 
    For $J_{R_1, C_1}$, notice that it is block triangular of the form 
    \begin{equation*}
        \begin{pNiceArray}{ccccc}[first-row,first-col, nullify-dots]
            & 0 \rightarrow 0 & 1 \rightarrow 0 & 1 \rightarrow 1 & 0 \rightarrow 1 \\ 
            (01) \;\,\  &  & & &   \\ 
            (001) \;\,\ & & \mathcal{J}_{\mathrm{off}}^{G'}(S,T)  & &   \\ 
            (011) \;\,\  &   &  &      &   \\ 
            (122) \;\,\  &  & 0 &  &  J^G_3(T,a)_{(122),(0 \rightarrow 1)}   \\
        \end{pNiceArray},
    \end{equation*}
    where $G'$ is the graph on nodes $0$ and $1$ with both self-loops and the edge $1 \rightarrow 0$. This matrix has generically full rank because $\mathcal{J}_{\mathrm{off}}^{G'}(S, T)$ has full rank, as $G'$ is locally identifiable by Example \ref{ex::TwoNodeJacobian}, and $J^G_3(T,a)_{(122),(0 \rightarrow 1)}$ is generically non-zero in the graph $G$ due to existence of equitreks between $0$ and $2$. 
\end{proof}

\begin{lemma} \label{lem:generalized2Star}
    Let $G = (V,E)$ be a directed graph on $p\geq 3$ nodes whose undirected skeleton is a generalized two star. Then $\mathcal{J}^{G}_{\mathrm{off}}(S,T)$ generically has full rank, and the graph is locally identifiable.  
\end{lemma}

\begin{proof}
    The proof proceeds by induction. The base case contains generalized two stars (including the stars as a special case), which are covered by \cref{lem::ColliderTwostar}, as well as generalized two stars on at most five nodes, which are handled by \cref{lem::Case345}.

    Now, let $p \geq 6$ and let $G = (\{0\}\cup [p-1],E)$ be a generalized two star on $p$ nodes, with $0$ denoting the center. If $G$ is a star or a generalized two star such that all non-center nodes either have an edge from it to the center or is connected (direction irrelevant) to another node that has an edge from it to the center, then it is covered by \cref{lem::ColliderTwostar}. Therefore, assume that $G$ is a graph not covered by the base case. This implies that one of the five two- or three-node subgraphs in \cref{fig::CasesForGen2star} is a subgraph of $G$; denote this subgraph by $H$. Moreover, if we remove the nodes not equal to the center, the remaining graph would have at least three nodes. By the induction hypothesis, if the subgraph $H$ is removed (except $0$ and its self-loop), this remaining graph $G'$ is locally identifiable, since it has $p-1$ or $p-2$ nodes. 
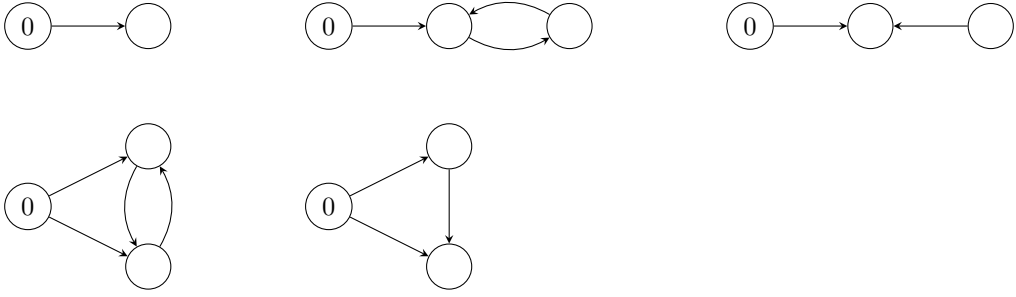
\begin{figure}[H]
    \centering
\begin{center}
    \begin{tikzpicture}[scale = 0.8]
      \node[circle, draw, minimum size=0.6cm] (0) at (0,0) {0};
      \node[circle, draw, minimum size=0.6cm] (1) at (2,0) {};
      
      \node[circle, draw, minimum size=0.6cm] (2) at (5,0) {0};
      \node[circle, draw, minimum size=0.6cm] (3) at (7,0) {};
      \node[circle, draw, minimum size=0.6cm] (4) at (9,0) {};

      \node[circle, draw, minimum size=0.6cm] (5) at (12,0) {0};
      \node[circle, draw, minimum size=0.6cm] (6) at (14,0) {};
      \node[circle, draw, minimum size=0.6cm] (7) at (16,0) {};

      \node[circle, draw, minimum size=0.6cm] (8) at (0,-3) {0};
      \node[circle, draw, minimum size=0.6cm] (9) at (2,-2) {};
      \node[circle, draw, minimum size=0.6cm] (10) at (2,-4) {};

      \node[circle, draw, minimum size=0.6cm] (11) at (5,-3) {0};
      \node[circle, draw, minimum size=0.6cm] (12) at (7,-2) {};
      \node[circle, draw, minimum size=0.6cm] (13) at (7,-4) {};

      \draw[->,  >=stealth] (0) edge  (1);
      
      \draw[->,  >=stealth] (3) edge[bend right] (4);
      \draw[->,  >=stealth] (4) edge[bend right] (3);
      \draw[->,  >=stealth] (2) edge (3);

      \draw[->,  >=stealth] (5) edge (6);
      \draw[->,  >=stealth] (7) edge (6);

      \draw[->,  >=stealth] (8) edge (9);
      \draw[->,  >=stealth] (8) edge (10);
      \draw[->,  >=stealth] (9) edge[bend right] (10);
      \draw[->,  >=stealth] (10) edge[bend right] (9);

      \draw[->,  >=stealth] (11) edge (12);
      \draw[->,  >=stealth] (11) edge (13);
      \draw[->,  >=stealth] (12) edge (13);

    \end{tikzpicture} 
    \end{center}
    \caption{Possible types of extensions from the center in the induction step, shown without self-loops.
    }
    \label{fig::CasesForGen2star}
\end{figure}

    Since all edges involving $0$ in the graphs from Figure \ref{fig::CasesForGen2star} are outgoing, adding any of these five subgraphs to $G'$ (to obtain $G$) will not change any of the treks among the nodes in $G'$, and therefore does not affect the corresponding cumulants, Jacobian or modified Jacobian. Consequently, by the induction hypothesis, all edges in the subgraph $G'$ are indeed locally identifiable in $G$.

    The final step is to conclude that the remaining edges are locally identifiable. There are at least two and at most six such edges, depending on the case from \cref{fig::CasesForGen2star} (including self-loops). Consider the Jacobian of the second- and third-order cumulants involving only the center node $0$ and the potential one or two extra nodes, denoted $\alpha$ and $\beta$. As before, it suffices to consider the non-diagonal cumulant rows of the modified Jacobian. The relevant cumulants are indexed by
    \begin{align*}
        R = \{(0 \alpha), (0 \beta), (\alpha \beta), (00\alpha), (00 \beta), (0 \alpha \alpha), (0 \alpha \beta), (0 \beta \beta), (\alpha \alpha \beta) (\alpha \beta \beta) \},
    \end{align*}
    or, in the case where only one node is added, only the ones involving $\alpha$ and $0$. 
    In all cases, there are more cumulants than edges added. Furthermore, these cumulants can be viewed as equations only in the new edges (those between $0$, $\alpha$ and $\beta$), with the remaining variables being either other cumulants or known $a$-parameters which are already identified. Here, we highlight that all the cumulants not involving combinations of $0$, $\alpha$ and $\beta$ were not affected by adding $\alpha$ and $\beta$. Thus, to estabilsh local identifiability of the entire graph, it suffices to prove local identifiability of the edges between  $0$, $\alpha$ and $\beta$, assuming the edges in $G'$ are already known. 
    
    Therefore, we consider the submatrix of the modified Jacobian indexed by rows in $R$ and columns corresponding to edges between $0$, $\alpha$ and $\beta$ (excluding $0 \rightarrow 0$), denoted $\mathcal{J}^{G \setminus G'}_{\mathrm{off}}$. We wish to prove that this submatrix generically has full rank. Let us choose $A$ generically such that all edges not involving only $0$, $\alpha$ and $\beta$ are set to zero. Then the only treks will be among $0$, $\alpha$ and $\beta$, and $\mathcal{J}^{G \setminus G'}_{\mathrm{off}}$ will be exactly equal to the submatrix of the modified Jacobian of $H$ (on two or three nodes), $\mathcal{J}^{H}_{\mathrm{off}}$, with the column $0 \rightarrow 0$ removed. Since $H$ is either a DAG or a graph on three nodes, $\mathcal{J}^{H}_{\mathrm{off}}$ has full column rank by \cref{lem::Case345} and \cref{prop:IdAllSelfLoops}. It will also have full column rank if the $0 \rightarrow 0$ column is removed. Therefore, $\mathcal{J}^{G \setminus G'}_{\mathrm{off}}$ generically has full column rank and the edges among $0$, $\alpha$ and $\beta$ (except $0 \rightarrow 0$) can be locally identified, assuming all edges in $G'$ had already been locally identified, as guaranteed by the induction hypothesis. This completes the induction step. 
\end{proof}

\begin{example}
\label{ex::TwoNodeJacobian}
     Let $G = (\{0,1\}, E)$ be either the complete graph on two nodes or the complete graph without the edge $1 \rightarrow 0$. Then any $3 \times 3$ submatrix of the $3 \times |E|$ matrix $\mathcal{J}^G_{\mathrm{off}}(S, T)$ generically has rank $3$. We see this by letting 
    \begin{equation*}
        A = \begin{pmatrix}
            1/2 & 0 \\
            1 & 1/2 \\
        \end{pmatrix},
    \end{equation*}
    and computing the rank for all $3\times3$ submatrices of $\mathcal{J}^G_{\mathrm{off}}(S, T)$ in Maple\footnote{\url{ https://github.com/cecilie2424/Local-Identifiability-in-Non-Gaussian-Discrete-Lyapunov-Models}.}.
\end{example}

\begin{lemma}
\label{lem::FullrankMmatrix}
    The matrix $M_{G_i \rightarrow g^j_{k}}$ with $i \neq j$, $i,j \in \{1,2\}$, which appears in the proof of \cref{thm:LocIdAllGraphs} as a submatrix of the modified Jacobian, generically has full rank.
\end{lemma}

\begin{proof}
    Let $A$ be diagonal. Then $M_{G_i \rightarrow g^j_{k}}$ is also diagonal with generically non-zero diagonal entries: 
    By \cref{prop::J_3(T)descriptionOfEntries}, $M_{G_i \rightarrow g^j_{k}}$ is an $n_i \times n_i$ matrix with rows indexed by $R_{G_i \rightarrow g^j_{k}}$ and columns indexed by $C_{G_i \rightarrow g^j_{k}}$, where the $(l,m)$-th entry is given by 
    \begin{align*}
         \left( M_{G_i \rightarrow g^j_{k}}\right)_{l,m} = \left( J_3^G(T,a)) \right)_{(g^i_{l} g^i_l g^j_k), (g^i_m \rightarrow g^j_{k})} = \sum_{ x \in  \text{pa}(g^i_l)} \sum_{y \in \text{pa}(g^i_l)} a_{g^i_l x} a_{g^i_l y} t_{g^i_m x y}. 
    \end{align*}
    Thus, the only way for the above to be non-zero for diagonal $A$ is exactly when $g^i_l = x = y = g^i_m$, since $A$ being diagonal implies that $T$ is diagonal by the trek rule. Consequently, $M_{G_i \rightarrow g^j_{k}}$ is indeed diagonal, with diagonal entries $a_{g^i_l g^i_l}^2 t_{g^i_l g^i_l g^i_l}$ for $l = 1, \dots, n_i$. For $A$ diagonal with non-zero diagonal entries (and $\Omega^{(3)}$ non-zero diagonal), these diagonal entries are non-zero by the trek rule. Therefore, $M_{G_i \rightarrow g^j_{k}}$ generically has full rank. 
\end{proof}

\section{Proofs for \texorpdfstring{\cref{sec::Equations}}{Section 6}} \label{appendix::Defining equations}

\subsection{Proofs for \texorpdfstring{\cref{sec::equations_polytree_single_source}}{Section 6.1}}

\begin{proof}[Proof of \cref{prop:: levels recovery}] 
For a vertex $i\in V$, let  $|\tau_i|$ denote the length of the shortest path from $0$ to $i$, and let $a^{\tau_i}$ denote the corresponding path monomial. In the following, we assume that $v_i^{(2)}$, $v_i^{(3)}$, and $a_{ij}$ do not vanish for all $i, j\in V$, since the equations hold generically.
    \begin{enumerate}
        \item By the monomial parametrization, 
        $$s_{ij} = v^{(2)}_{top(\tau(i, j))}a^{\tau_1}a^{\tau_2}, \qquad s_{ii} =  v^{(2)}_{i}.$$
        Since each $a_{ij}$ is nonzero, the polynomials $s_{ij}^3t_{iii}^2 - s_{ii}^3t_{iij}t_{ijj}$ vanish for every $j$ if and only if for every $j$,
        $v^{(2)}_{top(\tau(i, j))} = v^{(2)}_{i},$
        which is equivalent to $i$ being the source. 
        \item By the monomial parametrization, 
        $$s_{0i} = v^{(2)}_0 a_{00}^{|\tau_i|}a^{\tau_i}, \qquad t_{00i} = v^{(3)}_{0}a_{00}^{2|\tau_i|}a^{\tau_i}.$$
        Then the polynomial becomes 
        $$v^{(2)}_0 a_{00}^{|\tau_i|}a^{\tau_i}\cdot v^{(3)}_{0}a_{00}^{2|\tau_j|}a^{\tau_j} - v^{(2)}_0 a_{00}^{|\tau_j|}a^{\tau_j}\cdot v^{(3)}_{0}a_{00}^{2|\tau_i|}a^{\tau_i} =0. $$

        Since each $a_{ij}$, $v_{i}^{(2)}$ and $v_i^{(3)}$ is nonzero, this holds if  and only if 
        $a_{00}^{|\tau_i|} - a_{00}^{|\tau_j|} = 0$, which implies that $|\tau_i| = |\tau_j|.$   
        \item Since $i$ and $j$ are on different levels, all equitreks between them must start at the source. Therefore,
        $$s_{ij} = v^{(2)}_0a^{\tau_i}a^{\tau_j} a_{00}^{||\tau_i|-|\tau_j||}.$$
        The polynomial becomes
        $$v^{(2)}a^{\tau_i}a^{\tau_j} a_{00}^{||\tau_i|-|\tau_j||} \cdot v^{(3)}_{0}a_{00}^{2|\tau_j|}a^{\tau_j} - v^{(2)}_0 a_{00}^{|\tau_j|}a^{\tau_j} \cdot v^{(3)}_{0}a^{\tau_i}a^{\tau_j} a_{00}^{\max(|\tau_i|, |\tau_j|) + ||\tau_i|-|\tau_j||} =0$$
        The nonvanishing of the parameters again implies that 
        $ a_{00}^{|\tau_j|} - a_{00}^{\max(|\tau_i|, |\tau_j|)} =0$. So we conclude that $ |\tau_j| = \max(|\tau_i|, |\tau_j|).$ \qedhere
    \end{enumerate}
\end{proof}

\begin{proof}[Proof of \cref{prop: top trek recovery}]
    $\Leftarrow$: Suppose $l$ is the top of the shortest equitrek between $i$ and $j$, and let $\tau_i$ and $\tau_j$ denote the paths from $l$ to $i$ and $j$, respectively. By the monomial parametrization, we have 
    $$s_{ij} = s_{ll}a^{\tau_i}a^{\tau_j}, \qquad s_{0i} = s_{0l}a^{|\tau_i|}_{00}a^{\tau_i}, \qquad t_{ljj}=t_{lj}\frac{a^{\tau_j}}{a^{|\tau_i|}_{00}}.$$
    As a result,
    $$s_{0l}s_{ij}t_{llj}-s_{0i}s_{ll}t_{ljj} = s_{0l}s_{ll}a^{\tau_i}a^{\tau_j}t_{llj} - s_{0l}a^{|\tau_i|}_{00}a^{\tau_i}s_{ll}t_{llj}\frac{a^{\tau_j}}{a^{|\tau_i|}_{00}} = 0.$$
    $\Rightarrow$: When writing the parametrization of variables, note that $s_{ll}=v_{l}^{(2)}$. For the polynomial to vanish, this parameter $v_{l}^{(2)}$ must appear in the first term with degree one. The variables in the first terms are $s_{0l}$, $t_{llj}$ and $s_{ij}$, so the parameter $v_l^{(2)}$ might only appear from the parametrization of $s_{ij}$. That implies that $l$ is the top of the shortest equitrek between $i$ and $j$.
\end{proof}

\begin{proof}[Proof of \cref{lem: swapping}] 
    Let $k$ denote the top vertex of the shortest equitrek between $i$ and $j$. Let $A$ and $B$ be the weighted adjacency matrices of $G$ and $H$, respectively. Let $p(i)$ and $p(j)$ denote the parents of $i$ and $j$ in $G$, and let $c(i)$ and $c(j)$ denote the children of $i$ and $j$ in $G$, if they exist. The matrices $A$ and $B$ differ in eight entries.  After swapping the edges, the entries $a_{ip(i)}, a_{jp(j)}, a_{c(i)i}$ and $a_{c(j)j}$ become zero, while the new edges $b_{jp(i)}, b_{ip(j)}, b_{c(j)i}, b_{c(i)j}$ appear in $B$. In the special case where $p(i)=p(j)$, only four elements change, since $a_{ip(i)} = b_{ip(j)}$ and $a_{jp(j)} = b_{jp(i)}$; however, this case is already covered by the general one.

    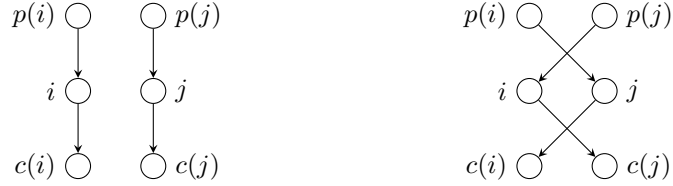
\begin{figure}[H]
    \centering
    \begin{center}
    \begin{tikzpicture} [vertex/.style args = {#1 #2}{circle, draw, minimum size=0.2cm, label=#1:#2}]
      \node[vertex=left $p(i)$] (1) at (-0.5,1) {};
      \node[vertex=right $p(j)$] (2) at (0.5,1) {};
      \node[vertex=left $i$] (3) at (-0.5,0) {};
      \node[vertex=right $j$] (4) at (0.5,0) {};
      \node[vertex=left $c(i)$] (5) at (-0.5,-1) {};
      \node[vertex=right $c(j)$] (6) at (0.5,-1) {};

      \node[vertex=left $p(i)$] (7) at (5.5,1) {};
      \node[vertex=right $p(j)$] (8) at (6.5,1) {};
      \node[vertex=left $i$] (9) at (5.5,0) {};
      \node[vertex=right $j$] (10) at (6.5,0) {};
      \node[vertex=left $c(i)$] (11) at (5.5,-1) {};
      \node[vertex=right $c(j)$] (12) at (6.5,-1) {};

      \draw[->,  >=stealth] (1) edge  (3);
      \draw[->,  >=stealth] (2) edge  (4);
      \draw[->,  >=stealth] (3) edge  (5);
      \draw[->,  >=stealth] (4) edge  (6);

      \draw[->,  >=stealth] (7) edge  (10);
      \draw[->,  >=stealth] (8) edge  (9);
      \draw[->,  >=stealth] (10) edge  (11);
      \draw[->,  >=stealth] (9) edge  (12);

    \end{tikzpicture} 
    \end{center}
    \caption{Swapping operation for vertices $i$ and $j$.}
    \label{fig::SwappingVertices}
    \end{figure}
    
    Let $P_G$ and $P_H$ denote the $(p+|E(G)|)\times \frac{p(p-1)}{2}$ shortest equitrek parametrization matrices of the second-order cumulant models corresponding to $G$ and $H$, respectively. The first $p$ rows are labeled by $v_i^{(2)}$ for $i\in \{0\}\cup [p-1]$, and the remaining rows are labeled by edges. We will show that these matrices are row equivalent. A similar reasoning works for higher-order cumulant models.

    Notice that the first $p$ rows of $P_G$ and $P_H$ coincide, as the top vertices of all shortest equitreks remain unchanged after the swapping operation. Indeed, consider two vertices $u, v\in V(G)$.
    \begin{itemize}
        \item If $u$ and $v$ have different levels, then the top vertex of their shortest equitrek is 0. Since the swap does not alter vertex levels, $u$ and $v$ still lie at different levels in $H$, and their top vertex remains $0$.
        \item If $u$ and $v$ have the same level and are located above $i$ and $j$, then the top of their shortest equitrek also does not change after the swap, since the equitrek between them does not use any of the changed edges.
        \item Suppose $u$ and $v$ lie on the same level and are located below $i$ and $j$. If the shortest equitrek in $G$ passes through both $i$ and $j$, then its top vertex must be $k$, which remains unchanged after the swap. Both paths of the trek can not pass through either $i$ or $j$, since these vertices have at most one child; such a trek would therefore not be the shortest. One of the paths might pass through $i$ or $j$, but then there exists a corresponding trek in $H$, which uses the swapped edges.
        \item The last case is when $u$ and $v$ lie on the same level as $i$ and $j$. If they do not coincide with neither $i$ nor $j$, then the shortest equitrek between them in $G$ does not use any of the edges affected by the swap, and hence they have the same equitrek in $H$. If the pair $u, v$ coincides with $i, j$, then the top of the shortest equitrek remains unchanged after swapping. Finally, suppose without loss of generality that $u=i$ and $v\neq j$. Let $k'$ denote the top of the trek between $i$ and $v$ in $G$. By the statement, $k'$ lies above $k$, and the path from $k'$ to $i$ passes through $k$. A corresponding shortest equitrek between $i$ and $v$ in $H$ can then be constructed as follows: the path from $k'$ to $v$ remains unchanged, while the path from $k'$ to $i$ follows the path from $k'$ to $k$, then to $p(j)$, and finally uses the swapped edge to reach $i$.
        \end{itemize}

    Some rows of $P_H$ coincide with some rows of $P_G$. We now identify all such equal rows. First, observe that $(P_G)_{a_{00}} = (P_H)_{b_{00}}.$
    For a column $s_{uv}$, the corresponding entry in this row corresponds to the number of self-loops used in a trek between $u$ and $v$, which is determined by the levels of the vertices.
    
    A row corresponding to $a_{uv}$ in $G$ is equal to a row corresponding to $b_{uv}$ in $H$ if the edge is not adjacent to either $i$ or $j$. This follows since for each column index $s_{xy}$, the entry $(P_G)_{a_{uv}, s_{xy}}$ equals the number of times the edge $a_{uv}$ is used in the shortest equitrek between $x$ and $y$.

    Finally, by the same reasoning, the following rows also coincide: 
    $$ (P_H)_{b_{c(i)j}}=(P_G)_{a_{c(i)i}}, \qquad (P_H)_{b_{c(j)i}}=(P_G)_{a_{c(j)j}}.$$
    The only remaining rows that change correspond to edges leading to $i$ and $j$ from their parents.
    We claim that  \[(P_H)_{b_{ip(j)}} = (P_G)_{a_{ip(i)}} - (P_G)_{a_{c(i)i}}+(P_G)_{a_{c(j)j}}.\]
    We verify this equality column by column. Fix some column corresponding to $s_{uv}$.
    \begin{itemize}
        \item First, suppose that $(P_H)_{b_{ip(j)}, s_{uv}}=0$. Then none of the paths of the equitrek between $u$ and $v$ pass through $i$ in $H$. This implies that none of the paths of the equitrek pass through $c(j)$ in $H$, and therefore none pass through $c(j)$ in $G$. Hence, $(P_G)_{a_{c(j)j}, s_{uv}}=0$. Moreover, $(P_G)_{a_{ip(i)}, s_{uv}}=(P_G)_{a_{c(i)i}, s_{uv}}$ unless either $u$ or $v$ coincides with $i$. But if $u$ or $v$ were $i$, then the corresponding entry $(P_H)_{b_{ip(j)}, s_{uv}}$ would be nonzero, which contradicts the assumption. Therefore, the right-hand side also evaluates to zero.
        
        \item Next, suppose that $(P_H)_{b_{ip(j)}, s_{uv}}=1$. This means that one of the paths of the trek between $u$ and $v$ crosses $i$ in $H$. This happens if, for example, the path of the trek stops there, so one of $u$ and $v$ coincides with $i$. In this case,  $(P_G)_{a_{c(j)j}, s_{uv}}=0$, since none of the treks go through $c(j)$. Moreover,  $(P_G)_{a_{ip(i)}, s_{uv}} = (P_G)_{a_{c(i)i}, s_{uv}}+ 1$, so the claimed equality holds in this case.

        In the other case, neither $u$ nor $v$ coincide with $i$. Then one of the paths must pass through $c(j)$ in $H$. This implies that in $G$, one of the paths also passes though $c(j)$, so $(P_G)_{a_{c(j)j}, s_{uv}}=1$. On the other hand, the entries of the rows $(P_G)_{a_{ip(i)}}$ and $(P_G)_{a_{c(i)}i}$ are equal.  

        \item Finally, suppose that $(P_H)_{b_{ip(j)}, s_{uv}}=2$. This means that both paths of the equitrek pass through $i$ in~$H$. By the same reasoning as above, equality holds.
    \end{itemize}
    We have verified that the claimed equality holds for all columns of $(P_H)_{b_{ip(j)}}$. By an analogous argument, the remaining changed row satisfies   
    \[(P_H)_{b_{jp(i)}} = (P_G)_{a_{jp(j)}} - (P_G)_{a_{c(j)}j}+(P_G)_{a_{c(i)i}}.\]
    This shows that $P_G$ and $P_H$ are row equivalent, and the explicit row equivalence has been identified. 

    Now suppose that $P_G$ and $P_H$ are the parametrization matrices for cumulants up to order $k$. These matrices have $kp+|E(G)|$ rows. Again, the first $kp$ rows are equal, since the top vertices of the equitreks remain unchanged after the edge swap. 
    For the rows corresponding to edges, we need to show that the same relations hold as in the second-order case described above. This can be checked column by column in a similar way. 
    Indeed, in the new columns corresponding to cumulants of order three and higher, the only difference is that we consider equitreks between $i$-tuples of vertices for $3\leq i\leq k$, and the argument above still holds.
\end{proof}

\begin{proof}[Proof of \cref{lem: row equiv}]
    $\Leftarrow$: The swapping operation does not change the equitreks and the levels, as shown in the previous lemma.\\
    $\Rightarrow$: We proceed by induction on the number of levels. The base case covers graphs with two levels. Such graphs define equivalent models if they have the same source and the same number of vertices.

    Assume the statement holds for graphs with $k-1$ levels. Let $G$ and $H$ be graphs with $p$ vertices and $k$ levels, satisfying the conditions from the statement. Let $G'$ and $H'$ be the graphs obtained from $G$ and $H$, respectively, by removing all vertices in the last level. Then $G'$ and $H'$ also satisfy the conditions, and by the induction hypothesis, $H'$ can be obtained from $G'$ by a sequence of swapping operations.
    
    We now explain why the same swapping operations can be applied to $H$. 
    Observe that the vertices at level $k-1$ of $G$ and $H$ that have more than one child coincide, since every vertex with at least two children is the top of an equitrek. Moreover, if $u$ and $v$ are vertices at the last level of $G$ whose equitrek has top $w$, then the equitrek between $pa(u)$ and $pa(v)$ has also top $w$.
    Hence, if swapping vertices in $H'$ preserves the tops of equitreks, then the same swapping operations preserve equitrek tops in $H$ as well.

    Applying the same swapping operations to $H$, the resulting graph agrees with $G$ on all levels except possibly the last one, using the induction hypothesis. Finally, additional swaps within the last level can be performed to make the two graph identical. This completes the inductive step. 
\end{proof}

\subsection{Proofs for \texorpdfstring{\cref{sec::VanishingDeterminants}}{Section 6.2}}

\begin{proof}[Proof of \cref{prop::Determinant_Parents}]
    The matrix $S'$ has size $(p-k)\times k$, hence $rk(S')\leq \min(k, p-k)$. Assume that $|pa(U)|<\min(k, p-k)$. We prove the desired rank inequality by decomposing $S'$ as a product of two matrices of appropriate sizes. 
    Let $s_{iv_j}$ be an entry of $S'$. Then $i\in V\setminus U$. For any $u\in pa(U)$ such that $u\notin pa(v_j)$, we have $a_{v_ju} = 0$. Since $s_{iv_j}$ is not a diagonal entry, we may write
    $$s_{iv_j} = \sum_{u\in pa(U)}\sum_{w\in pa(i)} a_{v_ju}a_{iw}s_{wu} = \sum_{u\in pa(U)}a_{v_ju}\sum_{w\in pa(i)}a_{iw}s_{wu} = \sum_{u\in pa(U)}a_{v_ju} K_{ui},$$
    where $K_{ui} \coloneqq \sum_{w\in pa(i)}a_{iw}s_{wu}$. The quantities $K_{ui}$ form a matrix $K$ of size $|pa(U)|\times (p-k)$. Let $A'$ be the $k\times |pa(U)|$ submatrix of $A$ whose rows correspond to $v_1, \ldots, v_k$ and whose columns correspond to vertices in $pa(U)$. Then
    $$S' = (A'K)^T,$$
    and therefore $rk(S')\leq |pa(U)|.$

    We now argue analogously for the Tensor $T$. Fix $i$ and consider the slice $T_i$. 
    Let $T_i'$ be the submatrix of $T_i$ formed by the columns $v_1, \ldots, v_k$ and excluding diagonal entries $t_{jjj}$. If $i\in U$, then $T_i'$ has size $(p-1)\times k$; otherwise, it has size $p\times k$. Assume that $|an_2(U)|$ is smaller than both dimensions of $T_i'$. 
    
    Let $t_{ijv_l}$ be an entry of $T_i'$. Then
    $$t_{ijv_l} = \sum_{u\in pa(U)} \sum_{w\in pa(i)} \sum_{x\in pa(j)} a_{v_ju}a_{iw}a_{jx}t_{uwx} =  \sum_{u\in pa(U)}a_{v_ju} K_{ui}, $$
    where $K_{ui} \coloneqq \sum_{w\in pa(i)} \sum_{x\in pa(j)} a_{iw}a_{jx}t_{uwx}$. Again, the quantities $K_{ui}$ form a matrix $K$ with $|pa(U)|$ rows. Let $A'$ be the $k\times |pa(U)|$ submatrix of $A$, with rows indexed by $v_1, \ldots, v_j$ and columns indexed by $pa(U)$. Then
    $$T_i' = (A'K)^T,$$
    and consequently $rk(T_i')\leq |pa(U)|.$
\end{proof}

\begin{proof}[Proof of \cref{prop::Determinant_Parents_full}]
    By \cref{prop::Determinant_Parents}, we can decompose $S'$ and each slice into a product of two matrices, $A'$ and $K$. Note that the first factor $A'$ is the same in all such decompositions. Using these factorizations, we construct a decomposition of $Q$ by taking $A'$ and stacking the corresponding matrices $K$ in the appropriate order.
\end{proof}

\begin{proof}[Proof of \cref{prop::Determinant_Ancestors}]
    The matrix $S'$ has size $(p-|sib(U)|)\times k$, hence $rk(S')\leq \min(k, p-|sib(U)|)$. Assume that $|an_2(U)|<\min(k, p-|sib(U)|)$. We establish the desired rank inequality by decomposing $S'$ as a product of two matrices of appropriate sizes. Let $s_{iv_j}$ be an entry of $S'$. Then $i\in V\setminus sib(U)$. We compute
    \begin{align*}
    s_{iv_j} &= \sum_{u\in pa(U)}\sum_{w\in pa(i)} a_{v_ju}a_{iw}s_{wu} \\
    &=\sum_{u\in pa(U)}\sum_{w\in pa(i)} a_{v_ju}a_{iw} \left(\sum_{g\in pa(u)}\sum_{x\in pa(w)} a_{ug}a_{wx} s_{gx}\right) \\
    &=\sum_{u\in pa(U)}\sum_{w\in pa(i)} a_{v_ju}a_{iw} \left(\sum_{g\in an_2(U)}\sum_{x\in pa(w)} a_{ug}a_{wx} s_{gx}\right)\\
    &=\sum_{u\in pa(U)}\sum_{w\in pa(i)}\sum_{g\in an_2(U)}\sum_{x\in pa(w)}a_{v_ju}a_{iw}a_{ug}a_{wx} s_{gx}\\
    &=\sum_{g\in an_2(U)} \left(\sum_{u\in pa(U)}a_{v_ju}a_{ug}\right)\left(\sum_{w\in pa(i)}\sum_{x\in pa(w)}a_{iw}a_{wx} s_{gx}\right)\\
    &=\sum_{g\in an_2(U)} B_{v_jg} K_{gi}.
    \end{align*}
    Here we define $$B_{v_jg} \coloneqq \sum_{u\in pa(U)}a_{v_ju}a_{ug}, \qquad K_{gi} \coloneqq \sum_{w\in pa(i)}\sum_{x\in pa(w)}a_{iw}a_{wx} s_{gx}.$$ These quantities form matrices $B$ of size $k\times |an_2(U)|$ and $K$ of size $|an_2(U)\times p-|sib(U)|$. Consequently,
    $$S' = (BK)^T,$$
    and therefore $rk(S')\leq |an_2(U)|.$

    Now consider a slice $T_i$ of the tensor $T$. Let $T_i'$ be the submatrix of $T_i$ formed by the columns $v_1,\ldots, v_k$ and obtained by removing all rows containing entries $t_{ijl}$ with $i, j, l\in sib(U)$. For an entry $t_{ilv_j}$ of $T_i'$, we have
    \begin{align*}
    t_{ilv_j} &= \sum_{u\in pa(U)}\sum_{w\in pa(i)} \sum_{x\in pa(l)}a_{v_ju}a_{iw}a_{lx}t_{wux} \\
    &=\sum_{u\in pa(U)}\sum_{w\in pa(i)}\sum_{x\in pa(l)} a_{v_ju}a_{iw}a_{lx} \left(\sum_{g\in an_2(U)}\sum_{y\in pa(w)}\sum_{z\in\pa(x)} a_{ug}a_{wy} a_{xz}t_{qyz}\right) \\
    & =\sum_{g\in an_2(U)} B_{v_jg} K_{gi},
    \end{align*}
    where $$B_{v_jg} \coloneqq \sum_{u\in pa(U)}a_{v_ju}a_{ug},\qquad K_{gi} \coloneqq \sum_{w\in pa(i)}\sum_{x\in pa(l)}\sum_{y\in pa(w)}\sum_{z\in pa(x)}a_{iw}a_{lx}a_{wy}a_{xz} t_{qyz}.$$
    As before, these quantities form matrices $B$ of size $k\times |an_2(U)|$ and $K$ of compatible dimensions, and we obtain
    $$T_i' = (BK)^T,$$
    which implies $rk(T_i')\leq |an_2(U)|.$

    Finally, observe that both $S'$ and each slices $T_i'$ admit a decomposition as a product of two matrices $B$ and $K$, where the first factor is the same in all cases. Using these decompositions, we construct a decomposition of $Q$ by fixing $B$ and stacking the corresponding matrices $K$ in the appropriate order.
\end{proof}

\end{document}